\documentclass[11pt]{amsart}
\usepackage{amsmath,amsthm,latexsym,amsfonts,amssymb,mathrsfs,array,manfnt}
\usepackage[pdftex]{graphicx}
\usepackage[all]{xy}
\usepackage{paralist}
\usepackage{framed}
\usepackage{cancel}
\usepackage{enumerate}
\usepackage{units}
\usepackage{gensymb}
\usepackage{setspace}
\usepackage{amscd}
\usepackage{microtype}
\usepackage{tikz}
\usepackage{hyperref}
\usetikzlibrary{matrix}

\newcommand{\margincolor}{red}      
\definecolor{darkgreen}{rgb}{0,0.7,0}

\newcounter{margincounter}
\newcommand{\marginnum}{
\ifnum\value{margincounter}<10
\textcolor{\margincolor}{\begin{picture}(0,0)\put(2.2,2.4){\circle{9}}\end{picture}\footnotesize\arabic{margincounter}}
\else\ifnum\value{margincounter}<100
\textcolor{\margincolor}{\begin{picture}(0,0)\put(4.256,2.5){\circle{11}}\end{picture}\footnotesize\arabic{margincounter}}
\else
\textcolor{\margincolor}{\begin{picture}(0,0)\put(6.8,2.5){\circle{14}}\end{picture}\footnotesize\arabic{margincounter}}
\fi\fi
}

\newcommand{\into}{\hookrightarrow}

\newcommand{\abs}[1]{\left\lvert#1\right\rvert}

\newcommand{\Z}{\ensuremath{\mathbb{Z}}}
\newcommand{\R}{\ensuremath{\mathbb{R}}}
\newcommand{\C}{\ensuremath{\mathbb{C}}}
\newcommand{\Q}{\ensuremath{\mathbb{Q}}}

\renewcommand{\P}{\ensuremath{\mathbb{P}}}
\renewcommand{\H}{\ensuremath{\mathcal{H}}}
\renewcommand{\O}{\ensuremath{\mathcal{O}}}

\newcommand{\floor}[1]{\left\lfloor#1\right\rfloor}
\newcommand{\ceil}[1]{\left\lceil#1\right\rceil}
\renewcommand{\bar}[1]{\overline{#1}}

\renewcommand{\Im}{\ensuremath{\operatorname{Im}}}

\DeclareMathOperator{\age}{age}

\DeclareMathOperator{\id}{id}

\DeclareMathOperator{\Crit}{Crit}

\DeclareMathOperator{\Res}{Res}

\DeclareMathOperator{\Id}{Id}

\DeclareMathOperator{\Spec}{Spec}

\renewcommand{\Im}{\mathop{\mathrm{Im}}}

\DeclareMathOperator{\GL}{GL}

\DeclareMathOperator{\pr}{pr}

\DeclareMathOperator{\ev}{ev}

\usepackage{color}

\newcommand{\cM}{\mathcal{M}}

\theoremstyle{plain}
\newtheorem{theorem}{Theorem}
\numberwithin{theorem}{section}

\newtheorem{thm}[theorem]{Theorem}

\newtheorem{prop}[theorem]{Proposition}

\newtheorem{cor}[theorem]{Corollary}

\newtheorem{lem}[theorem]{Lemma}
\newtheorem{observation}[theorem]{Observation}

\theoremstyle{definition}
\newtheorem{Definition/Theorem}[theorem]{Definition/Theorem}
\newtheorem{Definition/Proposition}[theorem]{Definition/Proposition}

\newtheorem{Def}[theorem]{Definition}

\newtheorem{Corollary/Definition}[theorem]{Corollary/Definition}

\theoremstyle{remark}

\newtheorem{rem}[theorem]{Remark}
\newtheorem{terminology}[theorem]{Terminology}
\newtheorem{notation}[theorem]{Notation}

\usepackage{fullpage}
\usepackage{longtable}

\newcommand{\Mbar}{\bar{\cM}}

\renewcommand{\setminus}{\smallsetminus}
\DeclareMathOperator{\LGQ}{LGQ}
\DeclareMathOperator{\LGQG}{LGQG}
\DeclareMathOperator{\UQ}{UQ}
\DeclareMathOperator{\UQG}{UQG}
\DeclareMathOperator{\rig}{rig}
\DeclareMathOperator{\uns}{uns}

\DeclareMathOperator{\tw}{tw}
\DeclareMathOperator{\mult}{mult}
\DeclareMathOperator{\loc}{loc}
\DeclareMathOperator{\vir}{vir}

\DeclareMathOperator{\In}{In}
\DeclareMathOperator{\Rec}{Rec}
\DeclareMathOperator{\Giv}{Giv}
\DeclareMathOperator{\old}{old}
\DeclareMathOperator{\nar}{nar}
\DeclareMathOperator{\mov}{mov}

\renewcommand{\hat}{\widehat}

\usetikzlibrary{arrows}
\usetikzlibrary{calc}

\newcommand{\sslash}{\mathbin{/\mkern-6mu/}}

\author{Rob Silversmith}
\address{Department of Mathematics, University of Michigan, 
Ann Arbor, MI 48109}
\email{rsilvers@umich.edu}
\date{\today}
\title{Gromov-Witten theory of toroidal orbifolds and GIT wall-crossing}

\begin{document}
\begin{abstract}
  Toroidal 3-orbifolds $(S^1)^6/G$, for $G$ a finite group, were some
  of the earliest examples of Calabi-Yau 3-orbifolds to be studied in
  string theory. While much mathematical progress towards the
  predictions of string theory has been made in the meantime, most of
  it has dealt with hypersurfaces in toric varieties. As a result,
  very little is known about curve-counting theories on toroidal
  orbifolds. In this paper, we initiate a program to study mirror
  symmetry and the Landau-Ginzburg/Calabi-Yau (LG/CY) correspondence
  for toroidal orbifolds. We focus on the simplest example
  $[E^3/\mu_3],$ where $E\subseteq\P^2$ is the elliptic curve
  $\mathbb{V}(x_0^3+x_1^3+x_2^3).$ We study this orbifold from the
  point of GIT wall-crossing using the gauged linear sigma model, a
  collection of moduli spaces generalizing spaces of stable maps. Our
  main result is a mirror symmetry theorem that applies simultaneously
  to the different GIT chambers. Using this, we analyze wall-crossing
  behavior to obtain an LG/CY correspondence relating the genus-zero
  Gromov-Witten invariants of $[E^3/\mu_3]$ to generalized
  Fan-Jarvis-Ruan-Witten invariants.
\end{abstract}
\maketitle
\section{Introduction}
\subsection{The gauged linear sigma model}\label{intro:GaugedLinearSigmaModel}
Landau-Ginzburg/Calabi-Yau (LG/CY) correspondences are conjectural
relations between invariants of certain moduli spaces. On one hand,
Gromov-Witten theory provides a collection of ``virtual curve counts''
on a Calabi-Yau orbifold $Z$. These are integrals over the moduli
stacks $\Mbar_{g,n}(Z,\beta)$ of twisted stable maps
(\cite{AbramovichGraberVistoli2008}). On the other hand,
Fan-Jarvis-Ruan (\cite{FanJarvisRuan2013}, based on ideas of Witten)
constructed moduli stacks $\mathcal{W}_{g,n}^Z$ parametrizing roots of
line bundles on orbifold curves. These are ``combinatorial'' in
nature, whereas the spaces $\Mbar_{g,n}(Z,\beta)$ are
``geometric''. One may generate \emph{Fan-Jarvis-Ruan-Witten (FJRW)
  invariants} by integrating cohomology classes over
$\mathcal{W}_{g,n}^Z$. Using motivation from string theory, Witten
\cite{Witten1993} predicted that either of these sets of invariants
--- Gromov-Witten or FJRW --- could be computed from the other. A
far-reaching conjecture was precisely formulated by Ruan
(\cite{Ruan2011}), and was recently proven by Chiodo-Iritani-Ruan
(\cite{ChiodoIritaniRuan2013}) in the case where $Z$ is a Calabi-Yau
hypersurface in weighted projective space . The form of the conjecture
is described in more detail below.

Toroidal 3-orbifolds $[(S^1)^6/G]$ are some of the earliest examples
studied in string theory (\cite{DixonHarveyVafaWitten1986}). They form
a rich class of very explicit Calabi-Yau orbifolds (see the
classification
\cite{FischerRatzTorradoVaudrevange2013}). Nevertheless, in many ways
we know very little about them; for example, the program of
\emph{mirror symmetry} for Calabi-Yau orbifolds has been worked out
only for complete intersections in toric stacks, whereas most toroidal
orbifolds are not of this type. Similarly LG/CY correspondences have
not been studied in this context.

Our goal in this paper is to initiate a program towards filling both
of these gaps. Our strategy is based on a common generalization of the
moduli stacks $\Mbar_{g,n}(Z,\beta)$ and $\mathcal{W}_{g,n}^Z$ above,
collectively called the \emph{gauged linear sigma model}, or GLSM. It
was proposed by Witten and formulated mathematically by
Fan-Jarvis-Ruan (\cite{FanJarvisRuan2015}). The main feature of these
new stacks is that they take as input a GIT presentation
$[V\sslash_\theta G].$ Dolgachev-Hu (\cite{DolgachevHu1998}) and
Thaddeus (\cite{Thaddeus1996}) studied how such GIT quotients change
if $\theta$ crosses a wall of a certain finite chamber decomposition,
and similarly the GLSM stacks depend on a chamber of this
decomposition. There is a \emph{geometric chamber} of this
decomposition whose GLSM moduli stack is $\Mbar_{g,n}(Z,\beta),$ and a
so-called \emph{pure Landau-Ginzburg (LG) chamber} whose GLSM moduli
stack is $\mathcal{W}_{g,n}^Z.$ The LG/CY correspondence is thus
recast as a ``GIT wall-crossing'' phenomenon.

We fully work out the genus-zero LG/CY correspondence in the simplest
example of a toroidal orbifold, $[E^3/\mu_3]$ for $E$ an elliptic
curve, and will apply our technique more generally in a subsequent
article. $[E^3/\mu_3]$ is a complete intersection in a GIT quotient
$[\C^{13}\sslash_\theta(\C^*)^4].$ This quotient has a chamber
decomposition with 16 chambers, namely the 16 hyperoctants
$(\pm,\pm,\pm,\pm)$ in $\R^4$. We find the geometric chamber to be the
hyperoctant $(+,+,+,+)$, and the pure LG-chamber to be $(-,-,-,+).$ We
choose a sequence of wall-crossings
\begin{align*}
  (+,+,+,+)\to(+,+,-,+)\to(+,-,-,+)\to(-,-,-,+)
\end{align*}
connecting these two chambers, and in each of the four chambers we
describe the corresponding GLSM moduli stacks as parametrizing
sections of certain line bundles on orbifold curves. To each of these
chambers is then associated a collection of numerical invariants.

GLSM moduli stacks depend upon an additional parameter
$\epsilon\in\Q_{>0}$. The special cases $\Mbar_{g,n}(Z,\beta)$ and
$\mathcal{W}_{g,n}^Z$ above correspond to $\epsilon\to\infty.$ We use
techniques based on those of Ciocan-Fontanine and Kim
(\cite{CiocanFontanineKim2014}) to prove an all-chamber \emph{mirror
  theorem}\footnote{In fact, as $\epsilon$ moves within a chamber
  decomposition of $\Q_{>0}$, this theorem is also a type of
  \emph{wall-crossing}. To avoid confusion, we use the term only to
  refer to walls of GIT chambers.}:

\medskip

\noindent\textbf{Theorem \ref{thm:MirrorTheorem}.}
Let $\{J^{\epsilon,\theta}\}$ be the generating functions of GLSM
invariants defined in Section \ref{sec:JFunction}. Then there is an
explicit invertible transformation identifying $J^{\epsilon,\theta}$
with $J^{\infty,\theta}.$

\medskip

Theorem \ref{thm:MirrorTheorem} is the core of the paper, as well as
its most notable aspect. Finding an appropriate statement of mirror
symmetry for each GIT chamber is the most difficult part of the LG/CY
correspondence. Previous examples of LG/CY correspondences have all
relied on proving mirror theorems in each chamber separately, whereas
our proof is \emph{uniform in $\theta$}, i.e. applies simultaneously
in all chambers.

Setting $\epsilon\to0$ and $\theta=(+,+,+,+)$ recovers an instance of
the \emph{mirror theorem for toric stacks}
(\cite{CoatesCortiIritaniTseng2015,CheongCiocanFontanineKim2015}), and
setting $\epsilon\to0$ and $\theta=(-,-,-,+)$ gives a
\emph{Landau-Ginzburg mirror theorem} similar to the one in
\cite{ChiodoRuan2010}. The purpose of Theorem \ref{thm:MirrorTheorem}
is that one may compute (a restriction of) $J^{0+,\theta}$ for each
$\theta$:

\medskip

\noindent\textbf{Corollary} (Section \ref{sec:CalculatingI})\textbf{.} There are explicit hypergeometric
functions $I^\theta$ that encode the GLSM invariants of
$[\C^{13}\sslash_\theta(\C^*)^4].$ For example (using notation defined
throughout the paper), the series
\begin{align*}
  I^{(+,+,-,+)}(q,\hbar)=\sum_{\substack{\beta_x\in\Z_{\ge0}\\\beta_y\in\Z_{\ge0}\\\beta_z\in(-1/3)\Z_{>0}\setminus\Z\\\beta_a\in(1/3)\Z\ge0}}q_x^{\beta_x}q_y^{\beta_y}q_z^{\beta_z+1/3}q_a^{\beta_a}\frac{\prod_{\substack{\rho\in\mathbf{R}\\\beta_\rho<-1}}\prod_{\floor{\beta_\rho}+1\le\nu\le-1}(D_\rho+(\beta_\rho-\nu)\hbar)}{\prod_{\substack{\rho\in\mathbf{R}\\\beta_\rho\ge0}}\prod_{0\le\nu\le\ceil{\beta_\rho}-1}(D_\rho+(\beta_\rho-\nu)\hbar)}A_xA_y1_{\langle-\beta\rangle}
\end{align*}
encodes the GLSM invariants of the chamber $(+,+,-,+)$.

\medskip

Finally, we relate the functions $I^\theta$ across the four GIT
chambers above. We think of $I^\theta$ as a holomorphic map from a certain
space $\mathcal{N}$ to a subquotient $\H(\theta)$ of the Chen-Ruan
cohomology $H^*_{CR}([\C^{13}\sslash_\theta(\C^*)^4])$. There are two
problems with comparing the functions $I^\theta.$ First, they do not
have the same codomain. Second, they are not defined on all of
$\mathcal{N};$ in fact, each is defined on a different small open
set. Thus we relate them in two steps:
\begin{enumerate}
\item We find a natural sequence of graded isomorphisms (Theorem
  \ref{thm:StateSpaceIsom})
\begin{align*}
  \H(+,+,+,+)\cong\H(+,+,-,+)\cong\H(+,-,-,+)\cong\H(-,-,-,+).
\end{align*}
\item We find a sequence of analytic continuations of $I^\theta$ on
  $\mathcal{N}$ (Section \ref{sec:RelatingIFunctions}).
\end{enumerate}
The method of analytic continuation is based on
\cite{ChiodoRuan2010}. Together, these prove:

\medskip

\noindent\textbf{Theorem \ref{thm:LGCY}.} After analytic continuation
and identification of GLSM state spaces, the functions $I^\theta$
differ by (explicit) linear transformations.

\medskip

The first genus-zero LG/CY correspondence was proved for the quintic
threefold by Chiodo and Ruan (\cite{ChiodoRuan2010}). It has since
been proven for several other classes of targets, including Calabi-Yau
hypersurfaces in weighted projective spaces
(\cite{ChiodoIritaniRuan2013}), many classes of Calabi-Yau complete
intersections in weighted projective spaces
(\cite{Clader2013,CladerRoss2015}), and some other examples
(\cite{PriddisShoemaker2014,Schaug2015}). (Acosta (\cite{Acosta2014})
also developed a similar correspondence for non-Calabi-Yau
hypersurfaces in weighted projective spaces.) All of these used
techniques quite different from those presented here;
\cite{ChiodoIritaniRuan2013}, \cite{Clader2013},
\cite{PriddisShoemaker2014}, and \cite{Schaug2015} used a direct
computational method, and \cite{CladerRoss2015} (following previous
work for hypersurfaces in \cite{LeePriddisShoemaker2014}) used a
reduction to the crepant transformation conjecture, previously
established in the relevant cases in
\cite{CoatesIritaniJiang2014}.

While toroidal 3-orbifolds are our primary motivation, we expect
Theorem \ref{thm:MirrorTheorem} to apply much more broadly, to large
classes of complete intersections in GIT quotients carrying certain
torus actions. Because of this, we view it as a general conceptual
approach to LG/CY correspondences. In the future we hope to explore
the generality in which this technique applies.

\medskip

\noindent\textbf{Plan of the paper.} Section \ref{sec:Background}
contains background facts about GIT quotients and orbifold curves. In
Sections \ref{sec:Setup} and \ref{sec:ModuliSpaces}, we introduce the
``target'' stacks $Z(\theta)\subseteq[\C^{13}\sslash_\theta(\C^*)^4]$
and their associated moduli stacks
$\LGQ_{0,m}^\epsilon(Z(\theta),\beta)$. In Section
\ref{sec:EvaluationMapsInvariants} we define the GLSM state spaces
$\H(\theta)$, and the GLSM invariants, which are integrals over the
moduli stacks $\LGQ_{0,m}^\epsilon(Z(\theta),\beta)$. In Section
\ref{sec:Localization} we define natural group actions on our moduli
stacks, to be used for fixed-point localization. Section
\ref{sec:GeneratingFunctions} contains the definitions of various
generating functions, which we use to prove our all-chamber mirror
theorem in Section \ref{sec:MirrorTheorems}. Finally, in Sections
\ref{sec:CalculatingI} and \ref{sec:RelatingIFunctions} we compute the
series $I^\theta(q,\hbar)$ for general $\theta$, and relate the
resulting formulas to each other. Section \ref{sec:NotationTable} is a
table of notation.

\medskip

\noindent \textbf{Acknowledgements.} I would like to thank Yongbin
Ruan for introducing me to the field, suggesting the problem, and
providing guidance along the way. I am grateful to Rohini Ramadas and
Dustin Ross for helpful conversations.

\section{Background}\label{sec:Background}
We work over $\C$. We denote by $\mu_d$ the group of $d$th roots of
unity in $\C.$
\subsection{Geometric invariant theory and stack quotients}\label{sec:GIT}
\begin{Def}
  Let $V$ be a smooth affine variety, let $G$ be a reductive algebraic
  group acting on $V$, and let $\theta:G\to\C^*$ be a character of
  $G$. This defines a $G$-action on the trivial bundle $V\times\C$ by
  $g\cdot(v,z)=(g\cdot v,\theta(g)z)$. The $G$-invariant sections of
  $V\times\C$ are called \emph{$\theta$-equivariant functions}. The
  \emph{$\theta$-unstable locus} $V^{uns}(\theta)$ in $V$ is the
  subvariety defined by the vanishing of all $N\theta$-equivariant
  functions, for $N\ge1$. The \emph{$\theta$-semistable locus}
  $V^{ss}(\theta)$ is the complement of $V^{uns}(\theta).$ The
  \emph{$\theta$-stable locus} $V^s(\theta)$ is the set of semistable
  points with finite $G$-stabilizer whose $G$-orbit is closed in
  $V^{ss}(\theta)$.
\end{Def}
In this paper we will always have $V^s(\theta)=V^{ss}(\theta).$ In
this case the \emph{GIT stack quotient} $[V\sslash_\theta
G]:=[V^{ss}(\theta)/G]$ is a smooth separated Deligne-Mumford stack.
\begin{prop}\label{prop:MapPrincipalBundle}
  Let $G$ be a group acting (on the left) on a variety $V$, and let
  $S$ be any scheme. There is a natural correspondence between
  \begin{enumerate}
  \item Maps $f:S\to[V/G]$,\label{item:Map}
  \item Principal $G$-bundles $\mathcal{P}$ on $S$ together with a
    $G$-equivariant map $\phi:\mathcal{P}\to V$,
    and\label{item:EquivMap}
  \item Principal $G$-bundles $\mathcal{P}$ on $S$ together with a
    section $\sigma_\phi$ of the associated fiber
    bundle $\mathcal{P}\times_GV\to S.$\label{item:Section}
  \end{enumerate}
\end{prop}
\begin{proof}
  The equivalence of \eqref{item:Map} and \eqref{item:EquivMap} is by
  definition of a map $S\to[V/G]$. We show that \eqref{item:EquivMap}
  and \eqref{item:Section} are equivalent. A section
  $\sigma:S\to\mathcal{P}\times_GV$ gives a map $S\to[V/G]$ by
  composition with the projection $\mathcal{P}\times_GV\to[V/G].$
  Conversely, given a $G$-equivariant map $\mathcal{P}\to V$, we
  define a section $S\to\mathcal{P}\times_GV$ by mapping
  $s\mapsto(p,\phi(p)).$ It is straightforward to check these are
  inverse to each other.
\end{proof}
With $V$, $G$, and $S$ as in Proposition
\ref{prop:MapPrincipalBundle}, let $\rho:G\to\C^*$ be a character of
$G$. This induces a $G$-equivariant structure on the trivial bundle
$V\times\C$ by
$$g\cdot(v,z):=(g\cdot v,\rho(g)z).$$ (We write $V\times\C_\rho$ to
keep track of the $G$-action.) There is an associated line bundle
$L_\rho=[(V\times\C_\rho)/G]$ on $[V/G].$ Abusing notation, we also
write $L_\rho$ for restrictions of $L_\rho$ to substacks
$[V\sslash_\theta G]$.
\begin{prop}
  Let $S$ be a scheme, and let $f:S\to[V/G]$ be a map, with
  corresponding principal $G$-bundle $\mathcal{P}$ and map
  $\phi:\mathcal{P}\to V.$ Then for any character $\rho$ of $G$,
  \begin{align*}
    f^*L_\rho\cong\mathcal{P}\times_G\C_\rho\cong\sigma_\phi^*(\mathcal{P}\times_G(V\times\C_\rho)).
  \end{align*}
\end{prop}

\subsection{3-stable curves and their line bundles}\label{sec:3StableCurves}
\begin{Def}\label{Def:OrbifoldCurve}
  An \emph{$m$-marked prestable orbifold curve} $(C,b_1,\ldots,b_m)$
  is a balanced twisted nodal $m$-pointed curve in the sense of
  \cite{AbramovichVistoli2002}. That is, \'etale locally at each point
  $P$ it is either:
  \begin{enumerate}
  \item isomorphic to $[\C/\mu_{d_P}]$ for some $d_P$, where
    $\mu_{d_P}$ acts by multiplication, and $P$ is identified with 0,
    or\label{item:OrbifoldPoint}
  \item isomorphic to $[\mathbb{V}(xy)/\mu_{d_P}],$ where $\mathbb{V}(xy)\subseteq\C^2$
    is the union of the coordinate axes, $\mu_{d_P}$ acts by
    multiplication by opposite roots of unity on $x$ and $y$, and $P$
    is identified with $(0,0)$,\label{item:Node}
  \end{enumerate}
  together with $m$ distinct marked points $b_1,\ldots,b_m$ of type
  \eqref{item:OrbifoldPoint}, including all of those with $d_P>1$. We
  refer to $d_P$ as the \emph{order} of $P$ and $\mu_{d_P}$ as the
  \emph{isotropy group} of $P$. We often write $C$ instead of
  $(C,b_1,\ldots,b_m).$
\end{Def}
\begin{rem}\label{rem:ChooseBranchOfNode}
  For points $P$ of type \eqref{item:OrbifoldPoint}, there is a
  canonical identification of the isotropy group of $P$ with
  $\mu_{d_P};$ the canonical generator is that which acts by
  multiplication by $e^{2\pi i/d_P}$ on $T_PC.$ However, this is not
  true for points of type \eqref{item:Node}, since each element acts
  by opposite roots of unity on the two branches. Instead, there is a
  canonical identification \emph{after} choosing a branch of the node.
\end{rem}
Note that $d_P=1$ for all but finitely many points $P$ of $C$. An
$m$-marked prestable orbifold curve admits a \emph{coarse moduli
  space} map to an ordinary $m$-marked prestable curve
$\bar{C}$. Olsson (\cite{Olsson2007}) proved that families of
$m$-marked orbifold curves whose coarse moduli spaces have arithmetic
genus $g$ form an algebraic stack $\mathfrak{M}_g^{\tw}$. We will only
be interested in the case $g=0.$ For the purposes of this paper we
restrict to an open substack of $\mathfrak{M}_0^{\tw}$, as follows.
\begin{Def}[\cite{RossRuan2015}]
  A genus zero $m$-marked orbifold curve is \emph{stable} if each
  irreducible component has at least three marked points or nodes. An
  $m$-marked \emph{3-stable curve} is a stable genus zero $m$-marked
  orbifold curve such that all marked points and nodes are orbifold
  points of order 3.
\end{Def}
Next we review some facts about line bundles on orbifold curves.
\begin{Def}
  Let $C$ be an $m$-marked prestable orbifold curve. A \emph{line
    bundle} on $C$ is a stack $\mathcal{L}$ with a map to $C,$ such
  that $\mathcal{L}$ is \'etale locally isomorphic on $C$ it is
  isomorphic to one of the following, corresponding to the cases in
  Definition \ref{Def:OrbifoldCurve}:
  \begin{enumerate}
  \item $[\C\times\C/\mu_{d_P}],$ where $\mu_{d_P}$ acts by
    multiplication on the first copy of $\C$ and linearly on
    the second copy, or\label{item:LineBundleAtOrbifoldPoint}
  \item $[\mathbb{V}(xy)\times\C/\mu_{d_P}],$ where $\mu_{d_P}$ acts on $\mathbb{V}(xy)$
    as in item \eqref{item:Node} of Definition
    \ref{Def:OrbifoldCurve}, and linearly on $\C$.
  \end{enumerate}
\end{Def}
\begin{Def}\label{Def:Multiplicity}
  In case \eqref{item:LineBundleAtOrbifoldPoint}, $e^{2\pi
    i/d_P}\in\mu_{d_P}$ acts on the second copy of $\C$ by
  multiplication by $e^{2\pi ik/d_P}$ for some $0\le k<d_P.$ We call
  the rational number $\mult_P(\mathcal{L}):=k/d_P$ the
  \emph{multiplicity} or the \emph{monodromy} of $\mathcal{L}$ at
  $P$. If $\mult_P(\mathcal{L})=0$, we say $\mathcal{L}$ has
  \emph{trivial monodromy} at $P$.
\end{Def}
\begin{rem}
  We also refer to the multiplicity of $\mathcal{L}$ at a node of
  $C$. As in Remark \ref{rem:ChooseBranchOfNode}, this is well-defined
  only after choosing a branch of the node. In this case we will refer
  to the multiplicity of $\mathcal{L}$ ``on one side of the node.''
\end{rem}
One can similarly define vector bundles and their duals, sections,
tensor products, and direct sums on orbifold curves. These behave
largely the same as on nonstacky curves, with a few differences. For
example, local sections of $\mathcal{L}$ at an orbifold point $P$ of
$C$ are $\mu_{d_P}$-invariant sections as in
\eqref{item:LineBundleAtOrbifoldPoint} above, so in particular, if
$\mathcal{L}$ has nontrivial monodromy at $P$ then every \emph{local}
section of $\mathcal{L}$ vanishes at $P$. More specifically, if we
define the order of vanishing of a section via pulling back along an
\'etale cover by a scheme, then the order of vanishing of a section at
an orbifold point $P$ is an element of $\mult_P(\mathcal{L})+\Z.$
Also, the monodromy of a tensor product of line bundles is given by
$$\mult_P(\mathcal{L}\otimes\mathcal{L}')=\mult_P(\mathcal{L})+\mult_P(\mathcal{L}')\mod1.$$

Isomorphism classes of line bundles on orbifold curves may be easily
understood via the \emph{divisor-line bundle correspondence for smooth
  orbifold curves}. We state it only for genus zero curves, as these
are all we consider.
\begin{Def}
  A \emph{Weil divisor} on a smooth orbifold curve is a (finite)
  formal sum $D=\sum_{P\in C}a_PP$ of points. The \emph{degree}
  $\deg(D)$ of $D$ is $\sum_{P\in C}\frac{a_P}{d_P},$ where $d_P$ is
  the order of $P$. The degree is clearly additive under addition of
  Weil divisors.
\end{Def}
The notion of rational equivalence of Weil divisors on orbifold curves
is identical to that for schemes. For genus zero orbifold curves, it
reduces to the following.
\begin{Def}
  Two Weil divisors $D=\sum_{P\in C}a_PP$ and $D'=\sum_{P\in C}a_P'P$
  on a genus zero smooth orbifold curve are \emph{rationally
    equivalent} if for each $P\in C$ we have $a_P\equiv a_P'\mod d_P$
  for all $P$, and $\deg(D)=\deg(D').$ We write $[D]$ of the rational
  equivalence class of $D.$
\end{Def}
In particular, the divisors $d_PP$ are rationally equivalent for all
points $P$. We have the following correspondence:
\begin{prop}[See \cite{VoightZureickBrown2015}]\label{prop:DivisorLineBundle}
  The additive group of Weil divisors on a smooth orbifold curve $C$
  up to rational equivalence is naturally isomorphic to the group of
  line bundles on $C$ up to isomorphism, with the operation of tensor
  product.
\end{prop}
As for schemes, the zeroes and poles of a rational section of a line
bundle $\mathcal{L}$ define a divisor, and this determines one
direction of the correspondence. From this we see that the
multiplicity $\mult_P(\mathcal{L})$ is identified with the rational
number $\frac{a_P}{d_P}\mod1$. The correspondence also allows us to
define the degree $\deg(\mathcal{L})$ of a line bundle, additive under
tensor product.

We will often refer to the \emph{log-canonical bundle}
$\omega_{C,\log}$ on a 3-stable curve $C$. This is a line bundle whose
sections are (correctly defined) holomorphic 1-forms on $C$, twisted
by the divisor $\sum_ib_i$ of marked points. As with ordinary nodal
curves, these holomorphic 1-forms may have simple poles at nodes. It
will be important that $\omega_{C,\log}$ has trivial monodromy at each
orbifold point, and has degree $2g-2+m=-2+m.$

The curve $\P_{3,1}:=[(\C^2\setminus\{0\})/\C^*]$, where $\C^*$ acts
with weights 3 and 1 on the coordinates respectively, will be
particularly useful to us. This curve is smooth, and has a single
orbifold marked point of order 3 at $[1:0],$ which we refer to as
$\infty$. (In particular, it is not 3-stable.) The group of Weil
divisor classes (and hence the group of isomorphism classes of line
bundles) is generated by a single element $[\infty]$. Following
convention we refer to the corresponding isomorphism class of line
bundles as $\O_{\P_{3,1}}(1),$ and its tensor powers by
$\O_{\P_{3,1}}(n).$ Note that $\deg(\O_{\P_{3,1}}(n))=n/3.$ The log
canonical bundle $\omega_{\P_{3,1},\log}$ (viewing $\P_{3,1}$ as a
1-marked orbifold curve) has degree $-2+1=-1,$ so it is isomorphic to
$\O_{\P_{3,1}}(-3).$

\subsection{The inertia stack and Chen-Ruan cohomology}\label{sec:ChenRuan} Let $X$
be a smooth complex orbifold, i.e. a smooth connected Deligne-Mumford
stack of finite type over $\C$.
\begin{Def}
  Suppose $X=[M/G],$ where $M$ is a smooth scheme and $G$ is an
  abelian group. Then the \emph{inertia stack} of $X$ is the quotient
  $IX:=[\tilde{M}/G],$ where $\tilde{M}$ is the scheme parametrizing
  pairs $(m,g)$ where $m\in M$ and $g\in G_m,$ where $G_m\subseteq G$
  is the stabilizer of $m$.

  For fixed $g\in G$, let $\tilde{M}(g)$ be the open closed subscheme
  of $\tilde{M}$ of elements of the form $(m,g).$ The \emph{rigidified
    inertia stack} is the union
  \begin{align*}
    \bar{I}X:=\bigcup_{g\in G}[\tilde{M}(g)/(G/\langle g\rangle)].
  \end{align*}
\end{Def}
(Note: In the cases we consider, $\tilde{M}(g)$ is empty for all but
finitely many $g$. More generally, the rigidified inertia stack is
slightly more difficult to define.) The inertia stack and rigidified
inertia stack are defined in much more generality (see
\cite{AbramovichGraberVistoli2008}), but we will only need the cases
above.
\begin{rem}
  The rigidified inertia stack has the same coarse moduli space as the
  inertia stack. Indeed, the only difference between the two is that
  the inertia stack has ``extra'' stack structure. For example,
  $B\mu_3:=[\Spec\C/\mu_3]$ has inertia stack $IB\mu_3\cong
  B\mu_3\sqcup B\mu_3\sqcup B\mu_3$, and rigidified inertia stack
  $\bar{I}B\mu_3\cong B\mu_3\sqcup\Spec\C\sqcup\Spec\C$.
\end{rem}
\begin{terminology}
  Connected components of $IX$ and $\bar{I}X$ are called
  \emph{sectors}. Both $IX$ and $\bar{I}X$ contain $X$ as a connected
  component, namely the quotient $[\tilde{M}(e)/G]$ for $e\in G$ the
  identity. This component is referred to as the \emph{untwisted
    sector} of $IX$ or $\bar{I}X$, and other components are called
  \emph{twisted sectors}. A twisted sector $X'$ has a corresponding
  element $g\in G$, so we refer to $X'$ as a \emph{$g$-twisted
    sector}.
\end{terminology}
\begin{rem}
  Note that there is a forgetful map $IX\to X$ that realizes each
  component of $IX$ as a closed substack of $X$. There are also
  \emph{inversion} automorphisms $\upsilon$ on $IX$ and $\bar{I}X,$
  which send $(m,g)\mapsto(m,g^{-1}).$
\end{rem}
\begin{Def}
  The \emph{Chen-Ruan cohomology} of $X$ is defined, as a $\C$-vector
  space, to be
  \begin{align*}
    H^*_{CR}(X):=H^*(\bar{I}X,\C),
  \end{align*}
  where the right side denotes the singular cohomology of the coarse
  moduli space.
\end{Def}
The grading of $H^*_{CR}(X)$ is different from the usual one. To
describe it we use the following:
\begin{Def}
  If $X'$ is a $g$-twisted sector, then for generic $(m,g)\in X'$ we
  diagonalize the automorphism of $T_mM$ induced by $g$. The
  eigenvalues are roots of unity $e^{2\pi i\alpha_j}$ with
  $0\le\alpha_j<1.$ The \emph{age} of $X'$, denoted $\age(X'),$ is
  defined to be $\sum_j\alpha_j.$
\end{Def}
We may now define the grading:
\begin{align*}
  H^k_{CR}(X):=\bigoplus_{X'}H^{k-2\age(X')}(X',\C).
\end{align*}
\begin{rem}
  There is a natural graded embedding $H^*(X,\C)\into H^*_{CR}(X)$,
  induced by the inclusion $X\into IX$ of the untwisted sector.
\end{rem}
\begin{notation}
  For each sector $X'$, there is a class $1_{X'}\in H^*_{CR}(X)$ that
  is the unit in $H^*(X',\C).$ (Its degree may be nonzero under the
  grading above.) If $\tilde{M}(g)$ is connected we will write
  $1_{X'}=1_g.$ We will also write $1_{X'}$ or $1_g$ for the
  corresponding class on the \emph{nonrigidified} inertia stack.
\end{notation}
Here are two important properties of Chen-Ruan cohomology:
\begin{enumerate}
\item There is a natural notion of cup product on $H^*_{CR}(X)$,
  compatible with the grading. (Note that the cup product on
  $H^*(IX,\C)$ is not compatible with the grading as defined.)
\item If $X$ is proper, there is a (perfect) Poincar\'e pairing on
  $H^*_{CR}(X),$ defined by
  $\langle\alpha,\beta\rangle_X=\int_{IX}\alpha\cup\upsilon^*\beta.$
  (The $\upsilon^*\beta$ in the integrand makes the pairing compatible
  with the grading.)
\end{enumerate}

\subsection{Cohomology of
  $\P^2/\mu_3$}\label{sec:CohomologyOfLineBundle}
One of the basic objects of this paper is the stack quotient
$[\P^2/\mu_3],$ where $\mu_3$ acts by multiplication on the first
coordinate. In this section we consider only singular cohomology of
the coarse moduli space $\P^2/\mu_3$, not Chen-Ruan cohomology. Write
$[x_0:x_1:x_2]$ for points of $\P^2$, and $\eta:\P^2\to[\P^2/\mu_3]$
for the quotient map. We use the following notation:
\begin{table}[h]
  \centering
  \begin{tabular}{l|l|l|l}
    Symbol&Locus in $\P^2$&Symbol&Locus in $[\P^2/\mu_3]$\\\hline
    $\tilde{L_0}$&Line $\mathbb{V}(x_0)\subseteq\P^2$&$L_0$&$\eta(\tilde{L_0})$\\
    $\tilde{P_0}$&Point $[1:0:0]\in\P^2$&$P_0$&$\eta(\tilde{P_0})$\\
    $\tilde{L'}$&Any line through $\tilde{P_0}$\quad\quad\quad\quad\quad\quad\quad\quad&$L'$&$\eta(\tilde{L'})$\\
    $\tilde{P'}$&Any point on $\tilde{L_0}$&$P'$&$\eta(\tilde{P'})$
  \end{tabular}
\end{table}
The $\mu_3$-fixed locus in $\P^2$ is $\tilde{L_0}\cup\tilde{P_0}.$ The
lines $\tilde{L'}$ are $\mu_3$-invariant, and these are the only
$\mu_3$-invariant lines (other than $\tilde{L_0}$).

$H^*(\P^2/\mu_3,\Z)$ is generated by $\{1,[L_0],[L'],[P_0],[P']\}.$
Note that $3[L_0]=3[L']$ and $3[P_0]=3[P'].$ Therefore we define the
generators $H:=[L_0]=[L']$ and $P:=[P_0]=[P']$ of the \emph{complex}
cohomology $H^*(\P^2/\mu_3,\C).$
\begin{rem}
  We usually consider not $[\P^2/\mu_3]$ but the line bundle
  $[\O_{\P^2}(-3)/\mu_3]$. Here $\O_{\P^2}(-3)$ is the quotient
  $((\C^3\setminus\{(0,0,0)\})\times\C)/\C^*$, where $\C^*$ acts with
  weights $(1,1,1,-3).$ The group $\mu_3$ acts on the
  (quasi-)homogeneous coordinates of $\O_{\P^2}(-3)$ by
  $\zeta\cdot[x_0:x_1:x_2:p_x]=[\zeta x_0:x_1:x_2:p_x]$. The pullback
  map $H^*(\P^2/\mu_3,\C)\to H^*(\O_{\P^2}(-3)/\mu_3,\C)$ is an
  isomorphism, and we will also use the symbols $H$ and $P$ to denote
  the corresponding classes in the latter.
\end{rem}

\section{The targets $Z(\theta)$}\label{sec:Setup}
\subsection{Notation}
We begin by fixing notation that we will use throughout the paper. It
is essentially in agreement with the notation of
\cite{FanJarvisRuan2015}. We let $V=\C^{13}$ with coordinates
\begin{align*}
  (x_0,x_1,x_2,y_0,y_1,y_2,z_0,z_1,z_2,a,p_x,p_y,p_z).
\end{align*}
Let $G=(\C^*)^4,$ with action on $V$ by
\begin{align*}
  (t_x,t_y,t_z,t_a)\cdot&(x_0,x_1,x_2,y_0,y_1,y_2,z_0,z_1,z_2,a,p_x,p_y,p_z)\\
  &=(t_xt_a^{-1}x_0,t_xx_1,t_xx_2,t_yt_a^{-1}y_0,t_yy_1,t_yy_2,t_zt_a^{-1}z_0,t_zz_1,t_zz_2,t_a^3a,t_x^{-3}p_x,t_x^{-3}p_y,t_x^{-3}p_z).
\end{align*}
Define another group $\C^*_R=\C^*$ (denoted thus to avoid confusion)
acting on $V$ by
\begin{align*}
  t_R\cdot&(x_0,x_1,x_2,y_0,y_1,y_2,z_0,z_1,z_2,a,p_x,p_y,p_z)\\
  &=(x_0,x_1,x_2,y_0,y_1,y_2,z_0,z_1,z_2,a,t_Rp_x,t_Rp_y,t_Rp_z).
\end{align*}
\begin{rem}
  The groups $G$ and $\C^*_R$ are ``independent'' in that
  $\langle G,\C^*_R\rangle\subseteq\GL(V)$ is isomorphic to
  $G\times\C^*_R$.
\end{rem}
Let $W:V\to\C$ be the $G$-invariant function
\begin{align*}
  W:=p_x(ax_0^3+x_1^3+x_2^3)+p_y(ay_0^3+y_1^3+y_2^3)+p_z(az_0^3+z_1^3+z_2^3).
\end{align*}
It is $\C^*_R$-homogeneous of degree 1.
\begin{terminology}
  In the literature, $W$ is referred to as a \emph{superpotential} on
  $V$, and $\C^*_R$ is known as an \emph{$R$-charge}.
\end{terminology}

As $G\times\C^*_R$ acts diagonally, $V$ is a direct sum of
1-dimensional $G\times\C^*_R$-representations, corresponding to the
list of characters
\begin{align*}
  \mathbf{R}=\{\hat{t_x}-\hat{t_a},\hat{t_x},\hat{t_x},\hat{t_y}-\hat{t_a},\hat{t_y},\hat{t_y},\hat{t_z}-\hat{t_a},\hat{t_z},\hat{t_z},3\hat{t_a},-3\hat{t_x}+\hat{t_R},-3\hat{t_y}+\hat{t_R},-3\hat{t_z}+\hat{t_R}\},
\end{align*}
where $\hat{t_x}$ is the character $(t_x,t_y,t_z,t_a,t_R)\mapsto t_x,$
and similarly for the others.

The critical locus $\Crit(W)$ of $W$ is $G$-invariant. Let
\begin{align*}
  X:&=[V/G]\\
  Z:&=[\Crit(W)/G],
\end{align*}
with $\iota:Z\into X$ the natural embedding. $X$ and $Z$ are
nonseparated Artin stacks, but we consider certain open GIT quotient
substacks, as follows.

\subsection{The characters $\theta$}\label{sec:Characters}
We study GIT quotients $[V\sslash_\theta G]=[V^{ss}(\theta)/G]$ where
$\theta:G\to\C^*$ varies. The Euclidean space parametrizing
$\theta:(t_x,t_y,t_z,t_a)\mapsto t_x^{e_x}t_y^{e_y}t_z^{e_z}t_a^{e_a}$
is isomorphic to $\Z^4\otimes\R=\R^4$. The 16 GIT chambers are those
on which the signs of $e_x$, $e_y,$ $e_z,$ and $e_a$ are constant. We
define characters
$\Theta=\{\theta^{xyza},\theta_{z}^{xya},\theta_{yz}^{xa},
\theta_{xyz}^{a}\}$ representing four of the chambers:
\begin{align*}
  \theta^{xyza}(t_x,t_y,t_z,t_a)&=t_x^3t_y^3t_z^{3}t_a^3\\
  \theta_{z}^{xya}(t_x,t_y,t_z,t_a)&=t_x^3t_y^3t_z^{-3}t_a^3\\
  \theta_{yz}^{xa}(t_x,t_y,t_z,t_a)&=t_x^3t_y^{-3}t_z^{-3}t_a^3\\
  \theta_{xyz}^{a}(t_x,t_y,t_z,t_a)&=t_x^{-3}t_y^{-3}t_z^{-3}t_a^3.
\end{align*}
(The multiples of 3 will simplify notation later.) We then define for
$\theta\in\Theta:$
\begin{align*}
  X(\theta):&=[V\sslash_\theta G]\subseteq X\\
  Z(\theta)=[\Crit(W)\sslash_\theta G]:&=[(\Crit(W)\cap
  V^{ss}(\theta))/G]\subseteq Z.
\end{align*}
Again, we use $\iota$ to denote the embedding
$Z(\theta)\into X(\theta)$.
\begin{terminology}
  If $x$ (resp. $y$, $z$) is in the subcript of $\theta,$ we will say
  ``$x$ (resp. $y$, $z$) is a subscript variable.'' Similarly we refer
  to ``superscript variables,'' and to modifying a character by
  ``moving $x$ from the superscript to the subscript.'' (In every
  case, $a$ is a superscript variable.)
\end{terminology}
\begin{rem}\label{rem:Interpolation}
  As mentioned in the introduction, the characters $\theta^{xyza}$ and
  $\theta_{xyz}^a$ are of primary interest, as from them we will
  construct moduli spaces previously studied in Gromov-Witten theory
  and Fan-Jarvis-Ruan-Witten theory, respectively. The characters
  $\theta^{xya}_z$ and $\theta^{xa}_{yz}$ will provide a means of
  interpolating between these moduli spaces. By symmetry of $x$, $y$,
  and $z$, everything that follows regarding the characters in
  $\Theta$ works equally well for the characters $\theta^{xza}_y,$
  $\theta^{yza}_x,$ $\theta^{ya}_{xz}$, and $\theta^{za}_{xy}.$

  Characters on walls of the chamber decomposition (such as
  $\theta_y^{xa}(t_x,t_y,t_z,t_a):=t_x^3t_y^{-3}t_a^3$) are not
  considered. The corresponding GIT quotients are not well-behaved,
  and the moduli spaces of Section \ref{sec:ModuliSpaces} are not
  defined in this situation. 

  The other missing characters are those where $a$ is a subscript
  variable, e.g.
  $\theta_{za}^{xy}(t_x,t_y,t_z,t_a):=t_x^3t_y^{3}t_z^{-3}t_a^{-3}.$
  These are not needed to carry out the interpolation
  mentioned. However, this case may be of independent interest and we
  hope to return to it in the future.
\end{rem}
\begin{terminology}
  The GIT chambers are called \emph{phases} in the physics literature,
  and various manifestations of GIT wall-crossing (such as the LG/CY
  correspondence) are known as \emph{phase transitions}.
\end{terminology}
The following will help us state the definitions of moduli spaces in
Section \ref{sec:ModuliSpaces}. Define characters
\begin{align*}
  \vartheta^{xyza}(t_x,t_y,t_z,t_a)&=t_x^3t_y^3t_z^{3}t_a^3\\
  \vartheta_{z}^{xya}(t_x,t_y,t_z,t_a)&=t_x^3t_y^3t_z^{-3}t_a^3t_R\\
  \vartheta_{yz}^{xa}(t_x,t_y,t_z,t_a)&=t_x^3t_y^{-3}t_z^{-3}t_a^3t_R^2\\
  \vartheta_{xyz}^{a}(t_x,t_y,t_z,t_a)&=t_x^{-3}t_y^{-3}t_z^{-3}t_a^3t_R^3.
\end{align*}
of $G\times\C^*_R.$ These lift the characters in $\Theta$ to
$G\times\C^*_R\supseteq G.$

\medskip

\noindent\textbf{The GIT quotients $X(\theta)$ and $Z(\theta)$.} A routine calculation of the equations defining $V^{\uns}(\theta)$ and
$\Crit(W)$ yields the following characterization:
\begin{enumerate}
\item If $x$ is a superscript variable of $\theta$, then
  $\mathbb{V}(x_0,x_1,x_2)\subseteq V^{\uns}(\theta)$. If $x$ is a subscript
  variable of $\theta,$ then $\mathbb{V}(p_x)\subseteq V^{\uns}(\theta)$. For
  every $\theta$ we have $\mathbb{V}(a)\subseteq V^{\uns}(\theta).$ These three
  conditions entirely cut out $V^{\uns}(\theta)$ in $V$.
\item $X(\theta^{xyza})$ is isomorphic to the total space of the rank
  3 vector bundle $[\O_{\P^2}(-3)^3/\mu_3]$ over $[(\P^2)^3/\mu_3].$
  Here $\mu_3$ acts on each copy of $\P^2$ and $\O_{\P^2}(-3)$ as in
  Section \ref{sec:CohomologyOfLineBundle}. $Z(\theta^{xyza})$ is
  isomorphic to the complete intersection $[E^3/\mu_3]$ inside the
  zero section of this vector bundle, where $E$ is the
  $\mu_3$-invariant elliptic curve
  $\mathbb{V}(x_0^3+x_1^3+x_2^3)\subseteq\P^2.$\\
\item
  $X(\theta^{xya}_z)\cong\left[(\O_{\P^2}(-3)^2\times[\C^3/\mu_3])\Big/\mu_3\right]$,
  where $\mu_3$ acts on $\C^3$ by
  scaling. $Z(\theta^{xya}_z)\cong[(E^2\times B\mu_3)/\mu_3]$, where
  $B\mu_3\subseteq[\C^3/\mu_3]$ is the origin.\\
\item
  $X(\theta^{xa}_{yz})\cong\left[(\O_{\P^2}(-3)\times[\C^3/\mu_3]^2)\Big/\mu_3\right]$,
  and $Z(\theta^{xa}_{yz})\cong[(E\times(B\mu_3)^2)/\mu_3]$.\\
\item
  $X(\theta^{a}_{xyz})\cong\left[[\C^3/\mu_3]^3/\mu_3\right]\cong[\C^9/(\mu_3)^4]$,
  and $Z(\theta^{a}_{xyz})\cong B((\mu_3)^4)$ is the origin.
\end{enumerate}
\begin{rem}\label{rem:CalabiYau}
  Using e.g. the $j$-invariant, we may check that $E$ is
  isomorphic to a quotient of $\C$ by the lattice generated by
  $\{1,e^{2\pi i/6}\},$ and the $\mu_3$-action lifts to the
  multiplication action of $\mu_3$ on $\C$. In this picture, we may
  identify $H^{1,0}(E)$ with $\C d\tau,$ where $\tau$ is the
  coordinate on $\C.$ The $\mu_3$-action on $H^{1,0}(E)$ is
  by multiplication.

  Since we have $H^{3,0}(E^3)\cong H^{1,0}(E)^{\otimes3},$ the
  diagonal $\mu_3$-action on $H^{3,0}(E^3)$ is trivial. In other
  words, the nonvanishing holomorphic 3-form on $E^3$ (unique up to
  scaling) is invariant under the $\mu_3$-action, so it descends to
  $Z(\theta^{xyza})$. That is, $Z^{xyza}$ is Calabi-Yau.
\end{rem}
\begin{rem}\label{rem:RActsTrivially}
  In every case, $\C^*_R$ acts trivially on $Z(\theta)$. For example,
  $\C^*_R$ acts on $X(\theta^{xyza}$ by scaling on the fibers of the
  vector bundle $[(\O_{\P^2}(-3))^3/\mu_3],$ so acts trivially on
  $Z(\theta^{xyza})$ since $Z(\theta^{xyza})$ lies inside the zero
  section. Similarly, $\C^*_R$ acts on $X(\theta^a_{xyz})$ by scaling
  the coordinates of $[\C^9/\mu_3]$, so acts trivially on the origin.
\end{rem}
\begin{rem}
  We may check that for $\theta\in\Theta,$ $V^{ss}(\theta)$ is equal
  to $V^{ss}(\vartheta),$ for $\vartheta$ the lift of $\theta$ defined
  above.
\end{rem}

\subsection{Toric divisors}\label{sec:ToricDivisors}
We will often refer to the \emph{toric divisors}
$D_\rho\in H^2(X(\theta),\C)$ and their pullbacks
$\iota^*D_\rho\in H^2(Z(\theta),\C).$ A character $\rho:G\to\C^*$
defines a line bundle $L_\rho$ as in Section \ref{sec:GIT}. For
$\rho\in\mathbf{R},$ the corresponding coordinate $s_\rho$ is a
section of $L_\rho$. As usual, abusing notation we also write $L_\rho$
and $s_\rho$ for the restriction to each quotient
$X(\theta)\subseteq X.$ We define $D_\rho:=c_1(L_\rho)$.

For $\theta=\theta^{xya}_z,$ we compute $D_\rho$ and $\iota^*D_\rho$
explicitly. Observe that $X(\theta)$ admits projection maps
$\pr_x,\pr_y:X(\theta)\to[\O_{\P^2}(-3)/\mu_3]$ and
$\pr_z:X(\theta)\to[\C^3/(\mu_3)^2].$ We consider the vanishing loci
in $X(\theta)$ of the sections $s_\rho.$ The sections
$s_{\rho_{x_0}},s_{\rho_{x_1}},s_{\rho_{x_2}}$ are pulled back along
$\pr_x$. They cut out the fibers in $[\O_{\P^2}(-3)/\mu_3]$ over the
coordinate lines in $[\P^2/\mu_3]$, and similarly for
$s_{\rho_{y_0}},s_{\rho_{y_1}},s_{\rho_{y_2}}$. Using Section
\ref{sec:CohomologyOfLineBundle}, these substacks give classes
$[L_0],[L'],[L'],$ respectively, and all of these are equal to $H.$
Thus we have
\begin{align*}
  D_{\rho_{x_0}}=D_{\rho_{x_1}}=D_{\rho_{x_2}}&=\pr_x^*(H)=:H_x\\
  D_{\rho_{y_0}}=D_{\rho_{y_1}}=D_{\rho_{y_2}}&=\pr_y^*(H)=:H_y.
\end{align*}
The sections $s_{\rho_{z_0}},s_{\rho_{z_1}},s_{\rho_{z_2}}$ are pulled
back along $\pr_z$. They cut out the coordinate planes in
$[\C^3/\mu_3]$. These are trivial in $H^2(X(\theta),\C)$, i.e.
\begin{align*}
  D_{\rho_{z_0}}=D_{\rho_{z_1}}=D_{\rho_{z_2}}&=0.
\end{align*}
The section $s_{\rho_a}$ is nonvanishing, so $D_{\rho_a}=0$. By the
same argument, $D_{\rho_{p_z}}=0$. Finally, $s_{\rho_{p_x}}$ is again
pulled back along $\pr_x$, and vanishes along the zero section of
$[\O_{\P^2}(-3)/\mu_3].$ Since $x_1^3p_x$ is a well-defined function
on $[\O_{\P^2}(-3)/\mu_3]$ (with coordinates as above), which vanishes
to order 3 along $L'=\{x_1=0\}$ and to order 1 along the zero
section. Hence the class of the zero section is $-3[L']=-3H,$ so
\begin{align*}
  D_{\rho_{p_x}}&=-3H_x\\
  D_{\rho_{p_y}}&=-3H_y.
\end{align*}

\section{Moduli spaces of sections of line bundles}\label{sec:ModuliSpaces}
\subsection{LG-quasimaps}
For each $\theta\in\Theta$, we define moduli spaces of
\emph{LG-quasimaps}. These were introduced in
\cite{FanJarvisRuan2015}, though in our examples the definitions
simplify significantly.
\begin{Def}\label{Def:LGQuasimapToX}
  A \emph{prestable genus zero $m$-marked LG-quasimap to $X(\theta)$}
  is a tuple $(C,u,\kappa),$ where
  \begin{enumerate}[(i)]
  \item $(C,b_1,\ldots,b_m)$ is an $m$-marked 3-stable curve,
  \item $u:C\to[V/(G\times\C^*_R)]=[X/\C^*_R]$ is a morphism of
    stacks, and\label{item:Quasimap}
  \item $\kappa:u^*L_{\widehat{t_R}}\to\omega_{C,\log}$ is an
    isomorphism of line bundles on $C$,\label{item:Isom}
  \end{enumerate}
  such that all marked points, nodes, and generic points of components
  map to $[X(\theta)/\C^*_R]=[V\sslash_{\vartheta}(G\times\C^*_R)]$. A
  \emph{prestable genus zero $m$-marked LG-quasimap to $Z(\theta)$} is
  a prestable genus zero $m$-marked LG-quasimap to $X(\theta)$ that
  factors through $[Z/\C^*_R]\into[X/\C^*_R]$.
\end{Def}
\begin{Def}
  Let $(C,u,\kappa)$ be a prestable genus zero $m$-marked LG-quasimap
  to $X(\theta)$. A point $P$ of $C$ for which
  $u(P)\in[V^{un}(\theta)/(G\times\C^*_R)]\subseteq[V/(G\times\C^*_R)]$
  is called a \emph{basepoint} of $u$.
\end{Def}
We will now reinterpret these definitions more algebraically. From
Section \ref{sec:GIT}, Definition
\ref{Def:LGQuasimapToX}\eqref{item:Quasimap} is the same as a
principal $(\C^*)^5$-bundle $\mathcal{P}$ on $C$ --- we will denote
the five corresponding line bundles by
$\mathcal{L}_x,\mathcal{L}_y,\mathcal{L}_z,\mathcal{L}_a,\mathcal{L}_R$,
and their degrees by $\beta_x,\beta_y,\beta_z,\beta_a,\beta_R$ --- and
a section $\sigma$ of
\begin{align*}
  \mathcal{E}:=\mathcal{P}\times_{(\C^*)^5}V=(&\mathcal{L}_x\otimes\mathcal{L}_a^*)\oplus\mathcal{L}_x\oplus\mathcal{L}_x\oplus(\mathcal{L}_y\otimes\mathcal{L}_a^*)\oplus\mathcal{L}_y\oplus\mathcal{L}_y\oplus(\mathcal{L}_z\otimes\mathcal{L}_a^*)\oplus\mathcal{L}_z\oplus\mathcal{L}_z\oplus\\
                                              &\oplus\mathcal{L}_a^{\otimes3}\oplus(\mathcal{L}_x^{\otimes-3}\otimes\mathcal{L}_R)\oplus(\mathcal{L}_y^{\otimes-3}\otimes\mathcal{L}_R)\oplus(\mathcal{L}_z^{\otimes-3}\otimes\mathcal{L}_R).
\end{align*}
Using \eqref{item:Isom} we may forget about
$\mathcal{L}_R$ altogether and replace the data
\eqref{item:Quasimap} and \eqref{item:Isom} with the data of the
line bundles
$\mathcal{L}_x,\mathcal{L}_y,\mathcal{L}_z,\mathcal{L}_a$ and a
section $\sigma$ of 
\begin{align*}
  \mathcal{E}=(&\mathcal{L}_x\otimes\mathcal{L}_a^*)\oplus\mathcal{L}_x\oplus\mathcal{L}_x\oplus(\mathcal{L}_y\otimes\mathcal{L}_a^*)\oplus\mathcal{L}_y\oplus\mathcal{L}_y\oplus(\mathcal{L}_z\otimes\mathcal{L}_a^*)\oplus\mathcal{L}_z\oplus\mathcal{L}_z\oplus\\
               &\oplus\mathcal{L}_a^{\otimes3}\oplus(\mathcal{L}_x^{\otimes-3}\otimes\omega_{C,\log})\oplus(\mathcal{L}_y^{\otimes-3}\otimes\omega_{C,\log})\oplus(\mathcal{L}_z^{\otimes-3}\otimes\omega_{C,\log}).
\end{align*}
We will write $\mathcal{L}_\rho:=u^*L_\rho$ for the summands of $\mathcal{E}$, i.e.
\begin{align*}
  \mathcal{E}&=\bigoplus_{\rho\in\mathbf{R}}\mathcal{L}_\rho,
\end{align*}
and
$$\sigma=(\sigma_{x_0},\sigma_{x_1},\sigma_{x_2},\sigma_{y_0},\sigma_{y_1},\sigma_{y_2},\sigma_{z_0},\sigma_{z_1},\sigma_{z_2},\sigma_a,\sigma_{p_x},\sigma_{p_y},\sigma_{p_z}).$$
Thus an LG-quasimap $(C,u,\kappa)$ to $X(\theta)$ is the same data
as a tuple $(C,\mathcal{L},\sigma),$ where $\mathcal{L}$ is
shorthand for
$(\mathcal{L}_x,\mathcal{L}_y,\mathcal{L}_z,\mathcal{L}_a).$
LG-quasimaps to $Z(\theta)$ are similarly reinterpreted, and we use
the notations $(C,u,\kappa)$ and $(C,\mathcal{L},\sigma)$
interchangeably.
\begin{Def}
  Let $(C,\mathcal{L},\sigma)$ be an LG-quasimap to $X(\theta)$ or
  $Z(\theta)$. The \emph{degree} of $(C,\mathcal{L},\sigma)$ is the
  tuple of rational numbers
  $\beta:=(\beta_x,\beta_y,\beta_z,\beta_a).$ (We do not need to
  include $\beta_R$ as it is necessarily equal to $-2+m$.)
\end{Def}
\begin{notation}\label{not:BetaRho}
  For an arbitrary character $\rho$ of $G\times\C^*_R,$ we define
  $\mathcal{L}_\rho:=u^*L_\rho=\mathcal{P}\times_{G\times\C^*_R}\C_\rho$
  and denote by $\beta_\rho$ the degree of $\mathcal{L}_\rho.$
\end{notation}
\begin{Def}\label{Def:EpsilonStability}
  For $\epsilon\in\Q_{>0}$, we say an $m$-marked, genus-zero
  LG-quasimap to $Z(\theta)$ or $X(\theta)$ is
  \emph{$\epsilon$-stable} if
  \begin{enumerate}
  \item The \emph{length} $\ell^\sigma(P)$ of $\sigma$ at each point
    $P$ is
    at most $1$, and
  \item $\omega_{C,\log}\otimes\mathcal{L}_\vartheta^{\epsilon}$ is
    ample.\label{item:Ample}
  \end{enumerate}
\end{Def}
For the general definition of length, see \cite{FanJarvisRuan2015}.
We describe it in the case $\theta=\theta^{xya}_z$, from which the
other cases are clear. The length is a sum over components of the
unstable locus $V^{\uns}(\theta)$. If
$(\sigma_{x_0},\sigma_{x_1},\sigma_{x_2})=(0,0,0)$ at a point
$P\in C,$ then the minimum order of vanishing of these sections at $P$
is the contribution to the $\ell^\sigma(P)$ from the component
$\{(x_0,x_1,x_2)=(0,0,0)\}$ of $V^{\uns}(\theta).$ We denote this
contribution by $\ell_x^\sigma(P)$. The contribution
$\ell_y^\sigma(P)$ is defined similarly. The contribution
$\ell_z^\sigma(P)$ from the component $\{p_z=0\}$ is even simpler ---
it is just the order of vanishing of $\sigma_{p_z}$ at $P$. Similarly
$\ell_a^\sigma(P)$ is the order of vanishing of $\sigma_a$ at
$P$. Finally, we define
$\ell^\sigma(P)=\ell^\sigma_x(P)+\ell^\sigma_y(P)+\ell^\sigma_z(P)+\ell^\sigma_a(P).$
\begin{Def}\label{Def:DegreeOfBasepoint}
  Let $(C,u,\kappa)$ be an LG-quasimap to $X(\theta)$ or $Z(\theta)$
  of degree $\beta$, and let $P$ be a basepoint of $u$. The
  \emph{degree
    $\beta(P)=(\beta_x(P),\beta_y(P),\beta_z(P),\beta_a(P))$ of the
    basepoint $P$} is defined, for $\theta=\theta_z^{xya},$ by
  \begin{align*}
    (\beta_x(P),\beta_y(P),\beta_z(P),\beta_a(P)):=(\ell_x^\sigma(P),\ell_y^\sigma(P),\frac{-1}{3}\ell_z^\sigma(P),\frac{1}{3}\ell_a^\sigma(P)).
  \end{align*}
\end{Def}
The reason for this definition is as follows. Restricting
$(C,u,\kappa)$ to $C\setminus P$ gives a section
$\sigma|_{C\setminus P}$ of $\mathcal{E}|_{C\setminus P}$. There is a
way (unique up to isomorphism) to extend
$(\mathcal{E}|_{C\setminus P},\sigma|_{C\setminus P})$ to
$(\mathcal{E}',\sigma'),$ where $\mathcal{E}'$ is a vector bundle on
$C$ and $\sigma'\in H^0(C,\mathcal{E}')$, such that
$\ell^{\sigma'}(P)=0.$ This bundle $\mathcal{E}'$ is associated to a
space of LG-quasimaps of degree $\beta-\beta(P).$ Therefore the degree
of $(C,u,\kappa)$ is equal to the degree ``over the generic points of
$C$'', plus the sum of the degrees of all basepoints.

\begin{rem}
  As in \cite{CiocanFontanineKim2014}, we may also define
  $\epsilon$-stability for $\epsilon$ equal to either of the symbols
  $0+$ and $\infty.$ A quasimap is $(0+)$-stable if it is
  $\epsilon$-stable for all $\epsilon$ sufficiently small, and
  $\infty$-stable if it is $\epsilon$-stable for all $\epsilon$
  sufficiently large.
\end{rem}
\begin{thm}[\cite{FanJarvisRuan2015}, Theorem 1.1.1]\label{thm:DMStack}
  For each $\theta\in\Theta$, $\epsilon\in[0+,\infty]$,
  $m\in\Z_{\ge0}$ and $\beta\in(\frac{1}{3}\Z)^4,$ there is a finite
  type, separated Deligne-Mumford stack
  $\LGQ_{0,m}^{\epsilon}(X(\theta),\beta)$ of families of
  $\epsilon$-stable genus zero $m$-marked LG-quasimaps to $X(\theta)$
  of degree $\beta$.
  
  There is also a finite type, separated, \emph{proper}
  Deligne-Mumford stack $\LGQ_{0,m}^{\epsilon}(Z(\theta),\beta)$ of
  families of $\epsilon$-stable genus zero $m$-marked LG-quasimaps to
  $Z(\theta)$ of degree $\beta$. 
\end{thm}
\begin{Def}[Graph spaces]
  We will often use the slightly modified moduli spaces
  $\LGQG_{0,m}^\epsilon(X(\theta),\beta)$ and
  $\LGQG_{0,m}^\epsilon(Z(\theta),\beta)$, called \emph{graph spaces}
  or spaces of \emph{LG-graph quasimaps}. They parametrize
  $\epsilon$-stable $m$-marked genus zero LG-quasimaps \emph{with a
    parametrized component}, i.e. a map $\tau:C\to\P^1$ of degree
  1. The stability condition is then imposed only on (the closure of)
  $C\setminus\hat{C}$, where $\hat{C}$ is the parametrized component.
\end{Def}

\subsection{Effective and extremal degrees}\label{sec:EffectiveDegrees}
\begin{Def}\label{Def:Effective}
  We say $(\beta,m)$ is \emph{$\theta$-effective} if
  $\LGQ_{0,m}^\epsilon(Z(\theta),\beta)$ is
  nonempty.
\end{Def}
\begin{prop}\label{prop:CriterionForEffectiveness}
  If $(\beta,m)$ is effective, we have:
  \begin{itemize}
  \item $\beta_a\ge0$,
  \item $\beta_x\ge0$ if $x$ is a superscript variable, and
  \item $\beta_x\le\frac{m-2}{3}$ if $x$ is a subscript variable.
  \end{itemize}
  (By symmetry these statements hold with $x$ replaced by $y$ or $z$.)
\end{prop}
\begin{proof}
  First, for all $\theta$ the condition in Definition
  \ref{Def:LGQuasimapToX}, together with the fact that
  $\{a=0\}\subseteq V^{\uns}(\theta),$ implies that $\mathcal{L}_a^3$
  has a global section that is nonvanishing at the generic point of
  each component of $C$. This implies $\beta_a\ge0.$

  If $x$ is a superscript variable, the fact that
  $\{(x_0,x_1,x_2)=(0,0,0)\}\subseteq V^{\uns}(\theta)$ shows that at
  least one of $\mathcal{L}_x\otimes\mathcal{L}_a^*$ and
  $\mathcal{L}_x$ has nonnegative degree. Since $\beta_a\ge0$ we have
  $\beta_x\ge0.$

  If $x$ is a subscript variable, since $\{p_x=0\}\subseteq V^{\uns}$,
  we have that $\mathcal{L}_x^{\otimes-3}\otimes\omega_{C,\log}$ has
  nonnegative degree, i.e. $\beta_x\le\frac{m-2}{3}$.
\end{proof}
\begin{cor}\label{rem:BetaThetaPositive}
  Let $\beta_\vartheta$ be as in Notation \ref{not:BetaRho}. Whenever
  $(\beta,m)$ is effective, we have $\beta_\vartheta\in\Z_{\ge0}.$
\end{cor}
\begin{Def}\label{Def:UnstableTuples}
  Even if the conditions of Proposition
  \ref{prop:CriterionForEffectiveness} are satisfied, we may have
  $-2+m+\epsilon\beta_\vartheta<0,$ in which case Condition
  \ref{item:Ample} of Definition \ref{Def:EpsilonStability} is never
  satisfied. We call such tuples $(\beta,m)$
  \emph{unstable}. Explicitly, $(\beta,m)$ is unstable when
  \begin{enumerate}
  \item $m=2$ and $\beta=\beta_0(\theta,2),$ or
  \item $m=1$ and $\beta_\vartheta>1/\epsilon,$ or
  \item $m=0$ and $\beta_\vartheta>2/\epsilon$. 
  \end{enumerate}
\end{Def}
\begin{rem}
  It is the existence of unstable tuples that allows explicit
  calculation in Section \ref{sec:CalculatingI}.
\end{rem}

\noindent\textbf{Extremal Degrees.} Let $\theta=\theta^{xya}_z,$ as this is the example we work out in
later sections. As we observed in Section \ref{sec:EffectiveDegrees},
if $\beta_x$, $\beta_y,$ or $\beta_a$ is negative, the moduli space is
empty. Also, if $\beta_z>\frac{m-2}{3},$ then the line bundle
$\mathcal{L}_z^{-3}\otimes\omega_{C,\log}$ has negative degree, so the
moduli space is empty for the same reason. Therefore we say that the
pair $((0,0,\frac{m-2}{3},0),m)$ is \emph{extremal}, and write
$\beta_0(\theta,m):=(0,0,\frac{m-2}{3},0)$. 
\begin{observation}
  We will often use the fact that
  $\beta_0(\theta,m_1+m_2)=\beta_0(\theta,m_1+1)+\beta_0(\theta,m_2+1).$
\end{observation}

\begin{rem}\label{rem:BetaThetaZero}
  It follows from Remark \ref{rem:BetaThetaPositive} that for
  $(\beta,m)$ effective, we have $\beta_\vartheta=0$ if and only if
  $(\beta,m)=(\beta_0(\theta,m),m)$ is extremal.
\end{rem}

\begin{Def}\label{Def:Contract}
  If $C$ is irreducible and the ``degree over the generic point''
  $\beta-\sum_P\beta(P)$ from Definition \ref{Def:DegreeOfBasepoint}
  is equal to $\beta_0(\theta,m),$ then we say $C$ is contracted by
  $u$. Similarly we can say an irreducible component $C'$ of $C$ is
  \emph{contracted}.
\end{Def}

The extremal degree $\beta_0(\theta,m)$ will play essentially the same
role for us as $\beta=0$ does in Gromov-Witten theory, with matters
complicated slightly by the fact that for $\theta\ne\theta^{xyza},$
$\beta_0(\theta,m)$ is a function of $m$.

\subsection{Connections to quasimaps and spin
  structures}\label{sec:Connections}
Given an $\epsilon$-stable LG-quasimap to $Z(\theta),$ we may extract
a quasimap to $Z(\theta)$ (in the sense of
\cite{CheongCiocanFontanineKim2015}) as follows.
\begin{Def}\label{Def:AssociatedQuasimap}
  By Remark \ref{rem:RActsTrivially}, $\C^*_R$ acts trivially on
  $\Crit(W).$ This implies that
  $$[Z/\C^*_R]\cong Z\times
  B\C^*_R$$
  has a projection map $\pr$ to $Z.$ For an LG-quasimap
  $(C,u,\kappa),$ the pair $(C,\pr\circ u)$ is a prestable quasimap to
  $Z(\theta)$, called the \emph{quasimap associated to
    $(C,u,\kappa)$}.
\end{Def}
\begin{rem}\label{rem:AssociatedQuasimapToX}
  As $\C^*_R$ acts nontrivially on $X(\theta),$ there is no way to
  extract a quasimap from a LG-quasimap $(C,u,\kappa)$ to
  $X(\theta)$, unless $u$ maps $C$ into the locus
  $X_R(\theta)\subseteq[X(\theta)/\C^*_R]$ of points whose isotropy
  group contains $\C^*_R.$ In this case we obtain a quasimap to
  $X_R^{\rig}(\theta),$ the rigidification of $X_R(\theta)$ by
  $\C^*_R$ (see \cite{AbramovichGraberVistoli2008}). From the
  $\C^*_R$-action on $X(\theta)$ we see that $X_R^{\rig}(\theta)$ is
  isomorphic to
  $$[((\P^2)^i\times B\mu_3^{3-i})/\mu_3]\subseteq[(\O_{\P^2}(-3)^i\times[\C^3/\mu_3]^{3-i})/\mu_3]$$
  for some $i$ depending on $\theta,$ where $\P^2$ denotes the zero
  section of $\O_{\P^2}(-3)$ and $B\mu_3$ denotes the origin in
  $[\C^3/\mu_3].$ For instance, for $\theta=\theta^{xya}_z$ there is a
  quasimap to $X(\theta)$ associated to $(C,\mathcal{L},\sigma)$
  exactly when
  $\sigma_{z_0}=\sigma_{z_1}=\sigma_{z_2}=\sigma_{p_x}=\sigma_{p_y}=0.$
  The space $X_R^{\rig}(\theta)$ will later allow us to reduce
  statements about LG-quasimaps to known facts about quasimaps.
\end{rem}

Notice that the quasimap associated to $(C,u,\kappa)$ captures
information about the superscript variables.  We may also extract
complementary data related to the subscript variables, as follows:
\begin{prop}
  If $x$ is in the subscript of $\theta$, then forgetting everything
  except for $C,$ $\mathcal{L}_x,$ and $\sigma_{p_x}$ gives maps from
  $\LGQ_{0,m}^\epsilon(Z(\theta),\beta)$ and
  $\LGQ_{0,m}^\epsilon(X(\theta),\beta)$ to a space
  $\mathcal{R}_{m,\epsilon(-3\beta_x+m-2)}^3$ of (prestable)
  \emph{spin structures}, see \cite{ChiodoZvonkine2009}.

  If there are several subscript variables, this gives maps to a
  product of spaces of spin structures, fibered over $\Mbar_{0,m}$.
\end{prop}
This follows from the following lemma, adapted from
\cite{RossRuan2015}.
\begin{lem}\label{lem:Concave}
  If $x$ is in the subscript of $\theta,$ then if $\mathcal{L}_x$ and
  $\mathcal{L}_x\otimes\mathcal{L}_a^*$ have nontrivial monodromy at
  each marked point, they have no global sections.
\end{lem}
\begin{proof}
  The proof of Lemma 1.5 of \cite{RossRuan2015} immediately generalizes to
  any line bundle on a nodal genus zero twisted curve such that a
  tensor power is a twist down of $\omega_{C,\log}$ by a effective
  divisor.
\end{proof}
\begin{rem}
  The assumption of nontrivial monodromy is very important, and we
  will later define numerical invariants to vanish when this
  assumption is not satisfied. (See Sections
  \ref{sec:CompactTypeStateSpace} and \ref{sec:VirtualCycle}.)
\end{rem}
\begin{rem}
  In the case $\theta=\theta_{xyz}^a$ one can straightforwardly mimic
  the entire argument of \cite{RossRuan2015}, which proves the analog
  of Theorem \ref{thm:MirrorTheorem} for a different class of moduli
  spaces coming from hypersurfaces in weighted projective
  spaces. Indeed for $\theta_{xyz}^a$, Sections \ref{sec:Localization}
  and \ref{sec:MirrorTheorems} are extremely simple, since the torus
  action of Section \ref{sec:TorusAction} is trivial.
\end{rem}
Consider the space $\LGQ_{0,m}^\epsilon(Z(\theta^{xyza}),\beta).$ Let
$(C,\mathcal{L},\sigma)$ be an $\epsilon$-stable genus-zero $m$-marked
LG-quasimap to $Z(\theta)$. As
$\sigma_{p_x}=\sigma_{p_y}=\sigma_{p_z}=0$ by the condition that
$\sigma$ land in $Z(\theta)$, we may rephrase the data of
$(C,\mathcal{L},\sigma)$ once again, and we arrive at exactly the data
of an $\epsilon$-stable quasimap (in the sense of
\cite{CiocanFontanineKim2014}) to $Z(\theta)$. That is, we have
\begin{align*}
  \LGQ_{0,m}^\epsilon(Z(\theta),\beta)=Q_{0,m}^\epsilon(Z(\theta),\beta),
\end{align*}
where the latter space is a moduli stack of quasimaps. (The $\beta$ on
the right must be reinterpreted slightly as a character of $G$.) This
isomorphism is, of course, the motivation for Definition
\ref{Def:LGQuasimapToX}. Note, however, that
$\LGQ_{0,m}^\epsilon(X(\theta),\beta)$ is \emph{not} isomorphic to
$Q_{0,m}^\epsilon(X(\theta),\beta).$

One may wonder why, in this setup, we express the orbifold
$\left[(\O_{\P^2}(-3))^3/\mu_3\right]$ as the complicated toric
variety $X(\theta^{xyza})=[\C^{13}\sslash(\C^*)^4]$, rather than the
more natural-seeming (and isomorphic) toric variety
$$[\C^{12}\sslash((\C^*)^3\times\mu_3)].$$ The reason is a slightly 
complicated one. In fact, we could have used either
presentation. However, we will later calculate an important generating
function $I^\theta(q,z)$, which is defined using the moduli space
$\LGQG_{0,1}^{0+}(Z(\theta),\beta).$ This space parametrizes
LG-quasimaps $(C,u,\kappa)$ to $Z(\theta),$ where $C\cong\P_{3,1}.$
One can easily see from the definitions in \cite{FanJarvisRuan2015}
that there is no change if $\P_{3,1}$ is replaced with $\P^1$; in
other words, the stack structure plays no role! The reason is that
part of the data of $u$ is a principal $\mu_3$-bundle on $C$, and a
principal $\mu_3$-bundle on $\P_{3,1}$ is trivial, since the orbifold
fundamental group of $\P_{3,1}$ is trivial.

This issue disappears in our setup, essentially because the line
bundle $\mathcal{L}_a$ may still be nontrivial for
$(C,\mathcal{L},\sigma)\in\LGQG_{0,1}^{0+}(Z(\theta),\beta).$ As a
result, $I^\theta(q,z)$ contains much less information when using the
``natural'' presentation of $X(\theta^{xyza})$ than it does when using
our presentation. This method of finding more informative
presentations of orbifolds is alluded to in
\cite{CiocanFontanineKim2013}, and is related to the notion of
\emph{$S$-extended $I$-functions} from
\cite{CoatesCortiIritaniTseng2015}.

\section{Evaluation maps, compact type state space, and invariants}\label{sec:EvaluationMapsInvariants}
\noindent\textbf{Evaluation maps.} The universal curve $\UQ_{0,m}^\epsilon(Z(\theta),\beta)$ over
$\LGQ_{0,m}^\epsilon(Z(\theta),\beta)$ has a universal map $u^{\rig}$
to $Z$, by Definition \ref{Def:AssociatedQuasimap}. By the condition
in Definition \ref{Def:LGQuasimapToX}, all marked points and
(relative) nodes of $\UQ_{0,m}^\epsilon(Z(\theta),\beta)$ map to
$Z(\theta)\subseteq Z.$ Thus by Section 4.4 of
\cite{AbramovichGraberVistoli2008}, there are \emph{evaluation maps}
$\ev_i:\LGQ_{0,m}^\epsilon(Z(\theta),\beta)\to\bar{I}Z(\theta),$ the
rigidified inertia stack of $Z(\theta)$, recording the image of the
$i$th marked point.  In \cite{FanJarvisRuan2015} there is a more
general (and more subtle) notion of evaluation map defined on
$\LGQ_{0,m}^\epsilon(X(\theta),\beta)$, and taking values in
$\bar{I}X(\theta)$. There are also evaluation maps on graph spaces:
$$ev_i:\LGQG_{0,m}^\epsilon(X(\theta),\beta)\to\bar{I}X(\theta)\times\P^1.$$

\subsection{Compact type state space}\label{sec:CompactTypeStateSpace}
Part of the setup of the gauged linear sigma model in
\cite{FanJarvisRuan2015} is a special graded vector space
$\mathscr{H}(\theta)$, with a pairing, called the \emph{state
  space}. It is defined via \emph{relative} Chen-Ruan cohomology
groups. In our case, we may work with a particularly simple subspace
$\H(\theta)$.
\begin{Def}
  A sector of $\bar{I}X(\theta)$ is called \emph{narrow}\footnote{Note
    that our definition of a narrow sector is different from that in
    \cite{FanJarvisRuan2015}; however, it is the correct one for this
    setting.} if for each subscript variable, the corresponding
  sector of $[\C^3/\mu_3]$ is compact, i.e. is isomorphic to $B\mu_3.$
  We denote the (open and closed) substack of narrow sectors by
  $\bar{I}X(\theta)^{\nar}\subseteq\bar{I}X(\theta)$. The cohomology
  classes of narrow sectors are a direct summand
  $H^*_{CR}(X(\theta))^{\nar}\subseteq H^*_{CR}(X(\theta)).$
\end{Def}
\begin{Def}\label{Def:AmbientNarrowStateSpace}
  The \emph{ambient narrow state space} $\H(\theta)$ is the image
  $\iota^*(H^*_{CR}(X(\theta))^{\nar})\subseteq H^*_{CR}(Z(\theta))$.
\end{Def}
\begin{rem}
  The Poincar\'e pairing on $H^*_{CR}(Z(\theta))$ descends to
  $\H(\theta).$
\end{rem}
$\H(\theta)$ inherits a grading from $H^*_{CR}(Z(\theta))$, but it is
not the correct one for our purposes. For example, for
$\theta=\theta_{xyz}^a$ this would result in a vector space
concentrated in degree zero, due to the fact that
$Z(\theta)\cong[\Spec\C/(\mu_3)^4].$ In fact, this issue also arises
when trying to define a graded pullback map on Chen-Ruan
cohomology. Instead, we define the grading as follows.

First, the ages of elements of $H^*_{CR}(Z(\theta))$ are calculated
from the normal bundle to the embedding of a $g$-twisted sector not
into $Z(\theta),$ but into the ambient space $X(\theta).$ (These are
equivalent in the case where $Z(\theta)$ intersects the $g$-fixed
locus in $X(\theta)$ transversely, which is the case for
$\theta=\theta^{xyza}.$ However, it is not true for
$\theta=\theta_{xyz}^a,$ where $Z(\theta)$ is contained in every
$g$-fixed locus.)

Second, we add to each degree the somewhat mysterious shift
$\dim(Z(\theta))-3.$ We do this in order to obtain the following:
\begin{thm}\label{thm:StateSpaceIsom}
  There is a graded isomorphism between $\H(\theta)$ and $\H(\theta')$
  for any $\theta,\theta'\in\Theta$. In particular, the graded
  subspaces have the dimensions:
  \begin{align}\label{eqn:DegreeTable}
    \begin{array}{c|c|c|c}
    \H^0(\theta)&\H^2(\theta)&\H^4(\theta)&\H^6(\theta)\\\hline
    1&4&4&1
  \end{array}
  \end{align}
\end{thm}
\begin{rem}
  Proving this is just a calculation; we call it a theorem because it
  is one of the parts of Ruan's original LG/CY conjecture
  (\cite{Ruan2011}). There is a method (\cite{ChiodoNagel2015}) for
  proving more general statements of this form, via careful use of an
  orbifold Thom isomorphism theorem, but it does not yet apply to this
  case.
\end{rem}
\begin{proof}
  We record the enlightening parts of the proof here.
  
\medskip

\noindent\textbf{Compact type state space of $Z(\theta^{xyza})$.}
Let $\theta=\theta^{xyza}$. The inclusion
$\iota:Z(\theta)\into X(\theta)$ factors as $\iota'\circ\iota''$,
where $\iota'$ is the inclusion $[(\P^2)^3/\mu_3]\into X(\theta).$
Since $\iota'$ is a homotopy equivalence, we have
$\Im(\iota^*)\cong Im((\iota'')^*)$.

The points of $(\P^2)^3$ with nontrivial stabilizer are those points
$(p_1,p_2,p_3),$ where $p_1,p_2,p_3\in\P^2$ are all fixed by
multiplication of the first coordinate by $\zeta$. In other words,
using the notation of Section \ref{sec:CohomologyOfLineBundle},
$p_i\in\tilde{L_0}$ or $p_i=\tilde{P_0}.$ From this we see that the
orbifold locus in $[(\P^2)^3/\mu_3]$ is a union of eight
components. The seven components isomorphic to $L_0\times L_0\times
P_0$, $L_0\times P_0\times P_0,$ and $P_0\times P_0\times P_0$ do not
intersect the critical locus, since $\tilde{P_0}\not\in E.$
The component $L_0\times L_0\times L_0$ intersects
$[E^3/\mu_3]$ in $3^3=27$ points (each isomorphic to
$B\mu_3$). That is,
\begin{prop}
  The rigidified inertia stack $\bar{I}Z(\theta)$ of $Z(\theta)$ is
  isomorphic to the disjoint union of $Z(\theta)$ and $27\cdot2=54$
  points.
\end{prop}
It is easy to check that a $\C$-basis of $\Im((\iota'')^*)$ consists
of the pullbacks of the classes
$$\{1,H_x,H_y,H_z,H_xH_y,H_xH_z,H_yH_z,H_xH_yH_z\}.$$
The images of the classes $1_\zeta,1_{\zeta^2}\in H^*_{CR}(X(\theta))$
are the sums of the 27 $\zeta$-twisted and $\zeta^2$-twisted classes,
respectively, in $H^*_{CR}(Z(\theta))$. The corresponding ages are 1
and 2, respectively. We thus obtain the table \eqref{eqn:DegreeTable}
for $\theta^{xyza}$ for $\H(\theta^{xyza}).$.
\begin{notation}
  We refer to e.g. $(\iota'')^*(H_yH_z)$ by the more-cumbersome
  notation $(1\otimes H_y\otimes H_z)_1,$ and denote
  $(\iota'')^*(1_\zeta)$ and $(\iota'')^*(1_{\zeta^2})$ by
  $(1\otimes1\otimes1)_\zeta$ and $(1\otimes1\otimes1)_{\zeta^2}$,
  respectively, since it will simplify our notation in the remaining
  cases.
\end{notation}

\noindent\textbf{Compact type state space of $Z(\theta_{xyz}^a)$.}
Let $\theta=\theta^{a}_{xyz}.$ Recall from Section
\ref{sec:Characters} that $X(\theta)\cong[\C^9/(\mu_3)^4],$ with
$Z(\theta)=[pt/(\mu_3)^4]$ the origin. As all fixed loci are
connected, the components of $\bar{I}X(\theta)$ and $\bar{I}Z(\theta)$
are both in bijection with $(\mu_3)^4.$ The pullback map $\iota^*$ is
surjective, and the narrow sectors correspond to elements of
$(\mu_3)^4$ that act trivially on $\C^9$. We can easily write down
these elements, and calculating their ages gives
\begin{align*}
  \begin{array}{c|c}
  \mbox{Element}&\mbox{Age}\\\hline
  (\zeta,\zeta,\zeta,1)&3\\
  (\zeta^2,\zeta,\zeta,1)&4\\
  (\zeta,\zeta^2,\zeta,1)&4\\
  (\zeta,\zeta,\zeta^2,1)&4\\
  (\zeta^2,\zeta^2,\zeta,1)&5\\
  (\zeta^2,\zeta,\zeta^2,1)&5\\
  (\zeta,\zeta^2,\zeta^2,1)&5\\
  (\zeta^2,\zeta^2,\zeta^2,1)&6\\
  (\zeta^2,\zeta^2,\zeta^2,\zeta)&5\\
  (\zeta,\zeta,\zeta,\zeta^2)&4
  \end{array}
\end{align*}
We denote the class associated to $(\zeta,\zeta,\zeta,1)$ by
$(1_\zeta\otimes1_\zeta\otimes1_\zeta)_1,$ and similarly for the first
eight rows of this list. The last two we denote
$(1_{\zeta^2}\otimes1_{\zeta^2}\otimes1_{\zeta^2})_{\zeta}$ and
$(1_{\zeta}\otimes1_{\zeta}\otimes1_{\zeta})_{\zeta^2}$, respectively.
The shifted degree of a class $\gamma$ is
$2\age(\gamma)-2\dim Z(\theta)-6=2\age(\gamma)-6$. Thus again we
obtain \eqref{eqn:DegreeTable}.

\medskip

\noindent\textbf{Compact type state space of $Z(\theta_{z}^{xya})$.}
Finally, we include the computation for $\H^*(\theta)$ where
$\theta=\theta_z^{xya},$ because this case will be worked out in
detail throughout the paper.

Recall that
\begin{align*}
  X(\theta)&\cong[((\O_{\P^2}(-3))^2\times[\C^3/\mu_3])/\mu_3]\\
  Z(\theta)&\cong[(E^2\times B\mu_3)/\mu_3]=[E^2/(\mu_3)^2].
\end{align*}
The elements
$(1,1),(1,\zeta),(1,\zeta^2),(\zeta,\zeta),(\zeta^2,\zeta^2)\in(\mu_3)^2$
do not give narrow sectors. We write the narrow sectors associated to
the other elementsof $(\mu_3)^2$:
\begin{table}[h]
  \centering
  \begin{tabular}{c|c}
    Group element&Sectors\\
\hline
    $(\zeta,1)$&$[E^2/\mu_3]$\\
    $(\zeta^2,1)$&$[E^2/\mu_3]$\\
    $(\zeta^2,\zeta)$&9 points\\
    $(\zeta,\zeta^2)$&9 points\\
  \end{tabular}
\end{table}

The image of the pullback is calculated in the same way as it was for
$\theta^{xyza}.$ In particular, we have a basis for $\H(\theta)$:
\begin{align*}
  \{&(1\otimes1\otimes1_\zeta)_1,(H_x\otimes1\otimes1_\zeta)_1,(1\otimes
  H_y\otimes1_\zeta)_1,(H_x\otimes H_y\otimes1_\zeta)_1,\\&(1\otimes1\otimes1_{\zeta^2})_1,(H_x\otimes1\otimes1_{\zeta^2})_1,(1\otimes
  H_y\otimes1_{\zeta^2})_1,(H_x\otimes H_y\otimes1_{\zeta^2})_1,\\&(1\otimes1\otimes1_{\zeta^2})_\zeta,(1\otimes1\otimes1_\zeta)_{\zeta^2}\}.
\end{align*}
Here $(1\otimes1\otimes1_{\zeta^2})_\zeta$ is the sum of the first set
of 9 points above, and $(H_x\otimes1\otimes1_\zeta)_1$ is the
pullback of the class $H_x$ on the $(\zeta,1)$-twisted sector of
$X(\theta),$ isomorphic to $[\O_{\P^2}(-3)^2/\mu_3].$ Calculating the
(properly shifted) degrees gives \eqref{eqn:DegreeTable} again.

It is straightforward to carry out the calculation for
$\theta_{yz}^{xa}$, with the same result. This proves the theorem.
\end{proof}
\begin{rem}
  This graded isomorphism holds for the larger state spaces
  $\mathscr{H}(\theta)$ defined in \cite{FanJarvisRuan2015} as well.
\end{rem}

\noindent\textbf{Explicit isomorphisms.} In Section
\ref{sec:RelatingIFunctions}, use an explicit identification of
$\H(\theta)$ with $\H(\theta')$, which we describe here. Let $\theta$
be such that $x$ is a superscript variable, and let $\theta'$ be the
character obtained by moving $x$ to the subscript. Elements of
$\H(\theta)$ are of the form $(1\otimes\alpha_y\otimes\gamma_z)_g$
or$(H_x\otimes\alpha_y\otimes\gamma_z)_g$ with $g\in\mu_3.$ We send:
\begin{align*}
  (1\otimes\alpha_y\otimes\gamma_z)_g&\mapsto(1_\zeta\otimes\alpha_y\otimes\gamma_z)_g\in\H(\theta')\\
  (H_x\otimes\alpha_y\otimes\gamma_z)_g&\mapsto(1_{\zeta^2}\otimes\alpha_y\otimes\gamma_z)_g\in\H(\theta').
\end{align*}
Repeating this process and its inverse gives explicit graded
isomorphisms between $H(\theta)$ and $\H(\theta')$ for any
$\theta,\theta'\in\Theta.$
\begin{rem}\label{rem:AnalogOfUntwistedSector}
  For each $\theta$, there is a special generator with degree
  zero. This is the element where $g=1\in\mu_3$ and all entries of the
  tensor are $1$ or $1_\zeta$, for $x$ a superscript or subscript
  variable respectively. We will abbreviate it by $1_\theta$.
\end{rem}

\subsection{Virtual class and invariants}\label{sec:VirtualCycle}
We define open and closed substacks:
\begin{align*}
  \LGQ_{0,m}^{\epsilon}(X(\theta),\beta)^{\nar}:&=\bigcap_{i=1}^m\ev_i^{-1}(\bar{I}X(\theta)^{\nar})\subseteq\LGQ_{0,m}^{\epsilon}(X(\theta),\beta)\\
  \LGQ_{0,m}^{\epsilon}(Z(\theta),\beta)^{\nar}:&=\bigcap_{i=1}^m\ev_i^{-1}(\bar{I}Z(\theta)^{\nar})\subseteq\LGQ_{0,m}^{\epsilon}(Z(\theta),\beta).
\end{align*}
\begin{thm}[\cite{FanJarvisRuan2015}]\label{thm:VirtualCycle}
  The complex $R^\bullet\pi_*\mathcal{E}$ is a perfect obstruction
  theory on $\LGQ_{0,m}^{\epsilon}(X(\theta),\beta)^{\nar}.$ It
  induces (via \emph{cosection localization}, see
  \cite{KiemLi2013,FanJarvisRuan2015}) a virtual fundamental class
  $$[\LGQ_{0,m}^\epsilon(Z(\theta),\beta)^{\nar}]^{\vir}\in
  H_*(\LGQ_{0,m}^\epsilon(Z(\theta),\beta)^{\nar},\C).$$
  There is similarly a virtual fundamental class on each graph space
  $\LGQG_{0,m}^\epsilon(Z(\theta),\beta)^{\nar}$.
\end{thm}
By a general fact about cosection localization,
\begin{align*}
  \iota_*[\LGQ_{0,m}^\epsilon(Z(\theta),\beta)^{\nar}]^{\vir}=[\LGQ_{0,m}^\epsilon(X(\theta),\beta)^{\nar}]^{\vir},
\end{align*}
where the latter is the virtual fundamental induced by the perfect
obstruction theory $R^\bullet\pi_*\mathcal{E}.$

Using this, we may define \emph{LG-quasimap invariants}:
\begin{Def}
  Let $\alpha_1,\ldots,\alpha_m\in\H(\theta).$ Then we define
  \begin{align*}
    \langle\alpha_1\psi^{a_1},\ldots,\alpha_m\psi^{a_m}\rangle_{0,m,\beta}^{\epsilon,\theta}:&=\int_{[\LGQ_{0,m}^\epsilon(Z(\theta),\beta)^{\nar}]^{\vir}}\prod_{i=1}^m(\psi_i^{a_i}\ev_i^*\alpha_i)\\
    \langle\alpha_1\psi^{a_1},\ldots,\alpha_m\psi^{a_m}\rangle_{0,m,\beta}^{\epsilon,\theta,Gr}:&=\int_{[\LGQG_{0,m}^\epsilon(Z(\theta),\beta)^{\nar}]^{\vir}}\prod_{i=1}^m(\psi_i^{a_i}\ev_i^*\alpha_i).
  \end{align*}
\end{Def}
\begin{rem}
  Theorem \ref{thm:VirtualCycle} also states that the \emph{unshifted
    virtual dimension} of
  $\LGQ_{0,m}^\epsilon(Z(\theta),\beta)^{\nar}$ is $m$. This implies
  that if $\alpha_1,\ldots,\alpha_m$ have degrees $k_1,\ldots,k_m,$
  then
  $\langle\alpha_1\psi^{a_1},\ldots,\alpha_m\psi^{a_m}\rangle_{0,m,\beta}^{\epsilon,\theta}$
  vanishes unless $\sum_i(k_i+a_i)=m.$\footnote{The term
    \emph{unshifted virtual dimension} is nonstandard, and we define
    it by this property. In the statement in \cite{FanJarvisRuan2015},
    the condition instead would read
    $\sum_i(k_i+a_i-\age_i)=m-\sum\age_i$, where $\age_i$ is the age
    of the twisted sector in which $\ev_i$ lands.} Similarly
  $\LGQG_{0,m}^\epsilon(Z(\theta),\beta)^{\nar}$ has unshifted virtual
  dimension $m+3.$
\end{rem}
\begin{rem}\label{rem:BroadVanish}
  We may define these invariants for arbitrary
  $\alpha\in H^*_{CR}(Z(\theta)),$ with the convention that they
  vanish if $\alpha_i$ is supported on
  $\bar{I}X(\theta)\setminus\bar{I}X(\theta)^{\nar}$ for some $i$. (We
  refer to such $\alpha_i$ as \emph{broad}, as we used \emph{narrow}
  to refer to elements of $\H(\theta).$)
\end{rem}

\section{Equivariant localization on
  $\LGQG_{0,m}^\epsilon(X(\theta),\beta)$ and
  $\LGQ_{0,m}^\epsilon(X(\theta),\beta)$
  }\label{sec:Localization}
  In this section we define two natural group actions on the moduli
  spaces; a $\C^*$-action on
  $\LGQG_{0,m}^\epsilon(X(\theta),\beta)$ induced by the $\C^*$-action
  on $P^1,$ and a torus action on
  $\LGQ_{0,m}^\epsilon(X(\theta),\beta)$ induced by the torus action
  on the toric variety $X(\theta).$

\subsection{The $\C^*$-action on
  $\LGQG_{0,m}^\epsilon(X(\theta),\beta)$}\label{sec:CStarAction}
The $\C^*$-action we define, as well as the graph spaces themselves,
are essentially combinatorial tools for analyzing the generating
functions defined in Section \ref{sec:GeneratingFunctions}. Recall
that an LG-graph quasimap to $X(\theta)$ is a tuple
$(C,\mathcal{L},\sigma,\tau)$, where $\tau:C\to\P^1$ is a degree-one
(nonrepresentable) map. For $\lambda\in\C^*$, let
$\lambda\cdot[s:t]=[\lambda s:t]$ denote the standard (left) action on
$\P^1.$ Then
\begin{align*}
  (C,\mathcal{L},\sigma,\tau)\mapsto(C,\mathcal{L},\sigma,\tau\circ\lambda)
\end{align*}
is a (right) $\C^*$-action on $\LGQG_{0,m}^\epsilon(X(\theta),\beta)$.

\medskip

\noindent\textbf{$\C^*$-fixed locus and normal bundles.} An LG-graph quasimap $(C,\mathcal{L},\sigma,\tau)$ to $X(\theta)$ is
$\C^*$-fixed if for each $\lambda\in\C^*$ there exists an automorphism
$\phi$ of $C$ commuting with $\sigma$ such that
$\phi\circ\tau=\tau\circ\lambda.$ Alternatively, let
$\hat{C}^\circ:=C\setminus\{\tau^{-1}(0),\tau^{-1}(\infty)\}$ and
$\hat{C}:=\bar{\hat{C}^\circ}\subseteq C$. Then
$(C,\mathcal{L},\sigma,\tau)$ is $\C^*$-fixed if
\begin{enumerate}
\item $\hat{C}^\circ$ contains no marked points, nodes, or basepoints
  of $\sigma$, and
\item $\hat{C}$ is contracted by $u$.
\end{enumerate}

\begin{notation}[From \cite{RossRuan2015}]
  We denote by $\hat{C}$ the closure of $\hat{C}^\circ,$ and we write
  $C_0$ and $C_\infty$ for $\tau^{-1}(0)$ and $\tau^{-1}(\infty)$,
  respectively.  We write $\bullet:=C_0\cap\hat{C}$ and
  $\check\bullet:=C_\infty\cap\hat{C}$. The point $\bullet$ may be a
  smooth (possibly orbifold) point (in which case $C_0$ is a single
  point), or it may be a node (in which case $C_0$ is a nodal curve).
\end{notation}

\begin{prop}\label{prop:PartitionFromFixedQuasimap}
  A $\C^*$-fixed $m$-marked LG-quasimap $(C,\mathcal{L},\sigma,\tau)$
  to $X(\theta)$ of degree $\beta$ defines a partition
  $B_0\sqcup B_\infty$ of $\{1,\ldots,m\}$ and a partition of tuples
  $\beta^0+\beta^\infty=\beta$, such that $(\beta^0,\abs{B_0}+1)$ and
  $(\beta^\infty,\abs{B_\infty}+1)$ are $\theta$-effective or
  unstable.
\end{prop}
\begin{proof}
  The fact that $\tau(b_i)$ is either 0 or $\infty$ for each $i$
  defines a partition $\{1,\ldots,m\}=B_0\sqcup B_\infty$.

  If $\bullet$ is a node, then
  $(C_0,\mathcal{L}|_{C_0},\sigma|_{C_0})$ is an $\epsilon$-stable
  LG-quasimap to $X(\theta).$ (Here $C_0$ has the extra marking
  $\bullet$.) We denote by $\beta^0$ its degree, and similarly
  $\beta^\infty$. If there is no node at $\bullet$
  (resp. $\check\bullet$), then we define $\beta^0$ to be the degree
  of the basepoint at $\bullet$ (resp $\check\bullet$, see Definition
  \ref{Def:DegreeOfBasepoint}).

  If there are nodes at $\bullet$ and $\check\bullet,$ we can check
  that $\hat\beta=0,$ so $\beta_0+\beta^\infty=\beta.$ If one or both
  of $\bullet$ and $\check\bullet$ is not a node,
  $\beta_0+\beta^\infty$ by the definition of the degree of a
  basepoint.
\end{proof}
\begin{rem}
  The tuple $(B_0,\beta^0)$ is unstable exactly when $\bullet$ is a
  smooth point.
\end{rem}
Proposition \ref{prop:PartitionFromFixedQuasimap} allows us to define
open and closed substacks $F_{B_0,\beta^0}^{B_\infty,\beta^\infty}$ of
$\LGQG_{0,m}^\epsilon(X(\theta),\beta)^{\C^*}$, consisting of those
LG-quasimaps that induce the partition $B_0\sqcup B_\infty$ of
$\{1,\ldots,m\}$ and the partition $\beta^0+\beta^\infty$ of $\beta.$
(We refer to these as ``components'' of
$\LGQG_{0,m}^\epsilon(X(\theta),\beta)^{\C^*}$, though they are almost
never connected.) If $(\beta^0,\abs{B_0}+1)$ and
$(\beta^\infty,\abs{B_\infty}+1)$ are effective rather than unstable, then
\begin{align}\label{eqn:DecomposeFixedLocus}
  F_{B_0,\beta^0}^{B_\infty,\beta^\infty}\cong\LGQ_{0,\abs{B_0}+\bullet}^\epsilon(X(\theta),\beta^0)\times_{\bar{I}Z(\theta)}\LGQ_{0,\abs{B_\infty}+\check\bullet}^\epsilon(X(\theta),\beta^\infty),
\end{align}
fibered over the evaluation maps $\ev_{\bullet}$ and
$\ev_{\check\bullet}.$\footnote{As in \cite{CiocanFontanineKim2014},
  one may define the factors by convention so that this remains true
  for $(\beta^0,\abs{B_0}+1)$ and
$(\beta^\infty,\abs{B_\infty}+1)$ unstable.}

When $\bullet$ and $\check\bullet$ are both nodes, we may calculate
the $\C^*$-equivariant Euler class of the virtual normal bundle to
$F_{B_0,\beta^0}^{B_\infty,\beta^\infty}.$ One may check that the
$\C^*$-moving infinitesimal deformations come from smoothing the nodes
and deforming the map $\tau$ (equivalently, moving the points
$\tau(\bullet)$ and $\tau(\check\bullet)$). By a classical
computation, smoothing the nodes contributes factors
$\hbar-\psi_\bullet$ and $-\hbar-\psi_{\check\bullet},$ pulled back to
the fiber product \eqref{eqn:DecomposeFixedLocus}. (These are the
weights of the deformation spaces
$T_\bullet C_0\otimes T_\bullet\hat{C}$ and
$T_{\check\bullet}C_\infty\otimes T_{\check\bullet}\hat{C}$,
respectively. Here we use the natural identification
$H^*_{\C^*}(\Spec\C,\C)\cong\C[\hbar]$.) Deforming $\tau$ gives
factors are $\hbar$ and $-\hbar$, the weights of the tangent spaces
$T_0\P^1$ and $T_\infty\P^1.$ Thus the $\C^*$-equivariant Euler class
of the virtual normal bundle is
$(-\hbar^2)(\hbar-\psi_\bullet)(-\hbar-\psi_{\check\bullet}).$

\begin{Def}\label{Def:FBetaPrime}
  We define here a special component
  $F_\beta':=F_{\star,\beta_0(2)}^{m,\beta}$
  of $\LGQG_{0,m+\star}^\epsilon(X(\theta),\beta)^{\C^*}$,
  which we will use in Sections \ref{sec:UnstableConventions} and
  \ref{sec:CalculatingI}.
\end{Def}

\begin{Def}\label{Def:FBeta}
  We may restrict all constructions in this section to the space
  $\LGQG_{0,m}^{\epsilon,\theta}(Z(\theta),\beta)$. Denote by
  $F_\beta$ the analog of the $F_\beta';$ then there is a fibered
  square
  \begin{center}
    \begin{tikzpicture}
      \matrix(m)[matrix of math nodes,row sep=2em,column
      sep=3em,minimum width=2em] {
        F_\beta&F_\beta'\\
        \LGQG_{0,m+1}^{\epsilon}(Z(\theta),\beta)&\LGQG_{0,m+1}^{\epsilon}(X(\theta),\beta)\\
      };
      \path[-stealth] (m-1-1) edge (m-2-1);
      \path[-stealth] (m-1-1) edge (m-1-2);
      \path[-stealth] (m-2-1) edge (m-2-2);
      \path[-stealth] (m-1-2) edge (m-2-2);
    \end{tikzpicture}
  \end{center}
\end{Def}

Finally, we define special classes in $H^2_{\C^*}(\P^1,\C).$ Let $p_0$
and $p_\infty$ denote the pushforwards of
$1\in H^*_{\C^*}(\Spec\C,\C)\cong\C[[\hbar]]$ along the equivariant
inclusions $0\into\P^1$ and $\infty\into\P^1,$ respectively. Then
(choosing a $\C^*$-action on $\O_{\P^1}(1)$) we have
  \begin{align*}
    p_0|_0&=\hbar&p_\infty|_\infty&=-\hbar&p_0|_\infty&=p_\infty|_0=0.
  \end{align*}

\subsection{The torus action on
  $\LGQ_{0,m}^\epsilon(X(\theta),\beta)$}\label{sec:TorusAction}
Torus actions on spaces of stable maps were used by Kontsevich to
carry out explicit computations of Gromov-Witten invariants of toric
varieties. They reduce the complicated geometry of curves in toric
varieties to combinatorics of fixed point sets, which are finite and
explicit. We will use the torus actions on spaces of LG-quasimaps to
obtain a recursive structure, leading to the proof of Theorem
\ref{thm:MirrorTheorem}. In fact, to our knowledge all of the many
such ``mirror'' theorems in Gromov-Witten theory use torus-fixed-point
localization.

For clarity, in this section we take $\theta=\theta_z^{xya}$ unless
stated otherwise. For everything we do, the appropriate changes to
make for the other characters will be clear.

There is a natural $T=(\C^*)^{13}$ action on $V$ by scaling the
coordinates. As all group actions on $V$ that we have discussed are by
scaling coordinates, they all commute. Thus we obtain $T$-actions on
$X(\theta)$ and $[X(\theta)/\C^*_R)]$ for each $\theta\in\Theta.$ The
latter induces a $T$-action on $\LGQ_{0,m}^\epsilon(X(\theta),\beta),$
and the various bundles and maps we consider have natural
$T$-equivariant lifts. For example, $\psi_i$ and $\ev_i$ have natural
equivariant lifts since since they are defined via the geometry of
maps to $[X(\theta)/\C^*_R]$. Similarly
$\mathcal{E}=\mathcal{P}\times_{G\times\C^*_R}V$ has a natural lift
induced by the action on $V$.

\medskip

\noindent\textbf{$T$-fixed locus.} 
By a classical argument of Gromov-Witten theory, $T$-fixed
LG-quasimaps to $X(\theta)$ are those that send $C$ into the closure
of 1-dimensional $T$-orbits in $[X/\C^*_R]$, and send all nodes,
markings, and ramification points of $(C,u,\kappa)$ to the $T$-fixed
locus of $[X/\C^*_R]$.

We check that the $T$-fixed locus of $[X/\C^*_R]$ is where:
\begin{itemize}
\item $p_x=p_y=0$,
\item $z_0=z_1=z_2=0$,
\item at most one of $x_0$, $x_1,$ and $x_2$ is nonzero, and
\item at most one of $y_0$, $y_1,$ and $y_2$ is nonzero.
\end{itemize}
These are exactly the coordinate points of
$X_R(\theta)\cong[((\P^2)^2\times B\mu_3)/\mu_3]\times B\C^*_R.$

Similarly, the 1-dimensional $T$-orbits
of $[X/\C^*_R]$
(with proper closure) are the coordinate lines in $X_R(\theta)$.

\begin{cor}\label{cor:TFixedAssociatedQuasimap}
  A $T$-fixed LG-quasimap $(C,u,\kappa)$ to $X(\theta)$ has an
  associated $T$-fixed quasimap $u^{\rig}$ to $X(\theta).$
\end{cor}

\begin{cor}
  The $T$-fixed locus in $\LGQ_{0,m}^\epsilon(X(\theta),\beta)$ is
  proper.
\end{cor}

As a result of the last fact, we can very closely mimic the
$T$-localization arguments for quasimaps in
\cite{CiocanFontanineKim2014,CheongCiocanFontanineKim2015}.

\begin{Def}
  Write $K\cong\C(\lambda_1,\ldots,\lambda_{13})$ for the \emph{localized} $T$-equivariant cohomology of
  a point.
\end{Def}
\begin{Def}
  Consider a 1-dimensional $T$-orbit $X_{\mu,\nu}$ in $X(\theta)$
  between $T$-fixed points $\mu$ and $\nu.$ (If such an $X_{\mu,\nu}$
  exists we say $\mu$ and $\nu$ are \emph{$T$-adjacent}.) We define
  the \emph{tangent weight} $w(\mu,\nu)$ to be
  $c_1(T_\mu X_{\mu,\nu})\in H^2_T(\mu,\C)\cong K.$
\end{Def}

\begin{rem}
  Everything in this section also applies to the graph space
  $\LGQG_{0,m}^\epsilon(X(\theta),\beta).$
\end{rem}

\section{Generating functions for genus zero LG-quasimap
  invariants}\label{sec:GeneratingFunctions}
Sections \ref{sec:GeneratingFunctions}, \ref{sec:MirrorTheorems} and
\ref{sec:CalculatingI} are based on Sections 5 and 7 of
\cite{CiocanFontanineKim2014} and Section 5 of
\cite{CheongCiocanFontanineKim2015}, respectively, with minor but
necessary modifications at each step. (The techniques in
\cite{CiocanFontanineKim2014} and \cite{CheongCiocanFontanineKim2015}
follow those of Givental (\cite{Givental1998}).) We define and compare
generating functions $J^{\epsilon,\theta}(t,q,\hbar)$,
$S^{\epsilon,\theta}(t,q,\hbar)$, and $P^{\epsilon,\theta}(t,q,\hbar)$,
encoding LG-quasimap invariants and LG-graph quasimap invariants of
$Z(\theta)$. (Note that the space
$\LGQ_{0,m}^\epsilon(X(\theta),\beta)$ and the $T$-action defined in
the last section do not appear in this section.) We continue to work
with $\theta=\theta_z^{xya}.$

\subsection{Double brackets}
From now on, we fix a basis $\{\gamma_j\}$ for $\H(\theta)$. Let
$\{\gamma^j\}$ be a dual basis with respect to the Poincar\'e pairing
on the \emph{nonrigidified} inertia stack $IZ(\theta)$.\footnote{Using
  $IZ(\theta)$ instead of $\bar{I}Z(\theta)$ will make our notation
  much simpler. This is discussed in Section 3.1 of
  \cite{CheongCiocanFontanineKim2015}.} Let
$t=\sum_jt_j\gamma_j\in\H(\theta)$. For
$\alpha_1,\ldots,\alpha_k\in\H(\theta),$ and
$a_1,\ldots,a_k\in\Z_{\ge0}$, we define the \emph{double bracket}
(compare with \cite{CiocanFontanineKim2014,RossRuan2015}):
\begin{align}
  \langle\langle\alpha_1\psi_1^{a_1},\ldots,\alpha_k\psi_k^{a_k}\rangle\rangle^{\epsilon,\theta}_{0,k}:&=\sum_{\beta,m}\frac{q^\beta
                                                                                                         q_z^{\frac{2-(k+m)}{3}}}{m!}\langle\alpha_1\psi_1^{a_1},\ldots,\alpha_k\psi_k^{a_k},t,\ldots,t\rangle_{0,k+m,\beta}^{\epsilon,\theta}\label{eqn:DoubleBracket}\\
                                                                                                       &=\sum_{\beta,m}\frac{q^{\beta-\beta_0(\theta,k+m)}}{m!}\langle\alpha_1\psi_1^{a_1},\ldots,\alpha_k\psi_k^{a_k},t,\ldots,t\rangle_{0,k+m,\beta}^{\epsilon,\theta}.\nonumber
\end{align}
Here $m\ge0$ and $\beta$ runs over degrees with $(\beta,m)$
$\theta$-effective.  The shifting factor $q_z^{\frac{2-(k+m)}{3}}$,
which does not appear in \cite{CiocanFontanineKim2014}, makes the
double bracket an element of $\C[[q_x,q_y,q_z^{-1},q_a]]$ rather
than $\C[[q_x,q_y,q_a]]((q_z^{-1}))$. We also define \emph{graph
  space double brackets} by replacing
$\langle\cdot\rangle_{0,k+m,\beta}^{\epsilon,\theta}$ in
\eqref{eqn:DoubleBracket} with
$\langle\cdot\rangle_{0,k+m,\beta}^{\epsilon,\theta,Gr}$.
\begin{notation}
  We write $\C[[q]]$ as shorthand for
  $\C[[q_x,q_y,q_z^{-1},q_a]]$. (Analogously for
  $\theta\ne\theta_z^{xya}$.)
\end{notation}

\subsection{Conventions for unstable
  tuples}\label{sec:UnstableConventions}
For small $k$, some terms of \eqref{eqn:DoubleBracket} correspond to
unstable tuples $(\beta,k+m)$ (recall Definition
\ref{Def:UnstableTuples}). In the following sections, setting those
terms to zero would not give the correct relations between generating
functions. To fix this, we now \emph{define} certain invariants
corresponding to unstable tuples.

First, we motivate these conventions. We apply $\C^*$-localization to
the graph space invariant
\begin{align}
\langle\alpha_1\psi_1^{a_1},\ldots,\alpha_m\psi_m^{a_m},\alpha_{\star}\otimes
p_\infty\rangle_{0,m+\star,\beta}^{\epsilon,\theta,Gr}=\int_{[\LGQG_{0,m}^{\epsilon}(Z(\theta),\beta)]^{\vir}}\prod_i\ev_i^*(\alpha_i)\psi_i^{a_i}\cup\ev_\star^*(\alpha_{\star}\otimes
p_\infty).\label{eqn:ConventionInvariant}
\end{align}

The result is a sum over the fixed loci
$F_{m_\infty,\beta_\infty}^{m_0,\beta_0}$. Consider the term
corresponding to the locus
$F_\beta=F_{\star,\beta_0(\theta,2)}^{m,\beta}.$ The tuple
$(\beta_0(\theta,2),2)$ is unstable, which implies that
$\check\bullet$ is a smooth point with the marking $\star$. Thus by
the computation in Section \ref{sec:CStarAction}, if the tuple
$(m,\beta)$ is stable, the normal bundle to $F_\beta$ is
$(-\hbar^2)(\hbar-\psi_\bullet),$ under the identification of $F_\beta$ with
$\LGQ_{0,m+\{\bullet\}}^\epsilon(Z(\theta),\beta).$ Also,
$\ev_\star^*(p_\infty)$ restricts on this locus to $-\hbar$ and
$\ev_\star$ is identified with $\ev_\bullet$. It follows that
\eqref{eqn:ConventionInvariant} can be written as
\begin{align*}
  \int_{[\LGQ_{0,m+\{\bullet\}}^\epsilon(Z(\theta),\beta)]^{\vir}}\prod_i\ev_i^*(\alpha_i)\psi_i^{a_i}\cup\frac{\ev_\star^*(\alpha_\star)}{\hbar(\hbar-\psi_\bullet)}=\langle\alpha_1\psi_1^{a_1},\ldots,\alpha_m\psi_m^{a_m},\frac{\alpha_{\star}}{\hbar(\hbar-\psi_\bullet)}\rangle_{0,m+\bullet,\beta}^{\epsilon,\theta}.
\end{align*}
This relation allows us to define invariants for $(\beta,m)$ unstable,
in the case where one entry of the bracket is of the form
$\frac{\alpha}{\hbar(\hbar-\psi_i)}$. That is, we set
$\langle\alpha_1\psi_1^{a_1},\ldots,\alpha_m\psi_m^{a_m},\frac{\alpha}{\hbar(\hbar-\psi_{m+1})}\rangle_{0,m+1,\beta}^{\epsilon,\theta}$
to be the contribution of $F_\beta$ to the equivariant integral
$$\langle\alpha_1\psi_1^{a_1},\ldots,\alpha_m\psi_m^{a_m},\alpha
\otimes p_\infty\rangle_{0,m+\star,\beta}^{\epsilon,\theta,Gr}.$$
\begin{rem}
  When used in LG-quasimap invariants (rather than LG-graph quasimap
  invariants), we may treat $\hbar$ as a formal variable, rather than a
  $\C^*$-equivariant class on $\P^1.$
\end{rem}

\subsection{The function $J^{\epsilon,\theta}(t,q,\hbar)$}\label{sec:JFunction}
Using the conventions in the last section, we define
\begin{align*}
  J^{\epsilon,\theta}(t,q,\hbar):=\sum_j\gamma_j\langle\langle\frac{\gamma^j}{\hbar(\hbar-\psi_1)}\rangle\rangle_{0,1}^{\epsilon,\theta}\in\H(\theta)[[q]]((\hbar^{-1})).
\end{align*}
(This should be thought of formally as a function
$\H(\theta)\to\H(\theta)[[q]]((\hbar^{-1})),$ without worrying about
convergence.) The unstable tuples contributing to
$J^{\epsilon,\theta}(t,q,\hbar)$ are:
\begin{align}\label{eqn:UnstableTermsOfJ}
  (\beta,m+1)&=(\beta_0(\theta,1),1)&(\beta,m+1)&=(\beta_0(\theta,2),2)\\
  (\beta,m+1)&=(\beta,1)\mbox{ with $\beta_\vartheta<1/\epsilon$}&&\nonumber
\end{align}

Calculating the terms coming from the tuples $(\beta,1)$ with
$\beta_\vartheta<1/\epsilon$ (in the case $\epsilon=0+$) is the
subject of Section \ref{sec:CalculatingI}. We compute the other two
terms here.

\medskip

\noindent\textbf{The term $(\beta_0(\theta,1),1)$.} This term is
defined, according to Section \ref{sec:UnstableConventions}, as the
$F_\beta$-contribution to the sum
\begin{align}
\sum_j\gamma_j\langle\gamma^j\otimes p_\infty\rangle_{0,\star,\beta_0(\theta,1)}^{\epsilon,\theta,Gr}.\label{eqn:UnstableJ1}
\end{align}

\noindent\textbf{Claim.} $\LGQG_{0,\star}^\epsilon(Z(\theta),\beta_0(\theta,1))$ parametrizes
the data:
\begin{itemize}
\item A parametrized curve $C\xrightarrow{\tau}{}\P^1$, with a marked
  orbifold point $\star$, and
\item A constant map $C\to[E^2/\mu_3]$ without basepoints, and with
  trivial monodromy at $\star$.
\end{itemize}
\begin{proof}
  The line bundles $\mathcal{L}_x,$ $\mathcal{L}_y,$ $\mathcal{L}_a$,
  and $\mathcal{L}_{\rho_{p_z}}$ have degree zero, and thus are
  trivial. (We may see from Proposition \ref{prop:DivisorLineBundle}
  that line bundles on $\P_{3,1}$ have trivial monodromy at $\star$.)
  Since $u$ lands in $[Z(\theta)/\C^*_R],$ the sections
  $\sigma_{z_0},\sigma_{z_1},\sigma_{z_2},\sigma_{p_x},\sigma_{p_y}$
  are all zero. 
  Thus up to isomorphism, $(C,\mathcal{L},\sigma)$ carries only the
  data of the parametrized marked curve $C$, the sections
  $\sigma_{x_i}$ and $\sigma_{y_i}$, and the line bundle
  $\mathcal{L}_z.$

  As $\mathcal{L}_{\rho_{p_z}}$ is trivial, we have
  $\mathcal{L}_z^{\otimes3}\cong\omega_{C,\log}$.  However, there is a
  unique such bundle up to isomorphism, with monodromy 2/3 at
  $\star$. It has automorphism group $\mu_3$, acting by multiplication
  on fibers, which commutes with
  $\kappa:\mathcal{L}_z^{\otimes3}\to\omega_{C,\log}.$

  The sections $\sigma_{x_i}$ and $\sigma_{y_i}$ define a map
  $C\to[E^2/\mu_3].$ It has trivial monodromy as $\mathcal{L}_a$ is
  trivial, and has no basepoints since $\mathcal{L}_x$ and
  $\mathcal{L}_y$ are trivial.
\end{proof}

$F_\beta$ is the locus where $\tau(\star)=\infty,$ so it is isomorphic
to $[E^2/\mu_3]\times B\mu_3.$ (The $B\mu_3$ comes from the
automorphisms of $\mathcal{L}_z$.) We see that $F_\beta$ is a twisted
sector of $IZ(\theta)$.

The virtual fundamental class is $[F_\beta]^{\vir}=[F_\beta]$, and
$\ev_\star$ is the $\mu_3$-rigidification map to a sector
$[E^2/\mu_3]\subseteq\bar{I}(Z(\theta))$. The class $p_\infty$
restricts to $-\hbar$ on $F_\beta,$ and the normal bundle to
$F_\beta\into\LGQG_{0,\star}^\epsilon(Z(\theta),\beta_0(\theta,1))$
comes from moving the image of $\star$ on $\P^1$, and has Euler class
$-\hbar$. Thus \eqref{eqn:UnstableJ1} is equal to:
\begin{align}\label{eqn:UnstableTermOfJ1}
  \sum_j\gamma_j\int_{[E^2/\mu_3]\times
  B\mu_3}\ev_\star^*(\gamma^j)=(1\otimes1\otimes1_\zeta)_1\int_{[E^2/\mu_3]\times B\mu_3}\ev_\star^*((H_x\otimes H_y\otimes1_{\zeta^2})_1)=1_\theta,
\end{align}
the twisted sector of $\H(\theta)$ from Remark
\ref{rem:AnalogOfUntwistedSector}.

\medskip

\noindent\textbf{The term $(\beta_0(\theta,2),2)$.} This term is the
$F_\beta$-contribution to the sum
\begin{align}
\sum_j\gamma_j\langle\gamma^j\otimes p_\infty,t\rangle_{0,\star,\beta_0(\theta,2)}^{\epsilon,\theta,Gr}.\label{eqn:UnstableJ2}
\end{align}
$\LGQG_{0,1+\star}^\epsilon(Z(\theta),\beta_0(\theta,2))$ parametrizes
constant maps from a parametrized curve $C$ with \emph{two} order 3
orbifold points to $[E^2/\mu_3],$ together with a 3rd root of
$\omega_{C,\log}\cong\O_C$. $\mathcal{L}_a$ may be nontrivial, and
there are three choices (one trivial) for the 3rd root
$\mathcal{L}_z.$ The resulting stack is isomorphic to $IZ(\theta),$
and in particular the isomorphism is $\ev_\star$ (after composition
with the rigidification $IZ(\theta)\to\bar{I}Z(\theta)$).

The virtual fundamental class was defined to vanish on the components
where $\mathcal{L}_z$ has trivial monodromy (Remark
\ref{rem:BroadVanish}), and on the other components it restricts to
the fundamental class. The union of these components is
$IZ(\theta)^{\nar}$. Since
$\mathcal{L}_x,\mathcal{L}_y,\mathcal{L}_z,\mathcal{L}_a$ have degree
zero, they have opposite monodromies at $b_1$ and $\star,$ so
$\ev_1=\upsilon\circ\ev_1.$ The normal bundle has contributions $\hbar$
and $-\hbar$ from moving the images of $\star$ and $b_1$, respectively, on $\P^1.$ Again,
$p_\infty|_{F_\beta}=-\hbar$. Thus \eqref{eqn:UnstableJ2} is equal to
\begin{align}\label{eqn:UnstableTermOfJ2}
  \sum_j\gamma_j\int_{IZ(\theta)^{\nar}}\frac{1}{\hbar}\upsilon^*\gamma^j\cup
  t=\frac{1}{\hbar}\sum_j\gamma_j\langle t,\gamma^j\rangle_{Z(\theta)}=t/\hbar.
\end{align}

\begin{rem}\label{rem:JIndependentOfEpsilon}
  In both of these calculations, the moduli spaces described
  parametrized maps LG-quasimaps \emph{without basepoints}. More
  generally, $m$-marked LG-quasimaps of degree $\beta_0(\theta,m)$
  never have basepoints. Thus we observe that integrals over these
  moduli spaces are \emph{independent of $\epsilon$}. In particular,
  the coefficient of $q^{(0,0,0,0)}$ in $J^{\epsilon,\theta}(t,q,\hbar)$
  is independent of $\epsilon.$
\end{rem}

\begin{prop}
  $J^{\epsilon,\theta}(0,q,\hbar)$ is homogeneous of degree zero, when
  $\deg(q):=0$ and $\deg(\hbar)=2$.
\end{prop}
\begin{proof}
  By Section \ref{sec:UnstableConventions}, any term of
  $J^{\epsilon,\theta}(0,q,\hbar)$ may be expressed as the $F_\beta$
  contribution to
  \begin{align}
    \sum_j\gamma_j\langle\gamma^j\otimes
    p_\infty\rangle_{0,1,\beta}^{\epsilon,\theta,Gr}=(\ev_1)_*\left(\ev_1^*p_\infty\right)\in\H(\theta).\label{eqn:JZero}
  \end{align}
  From Theorem \ref{thm:DMStack}, $\LGQG_{0,1}(Z(\theta),\beta)$ has
  unshifted virtual dimension 4, so this pushforward has relative
  dimension 1. (This statement depends on correctly shifting degrees
  as in Section \ref{sec:CompactTypeStateSpace}.) Thus
  \eqref{eqn:JZero} has degree zero.
\end{proof}
\begin{cor}\label{cor:JPowerSeries}
  $J^{\epsilon,\theta}(t,q,\hbar)\in\H(\theta)[[q,\hbar^{-1}]].$
\end{cor}
\begin{proof}
  A priori, $J^{\epsilon,\theta}(t,q,\hbar)$ may have positive powers of
  $\hbar$ coming from the unstable terms. However,
  $J^{\epsilon,\theta}(0,q,\hbar)$ includes all unstable terms, and is
  homogeneous of degree zero, so this does not occur.
\end{proof}

\begin{rem}\label{rem:JInftyPowerSeries}
  By Remark \ref{rem:BetaThetaZero}, only the tuples on the first line
  of \eqref{eqn:UnstableTermsOfJ} appear in the case
  $\epsilon=\infty$. In particular,
  \begin{align*}
    J^{\infty,\theta}(t,q,\hbar)=1_\theta+\frac{t}{\hbar}+O(1/\hbar^2).
  \end{align*}
  Combining this with Remark \ref{rem:JIndependentOfEpsilon} implies
  $J^{\epsilon,\theta}(t,q,\hbar)=1_\theta+\frac{t}{\hbar}+O(q)+O(1/\hbar^2),$
  where $O(q)\in\H(\theta)[[q]]$ has no $q^{(0,0,0,0)}$-coefficient.
\end{rem}

\subsection{The $S^{\epsilon,\theta}$-operator and its inverse}\label{sec:SOperator}
We define operators on $\H(\theta)[[q]]((\hbar^{-1}))$ by:
\begin{align*}
  S^{\epsilon,\theta}(t,q,\hbar)(\gamma)&=\sum_j\gamma_j\langle\langle\frac{\gamma^j}{\hbar-\psi},\gamma\rangle\rangle_{0,2}^{\epsilon,\theta}\\
  (S^{\epsilon,\theta})^\star(t,q,-\hbar)(\gamma)&=\sum_j\gamma^j\langle\langle\gamma_j,\frac{\gamma}{-\hbar-\psi}\rangle\rangle_{0,2}^{\epsilon,\theta}.
\end{align*}
\begin{rem}
  Under some conditions, we may make sense of applying these operators
  to power series in $\hbar$. The details are in Section 5.1 of
  \cite{CiocanFontanineKim2014}.
\end{rem}
\begin{prop}\label{prop:SOperatorIdentity}
  $S^{\epsilon,\theta}(t,q,\hbar)=\Id+O(1/\hbar)$ and
  $(S^{\epsilon,\theta})^{\star}(t,q,-\hbar)=\Id+O(1/\hbar)$.
\end{prop}
\begin{proof}
  As in Corollary \ref{cor:JPowerSeries}, the only terms of
  $S^{\epsilon,\theta}(t,q,\hbar)(\gamma)$ with nonnegative powers of $\hbar$
  come from unstable tuples $(\beta,m+2).$ The only such tuple is
  $(\beta_0(\theta,2),2).$ The corresponding term is the contribution
  of $F_\beta=F_{1,(0,0,0,0)}^{1,(0,0,0,0)}$ to
  \begin{align*}
    \sum_j\gamma_j\langle \hbar\gamma^j\otimes p_\infty,\gamma\rangle_{0,2,\beta_0(\theta,2)}^{\epsilon,\theta}
  \end{align*}
  This contribution was calculated (Section \ref{sec:JFunction},
  Equation \eqref{eqn:UnstableJ2}), and
  it is equal to $\gamma$. The argument for
  $(S^{\epsilon,\theta})^{\star}(t,q,-\hbar)$ is the same.
\end{proof}

\begin{prop}\label{prop:SOperatorInverse}
  $$(S^{\epsilon,\theta})^\star(t,q,-\hbar)\left(S^{\epsilon,\theta}(t,q,\hbar)(\gamma)\right)=\gamma.$$
\end{prop}
\begin{proof}
  The series
  $$\langle\langle\gamma\otimes[0],\delta\otimes[\infty]\rangle\rangle_2^{\epsilon,\theta,Gr}$$
  is a sum of equivariant integrals, hence a power series in $\hbar.$
  Applying localization gives a sum over fixed components where
  $\tau(b_1)=0\in\P^1$ and $\tau(b_2)=\infty\in\P^1.$ These fixed
  components were described in Section \ref{sec:CStarAction} as
  fibered products, and their normal bundles were calculated. These
  yield:
  \begin{align}\label{eqn:LocalizedP}
    \langle\langle\gamma\otimes[0],\delta\otimes[\infty]\rangle\rangle_2^{\epsilon,\theta,Gr}=\sum_j\langle\langle\frac{\gamma^j}{\hbar-\psi},\gamma\rangle\rangle_{0,2}^{\epsilon,\theta}\langle\langle\delta,\frac{\gamma_j}{-\hbar-\psi}\rangle\rangle_{0,2}^{\epsilon,\theta}.
  \end{align}
  The constant term in $\hbar$ of the right side is the contribution
  from the fixed component
  $F_{1,\beta_0(\theta,2)}^{1,\beta_0(\theta,2)}.$ Again, this
  calculation is essentially the one from Section \ref{sec:JFunction}
  (Equation \eqref{eqn:UnstableJ2}), and the answer is
  $\int_{IZ(\theta)}\upsilon^*\gamma\cup\delta$.

  As this is the only term of the right side of \eqref{eqn:LocalizedP}
  with a nonnegative power of $\hbar$, and the left side of
  \eqref{eqn:LocalizedP} is a power series in $\hbar$, we conclude:
  $$\sum_j\langle\langle\frac{\gamma^j}{\hbar-\psi},\gamma\rangle\rangle_{0,2}^{\epsilon,\theta}\langle\langle\delta,\frac{\gamma_j}{-\hbar-\psi}\rangle\rangle_{0,2}^{\epsilon,\theta}=\int_{IZ(\theta)}\upsilon^*\gamma\cup\delta.$$
  Using this, we have:
  \begin{align*}
    (S^{\epsilon,\theta})^\star(t,q,-\hbar)\left(S^{\epsilon,\theta}(t,q,\hbar)(\gamma)\right)&=\sum_j\gamma^j\langle\langle\gamma_j,\frac{\sum_{j'}\gamma_{j'}\langle\langle\frac{\gamma^{j'}}{\hbar-\psi},\gamma\rangle\rangle_{0,2}^{\epsilon,\theta}}{-\hbar-\psi}\rangle\rangle_{0,2}^{\epsilon,\theta}\\
                                                                                      &=\sum_{j,j'}\gamma^j\langle\langle\frac{\gamma^{j'}}{\hbar-\psi},\gamma\rangle\rangle_{0,2}^{\epsilon,\theta}\langle\langle\gamma_j,\frac{\gamma_{j'}}{-\hbar-\psi}\rangle\rangle_{0,2}^{\epsilon,\theta}\\
                                                                                      &=\sum_j\gamma^j\int_{IZ(\theta)}\upsilon^*\gamma\cup\gamma_j=\gamma.\qedhere
  \end{align*}
\end{proof}

\subsection{The $P$-series}\label{sec:PSeries}
Finally, we define:
\begin{align*}
  P^{\epsilon,\theta}(t,q,\hbar):=\sum_j\gamma^j\langle\langle
  \gamma_j\otimes p_\infty\rangle\rangle_1^{\epsilon,\theta,Gr}\in\H(\theta)[[q,\hbar]].
\end{align*}
\begin{prop}
  $P^{\epsilon,\theta}(t,q,\hbar)=1_\theta+O(q)$.
\end{prop}
\begin{proof}
  This coefficient of $q^{(0,0,0,0)}$ in $P^{\epsilon,\theta}(t,q,\hbar)$
  is
  \begin{align}
    \sum_{m,j}\frac{\gamma_j}{m!}\langle\gamma^j\otimes p_{\infty},t,\ldots,t\rangle_{1+m,\beta_0(\theta,m)}^{\epsilon,\theta,Gr},\label{eqn:CoeffOfQZeroP}
  \end{align}
  The moduli spaces
  $\LGQG_{0,1+m}^{\epsilon,\theta}(Z(\theta),(0,0,\frac{-2+m}{3},0))$
  are similar to the one described in Section \ref{sec:JFunction}, but
  somewhat more complicated. They have
  \begin{enumerate}
  \item Components corresponding to LG-quasimaps where $\mathcal{L}_a$
    is trivial, and\label{item:UntwistedComponents}
  \item Components corresponding to LG-quasimaps where $\mathcal{L}_a$
    has nontrivial monodromy at some marked
    point.\label{item:TwistedComponents}
  \end{enumerate}
  Since $\mathcal{L}_a$ has degree zero, these are the only
  possibilities.

  First we consider components of type
  \eqref{item:UntwistedComponents}. As $\mathcal{L}_x$ and
  $\mathcal{L}_y$ are trivial the union of these components is
  isomorphic to $[E^2/\mu_3]\times\mathcal{W}_{0,1+m}(\P^1)$, where
  $\mathcal{W}_{0,1+m}(\P^1)$ is a moduli space of \emph{spin curves}
  (see \cite{ChiodoZvonkine2009}). It has dimension $1+m$ and
  parametrizes
  \begin{itemize}
  \item A parametrized $1+m$-marked curve $C$, and
  \item A 3rd root of $\omega_{C,\log}.$
  \end{itemize}
  Under this identification the evaluation maps $\ev_i$ are given by
  the product
$$(\id,\mult_{b_i}(\mathcal{L}_z)):[E^2/\mu_3]\times\mathcal{W}_{0,1+m}(\P^1)\to[E^2/\mu_3]\times\bar{I}B\mu_3.$$
Write $\gamma_j=\gamma_{j,1}\otimes\gamma_{j,2},$ where
$\gamma_{j,1}\in H^*(E^2/\mu_3,\C)$ and
$\gamma_{j,2}\in H^*_{CR}(B\mu_3).$ Similarly write
$t=t_1\otimes t_2\in H^*(E^2/\mu_3,\C)\otimes H^*_{CR}(B\mu_3)$. Then
\eqref{eqn:CoeffOfQZeroP} is equal to the product
\begin{align}\label{eqn:UnstableCoefficientsOfP}
  \frac{1}{m!}\left(\int_{[E^2/\mu_3]}\gamma_{j,1}\cup
  t_1^m\right)\left(\int_{[\mathcal{W}_{0,1+m}(\P^1)]^{\vir}}\ev_1^*(\gamma_{j,2}\otimes
  p_\infty)\cup\prod_{i=1}^m\ev_i^*(t_2)\right),
\end{align}
Using the projection formula, we rewrite the second integral as
$\int_{\P^1}p_\infty\cup\alpha,$ where $\alpha\in H^{1+m}(\P^1,\C)$ is
a nonequivariant class pushed forward from
$\mathcal{W}_{0,1+m}(\P^1)$. Thus the second integral vanishes unless
$m=0$. The case $m=0$ has been computed (Section
\ref{sec:JFunction}, Equation \eqref{eqn:UnstableJ2}), and this term is equal to
$1_\theta$.

For components of type \eqref{item:TwistedComponents}, the moduli
space is only a fibered product, not a product, but we similarly find
that terms with $m>0$ do not contribute. However, there are no
components of type \eqref{item:TwistedComponents} with $m=0$, since
$\mathcal{L}_a$ is a 3rd root of the trivial bundle, and when $m=0$
there are no nontrivial 3rd roots of the trivial bundle. Thus
components of type \eqref{item:TwistedComponents} do not contribute to
the coefficient of $q^{(0,0,0,0)}.$
\end{proof}

\noindent\textbf{Factorization of $J^{\epsilon,\theta}(t,q,\hbar)$.} Next we apply $\C^*$-localization to $P^{\epsilon,\theta}$. As in the proof of
Proposition \ref{prop:SOperatorInverse}, we have
\begin{align*}
  P^{\epsilon,\theta}(t,q,\hbar)=\sum_{j,j'}\gamma_j\langle\langle\frac{\gamma^{j'}}{\hbar(\hbar-\psi)}\rangle\rangle_{0,1}^{\epsilon,\theta}\langle\langle(-\hbar)\gamma_j,\frac{\gamma_{j'}}{-\hbar(-\hbar-\psi)}\rangle\rangle_{0,2}^{\epsilon,\theta}.
\end{align*}
Now factoring gives
\begin{align*}
  P^{\epsilon,\theta}(t,q,\hbar)&=\sum_j\gamma_j\langle\langle(-\hbar)\gamma_j,\frac{\sum_{j'}\gamma_{j'}\langle\langle\frac{\gamma^{j'}}{\hbar(\hbar-\psi)}\rangle\rangle_{0,1}^{\epsilon,\theta}}{-\hbar(-\hbar-\psi)}\rangle\rangle_{0,2}^{\epsilon,\theta}\\
  &=\sum_j\gamma_j\langle\langle(-\hbar)\gamma_j,\frac{J^{\epsilon,\theta}(t,q,\hbar)}{-\hbar(-\hbar-\psi)}\rangle\rangle_{0,2}^{\epsilon,\theta}\\
  &=\sum_j\gamma_j\langle\langle\gamma_j,\frac{J^{\epsilon,\theta}(t,q,\hbar)}{(-\hbar-\psi)}\rangle\rangle_{0,2}^{\epsilon,\theta}=(S^{\epsilon,\theta})^\star(t,q,-\hbar)(J^{\epsilon,\theta}(t,q,\hbar)).
\end{align*}
The last expression contains no positive powers of $\hbar$, but
$P^{\epsilon,\theta}(t,q,\hbar)$ contains no negative powers of $\hbar$. Thus
$P^{\epsilon,\theta}(t,q,\hbar)=P^{\epsilon,\theta}(t,q)\in\H(\theta)[[q]].$
Applying Proposition \ref{prop:SOperatorInverse}, we have:
\begin{cor}\label{cor:BirkhoffFactorization}
  $J^{\epsilon,\theta}(t,q,\hbar)=S^{\epsilon,\theta}(t,q,-\hbar)(P^{\epsilon,\theta}(t,q)).$
\end{cor}
By Proposition \ref{prop:SOperatorIdentity} and Section
\ref{sec:JFunction}, we have
$$P^{\epsilon,\theta}(t,q)=P^{\epsilon,\theta}(q)\cdot1_\theta\in\C[[q]]\cdot1_\theta$$
is independent of $t$ and
$$J^{\epsilon,\theta}(t,q,\hbar)=P^{\epsilon,\theta}(q)\cdot1_\theta+O(1/\hbar).$$
In particular, Remark \ref{rem:JInftyPowerSeries} shows that for
$\epsilon=\infty$,
\begin{align*}
  J^{\infty,\theta}(t,q,\hbar)=1_\theta+O(1/\hbar),
\end{align*}
hence $P^{\infty,\theta}(q)=1_\theta.$ This implies
\begin{align*}
  J^{\infty,\theta}(t,q,\hbar)=S^{\infty,\theta}(t,q,\hbar)(1_\theta),
\end{align*}
which also follows from the string equation, correctly adapted as a
combination of that for stable maps (\cite{AbramovichGraberVistoli2008},
Theorem 8.3.1) and that appearing in FJRW theory
(\cite{FanJarvisRuan2013}, Theorem 4.2.9).

\section{Mirror theorems}\label{sec:MirrorTheorems}
\subsection{Setup}\label{sec:MirrorSetup}
In this section we will relate $J^{\infty,\theta}(t,q,\hbar)$ to
$J^{\epsilon,\theta}(t,q,\hbar),$ and we describe here precisely how they
are related.

In Section \ref{sec:GeneratingFunctions}, we could have taken
$t\in\H(\theta)[[q]]$ rather than $t\in\H(\theta)$ with no
changes. Doing so, we may formally view
$J^{\epsilon,\theta}(t,q,\hbar)$ as a map
$\H(\theta)[[q]]\to\H(\theta)[[q,\hbar^{-1}]].$ It is well-known (see
\cite{Tseng2010}) that the image of $J^{\infty,\theta}(t,q,\hbar)$
lies on and determines a (germ of a) cone
$\mathfrak{L}_\theta\subseteq\H(\theta)[[q,\hbar^{-1}]].$\footnote{Here
  the word \emph{cone} refers to a subset that is preserved under
  multiplication by elements from the ``base ring'' $\C[[q]].$ It is
  also a fact (which we will not need) that
  $\H(\theta)[[q]]((\hbar^{-1}))$ has a symplectic structure and that
  the cone is a Lagrangian submanifold.} We will show that
$J^{\epsilon,\theta}(t,q,\hbar)$ also lies on this cone also for all
$\epsilon$.

From Section \ref{sec:PSeries}, $J^{\epsilon,\theta}(t,q,\hbar)$ differs
by an element of $\C[[q]]$ from
$S^{\epsilon,\theta}(t,q,\hbar)(1_\theta).$ Thus the claim that
$J^{\epsilon,\theta}(t,q,\hbar)$ lies on the cone $\mathfrak{L}_\theta$
follows from the claim that $S^{\epsilon,\theta}(t,q,\hbar)(1_\theta)$
lies on $\mathfrak{L}_\theta.$

We will prove:
\begin{thm}[All-chamber mirror theorem]\label{thm:MirrorTheorem}
  There is an automorphism $\mathscr{T}^{\epsilon,\theta}$ of
  $\H(\theta)[[q]]$ such
  that
  $$S^{\epsilon,\theta}(t,q,\hbar)(1_\theta)=S^{\infty,\theta}(\mathscr{T}^{\epsilon,\theta}(t),q,\hbar)(1_\theta)=J^{\infty,\theta}(\mathscr{T}^{\epsilon,\theta}(t),q,\hbar).$$
  In particular, the image of $S^{\epsilon,\theta}(t,q,\hbar)(1_\theta)$
  is the same as the image of $J^{\infty,\theta}(t,q,\hbar).$
\end{thm}
To find $\mathscr{T}^{\epsilon,\theta}(t)$, we expand:
\begin{align*}
  S^{\epsilon,\theta}(t,q,\hbar)(1_\theta)&=1_\theta+\frac{1}{\hbar}\sum_{\substack{\beta,m,j\\(\beta,m)\ne(\beta_0(\theta,2),2)}}\gamma_j\frac{q^{\beta-\beta_0(\theta,m)}}{m!}\langle\gamma^j,1_\theta,t,\ldots,t\rangle_{0,2+m,\beta}^{\epsilon,\theta}+O(1/\hbar^2)\\
                                      &=1_\theta+\frac{1}{\hbar}\left(\sum_j\gamma_j\langle\langle\gamma^j,1_\theta\rangle\rangle_{0,2}^{\epsilon,\theta}-1_\theta\right)+O(1/\hbar^2).
\end{align*}
(For consistency, the unstable part
$\sum_j\gamma_j\langle\gamma^j,1_\theta\rangle_{0,2,\beta_0(\theta,2)}^{\epsilon,\theta}$
is defined to be $1_\theta$.) Set:
$$\mathscr{T}^{\epsilon,\theta}(t):=\sum_j\gamma_j\langle\langle\gamma^j,1_\theta\rangle\rangle_{0,2}^{\epsilon,\theta}-1_\theta.$$
\begin{observation}\label{obs:AgreeToFirstOrder}
  $J^{\infty,\theta}(\mathscr{T}^{\epsilon,\theta}(t),q,\hbar)=S^{\epsilon,\theta}(t,q,\hbar)(1_\theta)+O(1/\hbar^2).$
\end{observation}
To see that $\mathscr{T}^{\epsilon,\theta}(t)$ is an automorphism of
$\H(\theta)[[q]],$ we equate coefficients of $q^{(0,0,0,0)}\hbar^{-1}$ in
Corollary \ref{cor:BirkhoffFactorization}. Using Remark
\ref{rem:JInftyPowerSeries}, we see that
$\mathscr{T}^{\epsilon,\theta}(t)=t+O(q).$

Our strategy for establishing the equality of
$S^{\epsilon,\theta}(t,q,\hbar)(1_\theta)$ and
$J^{\infty,\theta}(\mathscr{T}^{\epsilon,\theta}(t),q,\hbar)$ is to
calculate both using $T$-localization. Since both involve integrals
over $\LGQ_{0,m}^\epsilon(Z(\theta),\beta),$ not
$\LGQ_{0,m}^\epsilon(X(\theta),\beta),$ we need to rewrite them. In
particular, the coefficients of
$S^{\epsilon,\theta}(t,q,\hbar)(\gamma)$ are integrals of the form
\begin{align}\label{eqn:QuasimapInvariant}
  \int_{[\LGQ_{0,m}^\epsilon(Z(\theta),\beta)]^{\vir}}\prod_{i=1}^m\psi_i^{a_i}\ev_i^*(\alpha_i).
\end{align}
From Section \ref{sec:CompactTypeStateSpace}, we only consider classes
$\alpha_i=\iota^*\mathfrak{a}_i$ pulled back from $\bar{I}X(\theta).$
Thus if
$e:\LGQ_{0,m}^\epsilon(Z(\theta),\beta)\into\LGQ_{0,m}^\epsilon(X(\theta),\beta)$
is the natural embedding, we may rewrite the above as
\begin{align*}
  \int_{e_*[\LGQ_{0,m}^\epsilon(Z(\theta),\beta)]^{\vir}}\prod_{i=1}^m\psi_i^{a_i}\ev_i^*(\mathfrak{a}_i).
\end{align*}
Here we use the fact that evaluation maps and $\psi$ classes are
compatible with $\iota.$ Now we use the general fact about cosection
localization that
\begin{align}\label{eqn:VirtualCyclePushforward}
  \iota_*[\LGQ_{0,m}^\epsilon(Z(\theta),\beta)]^{\vir}=[\LGQ_{0,m}^\epsilon(X(\theta),\beta)]^{\vir}
\end{align}
to we rewrite the $S$-operator (see \cite{KiemLi2013}). For
$\alpha\in\H(\theta),$ write $\tilde{\alpha}$ for the corresponding element of
$H^*_{CR}(X(\theta))^{\nar}/(\ker\iota^*).$ Then we have
\begin{align*}
  S^{\epsilon,\theta}(t,q,\hbar)(\gamma)=\sum_{\beta,m,j}\frac{q^{\beta-\beta_0(\theta,2+m)}}{m!}\gamma_j\int_{[\LGQ_{0,2+m}^\epsilon(X(\theta),\beta)]^{\vir}}\frac{\ev_1^*\tilde{\gamma^j}}{\hbar-\psi_1}\cup\ev_2^*\tilde{\gamma}\cup\ev_3^*\tilde{t}\cdots\ev_{m+2}^*\tilde{t}.
\end{align*}
(Note $\LGQ_{0,m}^\epsilon(X(\theta),\beta)$ is not proper; however,
by \eqref{eqn:VirtualCyclePushforward} the virtual fundamental class is
a homology class, in particular compactly supported, so the integral
makes sense.) We also see that $S^{\epsilon,\theta}(t,q,\hbar)(\gamma)$ is
pulled back from $\bar{I}X(\theta)$:
\begin{align*}
  \tilde{S}^{\epsilon,\theta}(t,q,\hbar)(\gamma):=\sum_{\beta,m,j}\frac{q^{\beta-\beta_0(\theta,2+m)}}{m!}\tilde{\gamma_j}\int_{[\LGQ_{0,2+m}^\epsilon(X(\theta),\beta)]^{\vir}}\frac{\ev_1^*\tilde{\gamma^j}}{\hbar-\psi_1}\cup\ev_2^*\tilde{\gamma}\cup\ev_3^*\tilde{t}\cdots\ev_{m+2}^*\tilde{t}.
\end{align*}
There is no Poincar\'e pairing on $\bar{I}X(\theta)$, as it is not
proper; however, we may still view the elements
$\tilde{\gamma_j}\in H_{CR}^*(X(\theta))^{nar}/(\ker\iota^*)$ as a
dual basis to $\{\tilde{\gamma^j}\}$ in the following sense. Since
$\bar{I}X(\theta)$ deformation retracts to
$\bar{I}X_R^{\rig}(\theta),$ which is proper and contains
$\bar{I}Z(\theta),$ there is a cohomology class $\underline{Z}$ such
that
$\underline{Z}\cap X_R^{\rig}(\theta)=\iota_*[\bar{I}Z(\theta)]\in
H_*(\bar{I}X(\theta))$.
Then the Poincar\'e pairing on $Z(\theta)$ induces the perfect pairing
on $H_{CR}^*(X(\theta))^{\nar}/(\ker\iota^*)$:
\begin{align*}
  \langle\alpha,\beta'\rangle_{\underline{Z}}:=\int_{[X_R^{\rig}(\theta)]}\alpha\cup\beta\cup\underline{Z},
\end{align*}
and under this pairing $\{\tilde{\gamma_j}\}$ and
$\{\tilde{\gamma^j}\}$ are again dual bases. Alternatively,
$\{\underline{Z}\cup\tilde{\gamma_j}\}$ and $\{\tilde{\gamma^j}\}$ are
dual bases with respect to a perfect pairing
$$\underline{Z}\cup H_{CR}^*(X(\theta))^{\nar}/(\ker\iota^*)\otimes
H_{CR}^*(X(\theta))^{\nar}/(\ker\iota^*)\to\C.$$
For this reason, we define
\begin{align}\label{eqn:ZOperator}
  \mathfrak{Z}^{\epsilon,\theta}(t,q,\hbar)(\gamma):=\underline{Z}\cup\tilde{S}^{\epsilon,\theta}(t,q,\hbar)(\gamma).
\end{align}
It is then sufficient to show:
\begin{align*}
  \mathfrak{Z}^{\infty,\theta}(\mathscr{T}^{\epsilon,\theta}(t),q,\hbar)=\mathfrak{Z}^{\epsilon,\theta}(t,q,\hbar)(1_\theta).
\end{align*}
\begin{rem}
  It will simplify things to choose lifts of classes in $\H(\theta)$, rather
  than working with elements of
  $H^*_{CR}(X(\theta))^{\nar}/(\ker\iota^*).$ Therefore in our
  notation we view $\tilde{\gamma_j},$ $\tilde{\gamma^j},$
  $\tilde{\gamma},$ $\tilde{t}$, and
  $\mathfrak{Z}^{\epsilon,\theta}(t,q,\hbar)(\gamma)$ as elements of
  $H^*_{CR}(X(\theta))^{\nar}[[q,\hbar^{-1}]].$ Of course, they are not
  well-defined, but their pullbacks to $Z(\theta)$ are.
\end{rem}

We are now in a situation to follow the arguments of
\cite{CiocanFontanineKim2014}. For each fixed point $\mu$ of
$\bar{I}X(\theta)$, consider the $T$-equivariant integral
$$\mathfrak{Z}^{\epsilon,\theta}_\mu(t,q,\hbar)(\gamma):=i_\mu^*\mathfrak{Z}^{\epsilon,\theta}(t,q,\hbar)(\gamma)=\int_{[X_R^{\rig}(\theta)]}\delta_\mu\cup\mathfrak{Z}^{\epsilon,\theta}(t,q,\hbar)(\gamma),$$
where $\delta_\mu\in H^*_{CR,T,\loc}(X(\theta))$ is the equivariant
fundamental class of $\mu$ and $i_\mu:\mu\into\bar{I}X(\theta)$ is the
inclusion. Precisely, as $\bar{I}X(\theta)$ has isolated fixed points,
the Atiyah-Bott localization formula states that its localized
equivariant cohomology groups $H^*_{CR,T,\loc}(X(\theta))$ are
generated by pushforwards of fundamental classes of the fixed points
from the equivariant cohomology of the fixed locus. We rewrite
\begin{align}\label{eqn:SMu}
  \mathfrak{Z}_\mu^{\epsilon,\theta}(t,q,\hbar)(\gamma)=\sum_{m,\beta}\frac{q^{\beta-\beta_0(\theta,2+m)}}{m!}\int_{[\LGQ_{0,2+m}^\epsilon(X(\theta),\beta)]^{\vir}}\frac{\ev_1^*\delta_\mu}{\hbar-\psi_1}\cup\ev_2^*\tilde{\gamma}\cup\ev_3^*\tilde{t}\cup\cdots\cup\ev_{m+2}^*\tilde{t}.
\end{align}
These are the objects of interest in the next section.

\subsection{Localization and recursion}
From Section \ref{sec:TorusAction}, \eqref{eqn:SMu} has a natural
equivariant lift, so we apply the fixed-point localization formula to
write:
\begin{align}\label{eqn:LocalizedSTilde}
  \mathfrak{Z}^{\epsilon,\theta}_\mu(t,q,\hbar)(\gamma)=\sum_{m,\beta,F}\frac{q^{\beta-\beta_0(\theta,2+m)}}{m!}\int_{[F]^{\vir}}\frac{i_F^*\left(\frac{\ev_1^*\delta_\mu}{\hbar-\psi_1}\cup\ev_2^*\tilde{\gamma}\cup\cdots\cup\ev_{m+2}^*\tilde{t}\right)}{e(N_F^{\vir})},
\end{align}
where $i_F:F\into\LGQ_{0,2+m}^\epsilon(X(\theta),\beta)$ is the
inclusion of a component of the fixed locus, $[F]^{\vir}$ is the
virtual fundamental class from the $T$-fixed part of
$R^\bullet\pi_*\mathcal{E}$, and $N_F^{\vir}$ is the $T$-equivariant
virtual normal bundle, defined to by the $T$-moving part of
$R^\bullet\pi_*\mathcal{E}$.

We recall terminology from \cite{CiocanFontanineKim2014}.
\begin{Def}\label{Def:ComponentTypes}
  For each fixed point $\mu$ of
  $\bar{I}X(\theta),$ we partition the components of the $T$-fixed
  locus of $\LGQ_{0,2+m}^\epsilon(X(\theta),\beta)$ into three
  subsets:
  \begin{itemize}
  \item $V(\mu,\beta,2+m)$ consists of components for which the first
    marking does not map to $\mu$,
  \item $\In(\mu,\beta,2+m)$ consists of components for which the
    first marking maps to $\mu$ and is on a contracted component of
    $C$ (see Definition \ref{Def:Contract}), and
  \item $\Rec(\mu,\beta,2+m)$ consists of components for which the
    first marking maps to $\mu$ and is not on a contracted component
    of $C$. In this case, $u^{\rig}$ sends this component to a fixed
    curve in $\bar{I}X(\theta)$ connecting $\mu$ to a unique other
    fixed point, which we denote by $\nu$.
  \end{itemize}
\end{Def}
This is Lemma 7.5.1 of \cite{CiocanFontanineKim2014}:
\begin{lem}\label{lem:SMuRationalFunction}
  For each $(\beta,(k_j)_j)$ with $\sum_jk_j=m$ the coefficient of
  $q^{\beta-\beta_0(\theta,2+m)}\prod_jt_j^{k_j}$ in
  $\mathfrak{Z}^{\epsilon,\theta}_\mu(\gamma)$ is a rational function
  of $\hbar$ with coefficients in $K$ (see Section
  \ref{sec:TorusAction}). This rational function decomposes as a
  finite sum of rational functions with denominators either powers of
  $\hbar$, of powers of linear factors $\hbar-\alpha$, where $-n\alpha$ is one
  of the weights of the $T$-representation $T_\mu(\bar{I}X(\theta))$
  for some $n\in\Z_{>0}$. 
\end{lem}
\begin{proof}
  The proof in \cite{CiocanFontanineKim2014} requires essentially no
  modification, and we summarize it here.
  
  From \eqref{eqn:LocalizedSTilde}, the coefficient of
  $q^{\beta-\beta_0(\theta,2+m)}\prod_jt_j^{k_j}$ in
  $\mathfrak{Z}^{\epsilon,\theta}_\mu(t,q,\hbar)(\gamma)$ is:
  \begin{align}\label{eqn:CoefficientRationalFunction}
    \sum_F\frac{1}{(m!)}\int_{[F]^{\vir}}\frac{i_F^*\left(\frac{\ev_1^*\delta_u}{\hbar-\psi}\cup
    A\right)}{e(N_F^{\vir})}&=\sum_{F,a}\frac{1}{(m!\hbar^{a+1})}\int_{[F]^{\vir}}\frac{i_F^*\left(\psi_1^a\ev_1^*\delta_u\cup
    A\right)}{e(N_F^{\vir})},
  \end{align}
  where $A$ is the product of factors from the evaluation maps
  $2,\ldots,2+m$, and depends on $\beta$ and $(k_j)_j$.
  \begin{itemize}
  \item On components in $V(\mu,\beta,2+m)$, the factor
    $\ev_1^*\delta_\mu$ restricts to zero.
  \item On components in $\In(\mu,\beta,2+m),$ $\psi_1$ is
    nonequivariant, hence nilpotent, so the denominators are (bounded)
    powers of $\hbar$.
  \item On components in $\Rec(\mu,\beta,2+m)$, the $\psi_1$ is an
    equivariant class. However, if $d$ is the degree of $u^{\rig}$ on
    the component containing $b_1$, then the fibers of the
    $(T_{b_1}^*C)^{\otimes d}$ is naturally isomorphic to
    $T_\mu^*X_{\mu,\nu}$ from Section \ref{sec:TorusAction}. Thus the
    left side of \eqref{eqn:CoefficientRationalFunction} has a simple
    pole at $\psi_1=\frac{-w(\mu,\nu)}{d}$.\qedhere
  \end{itemize}
\end{proof}
This is Lemma 7.5.2 of \cite{CiocanFontanineKim2014}, and is
essentially unchanged from Proposition 4.4 of \cite{Givental1998}.
\begin{lem}\label{lem:Recursion}
  $\mathfrak{Z}^{\epsilon,\theta}_\mu(t,q,\hbar)$ satisfies the recursion
  \begin{align}\label{eqn:Recursion}
    \mathfrak{Z}^{\epsilon,\theta}_\mu(t,q,\hbar)=R^{\epsilon,\theta}_\mu(t,q,\hbar)+\sum_{\substack{\nu\text{
    \emph{$T$-adjacent}}\\\text{\emph{to $\mu$}}}}\sum_{d=1}^\infty
    q^{d\beta(\mu,\nu)}\frac{C_{\mu,\nu,d}}{\hbar+\frac{w(\mu,\nu)}{d}}\mathfrak{Z}^{\epsilon,\theta}_\nu\left(t,q,\frac{w(\mu,\nu)}{d}\right),
  \end{align}
  such that 
  \begin{itemize}
  \item $R^{\epsilon,\theta}_\mu(t,q,\hbar)$ is a power series in $1/\hbar$
    such that for each $(\beta,(k_j)_j)$ with $\sum k_j=m,$ the
    coefficient of $q^{\beta-\beta_0(\theta,2+m)}\prod_jt_j^{k_j}$ is
    a polynomial in $1/\hbar,$
  \item $\beta(\mu,\nu)$ is a degree dependent only on $\mu$ and
    $\nu$,
  \item $w(\mu,\nu)$ is the tangent weight defined in Section
    \ref{sec:TorusAction}, and
  \item $C_{\mu,\nu,d}$ is independent of $\epsilon.$
  \end{itemize}
\end{lem}
The proof is similar to that in \cite{CiocanFontanineKim2014}. Note
also that in the case $\theta=\theta_{xyz}^a$ the second term is zero,
as every component is contracted.
\begin{proof}
  $\mathfrak{Z}^{\epsilon,\theta}_\mu(t,q,\hbar)$ has a single unstable
  term, from the unstable tuple $(\beta_0(\theta,2),2).$ Using Section
  \ref{sec:SOperator}, the contribution is
  $\langle{[Z(\theta)]\cup\tilde{\gamma},\delta_\mu}\rangle.$

  We analyze the contributions from $V(\mu,\beta,2+m),$
  $\In(\mu,\beta,2+m)$, and $\Rec(\mu,\beta,2+m)$ to
  \eqref{eqn:LocalizedSTilde}.  As in Lemma,
  \ref{lem:SMuRationalFunction}, the contribution of components in
  $V(\mu,\beta,2+m),$ is zero.

  Consider a fixed component in $\In(\mu,\beta,2+m).$ This
  parametrizes LG-quasimaps such that the associated quasimap
  $C\to[E^2/\mu_3]$ sends $b_1$ to $\mu$ and contracts the
  component containing $b_1.$ As in the proof of Lemma
  \ref{lem:SMuRationalFunction}, the contribution is a power series in
  $1/\hbar,$ whose
  $q^{\beta-\beta_0(\theta,2+m)}\prod_jt_j^{k_j}$-coefficient is an
  element of $K[1/\hbar].$ We define $R^{\epsilon,\theta}_\mu(t,q,\hbar)$ to
  be the sum of contributions from components in $\In(\mu,\beta,2+m)$.

  We now consider a fixed component $M\in\Rec(\mu,\beta,2+m),$
  corresponding to the term of \eqref{eqn:LocalizedSTilde}:
  \begin{align}\label{eqn:TermOfSMu}
    \frac{q^{\beta-\beta_0(\theta,2+m)}}{m!}\int_{[M]^{\vir}}\frac{\frac{\ev_1^*\delta_\mu}{\hbar-\psi_1}\cup\ev_2^*\tilde{\gamma}\cup\cdots\cup\ev_{m+2}^*\tilde{t}}{e(N_M^{\vir})}.
  \end{align}
  Let $\nu$ be as in Definition \ref{Def:ComponentTypes}.
  LG-quasimaps in (As with $\mu$, $\nu$ naturally lives in
  $\bar{I}X(\theta)^T$ rather than
  $X(\theta)^T.$) 
  By gluing LG-quasimaps, we can write $M$ as a fibered product
  $M'\times_{\bar{I}X(\theta)}M'',$ where $M'$ is a $T$-fixed
  component of $\LGQ_{0,1+\bullet}^\epsilon(X(\theta),\beta')$ and
  $M''$ is a $T$-fixed component of
  $\LGQ_{0,m+1+\check\bullet}^\epsilon(X(\theta),\beta-\beta')$. (Note
  that the meaning of $\check\bullet$ differs very slightly from that
  in Section \ref{sec:CStarAction}.) Here $\beta'$ is the degree of
  $(C',u',\kappa')$. The maps to $\bar{I}X(\theta)$ are, in the first
  case, the evaluation map at $\bullet$, and in the second case, the
  evaluation map at $\check\bullet,$ composed with the inversion map
  on $\bar{I}Z(\theta).$ (See Sections \ref{sec:3StableCurves} and
  \ref{sec:ChenRuan}.)

  As $C'$ has a single marked point, a single node, and no basepoints,
  we have $\beta_z'=0.$ Similarly $\beta_a'=0.$ Also, by the
  characterization of 1-dimensional $T$-orbits in Section
  \ref{sec:TorusAction}, either $\beta_x'=0$ or $\beta_y'=0,$ and by
  the noncontractedness of $C',$ the other is a positive integer. Thus
  it is of the form $d\beta(\mu,\nu),$ where $\beta(\mu,\nu)$ is
  either $(1,0,0,0)$ or $(0,1,0,0)$. In particular,
  $\beta'-\beta_0(\theta,2)\ne(0,0,0,0).$

  We wish to write \eqref{eqn:TermOfSMu} as a product of integrals
  over $M'$ and $M''.$ To do this, we need to compute the virtual
  class $[M]^{\vir}$ in terms of $[M']^{\vir}$ and $[M'']^{\vir}.$
  Smoothing the node $o:C'\cap C''$ gives the 
  distinguished triangle of relative perfect obstruction theories:
  \begin{align*}
    R^\bullet\pi_*\mathcal{E}\to R^\bullet\pi_*\mathcal{E}|_{C'}\oplus
    R^\bullet\pi_*\mathcal{E}|_{C'}\oplus\to R^\bullet\pi_*\mathcal{E}|_o\to R^\bullet\pi_*\mathcal{E}[1].
  \end{align*}
  The term $R^\bullet\pi_*\mathcal{E}|_o$ is isomorphic as a $G$-bundle
  over $M$ to the trivial bundle with fiber $V$, concentrated in
  degree zero.

  We instead need a perfect obstruction theory relative to the stack
  $\mathfrak{M}_{0,m}^{\tw}$ (Section \ref{sec:3StableCurves}). The
  difference comes from the relative tangent complex
  $\mathbb{T}_{\mathfrak{A}/\mathfrak{M}_{0,m}^{\tw}}$. This is equal
  to $\mathcal{P}\times_{G\times\C^*_R}\mathfrak{g}$, concentrated in
  degree -1, where $\mathfrak{g}$ is the Lie algebra of $G.$ Thus if
  we denote by $\mathfrak{F}^\bullet(C)$ the perfect obstruction
  theory of $\LGQ_{0,m}^\epsilon(X(\theta),\beta)$ relative to
  $\mathfrak{M}_{0,m}^{\tw}$, and by $\mathfrak{F}^\bullet(C')$, etc.,
  the corresponding perfect obstruction theories on $M',$ etc., the
  triangle above becomes
  \begin{align*}
    \mathfrak{F}^\bullet(C)\to\mathfrak{F}^\bullet(C')\oplus
    \mathfrak{F}^\bullet(C'')\to\mathfrak{F}^\bullet(o)\to\mathfrak{F}^\bullet[1],
  \end{align*}
  where every fiber of $\mathfrak{F}(o)$ can be canonically identified
  with $T_\nu\bar{I}X(\theta).$ As the $T$-fixed points of
  $\bar{I}X(\theta)$ are isolated, $T_\nu\bar{I}X(\theta)$ has no
  $T$-fixed part.

  Now, to be able to make statements about the \emph{absolute}
  obstruction theory of $\LGQ_{0,m}^\epsilon(X(\theta),\beta),$ we
  need to analyze the tangent complex of $\mathfrak{M}_{0,m}^{\tw}.$ Again
  we have a triangle
  \begin{align*}
    \mathbb{T}_{\mathfrak{M}_{0,m}^{\tw}}(C)\to\mathbb{T}_{\mathfrak{M}_{0,m}^{\tw}}(C')\oplus\mathbb{T}_{\mathfrak{M}_{0,m}^{\tw}}(C'')\to\mathbb{T}_{\mathfrak{M}_{0,m}^{\tw}}(o)\to\mathbb{T}_{\mathfrak{M}_{0,m}^{\tw}}[1],
  \end{align*}
  and here $\mathbb{T}_{\mathfrak{M}_{0,m}^{\tw}}(o)$ is the
  deformation space of the node $o$, with each fiber canonically
  isomorphic to $T_oC'\otimes T_oC''.$ The factor $T_oC'$ gives a
  topologically trivial bundle over $M$ up to torsion, with $T$-weight
  $w(\nu,\mu)/d$. The factor $T_oC''$ 
  may be topologically nontrivial (depending on $M''$), but in any
  case the $T$-action on $\mathbb{T}_{\mathfrak{M}_{0,m}^{\tw}}(o)$ is
  nontrivial. Thus we have, exactly as in
  \cite{CiocanFontanineKim2014}:
  \begin{align*}
    [M]^{\vir}&=[M']^{\vir}\times[M'']^{\vir}\\
    \frac{1}{e^T(N^{\vir}_M)}&=\frac{e^T(T_\nu\bar{I}X(\theta))}{e^T(N^{\vir}_{M'})e^T(N^{\vir}_{M''})(\frac{w(\nu,\mu)}{d}-\psi_1^{M''})}.
  \end{align*}
  Since $\ev_1|_{M''}$ is a constant map to $\nu\in\bar{I}X(\theta)$,
  and $w(\nu,\mu)=-w(\mu,\nu),$ \eqref{eqn:TermOfSMu} is equal to the
  product:
  \begin{align*}
    \left(q^{\beta'-\beta_0(\theta,2)}\int_{[M']^{\vir}}\frac{\ev_1^*\delta_\mu}{(\hbar+\frac{w(\mu,\nu)}{d})e^T(N^{\vir}_{M'})}\right)
    \left(\frac{q^{(\beta-\beta')-\beta_0(\theta,m+2)}}{m!}\int_{[M'']^{\vir}}\frac{\ev_1^*(\delta_\nu)\cup\ev_2^*(\tilde{\gamma})\cup\prod_i\ev_i^*(t)}{(\frac{-w(\mu,\nu)}{d}-\psi_1)e^T(N^{\vir}_{M''})}\right)
  \end{align*}
  The factor $\hbar+\frac{w(\mu,\nu)}{d}$ is pulled back from
  $H^*_{T\times\C^*}(\Spec\C,\C),$ so may be factored out. The resulting
  integral $C_{\mu,\nu,d}$ is over a moduli space of sections with
  \emph{no basepoints}, hence it is independent of $\epsilon.$ Summing
  over $M$, $\beta$, and $m$, we get \eqref{eqn:Recursion}.
\end{proof}

\begin{lem}\label{lem:Polynomiality}
  Define $(qe^{L_\rho})^\beta=q^\beta e^{\beta_\rho}.$ Then for any
  $\gamma\in\H(\theta)[[q]],$ the expression
  \begin{align*}
    D(\mathfrak{Z}_\mu^{\epsilon,\theta}):=\left(\mathfrak{Z}_\mu^{\epsilon,\theta}(t,qe^{YzL_\vartheta},\hbar)(\gamma)\right)\left(\mathfrak{Z}_{\upsilon(\mu)}^{\epsilon,\theta}(t,q,-\hbar)(\gamma)\right)
  \end{align*}
  is an element of $K[[q,Y,\hbar]].$
\end{lem}
\begin{proof}

  There is an important line bundle
  $U(\mathcal{L})_\theta^{\beta_1,\beta_2}$ on each quasimap graph
  space, defined in \cite{CiocanFontanineKim2014}. We define a
  modified version for LG-quasimap graph spaces.

  The line bundle $L_\vartheta$ on $[V/(G\times\C^*_R)]$ induces an
  embedding $[V/(G\times\C^*_R)]\into[\C^{N+1}/\C^*].$ Given an
  LG-graph quasimap $(C,u,\kappa,\tau)$ of degree $\beta$, composing
  gives a prestable graph quasimap $C\to\P^N$ of degree
  $\beta_\vartheta$. (The fact that this quasimap is prestable in the
  sense of \cite{CiocanFontanineKim2014} comes from the fact that
  $V^{ss}(\vartheta)=V^{ss}(\theta).$) Note that
  $\mathcal{L}_\vartheta=u^*L_\vartheta$ has trivial monodromy at
  every marked point of $C$, and indeed the map $C\to\P^N$ factors
  through the coarse moduli space $\bar{C}$ of $C$ by the definition
  of $\bar{C}$. Done in families, this construction yields a map
  $\LGQG_{0,m}^\epsilon(X(\theta),\beta)\mathcal{Q}(\beta_\vartheta)$,
  where $\mathcal{Q}(\beta_\vartheta)$ is a stack of prestable
  \emph{nonorbifold} graph quasimaps to $\P^N$ of degree
  $\beta_\vartheta$.

  The stack $\mathcal{Q}(\beta_\vartheta)$ has a forget-and-contract
  map to a stack $\mathcal{Q}'(\beta_\vartheta)$ as in Section 3 of
  \cite{CiocanFontanineKim2014}, remembering only the restriction of a
  quasimap to the parametrized component; all marked points are
  forgotten (possible since the orbifold structure has been removed),
  and nodes are replaced with basepoints of degree equal to the total
  degree of the line bundle ``on the other side'' of the node. The
  stack $\mathcal{Q}'(\beta_\vartheta)$ parametrizes sections of line
  bundles on $\P^1$, with no stability conditions --- in fact, it is a
  projective space. Denote by $U(\mathcal{L})_\vartheta$ the pullback
  to $\LGQG_{0,m}^\epsilon(X(\theta),\beta)$ of
  $\O_{\mathcal{Q}'(\beta_\vartheta)}(1).$

  Instead of forgetting the marked points, one may replace them with
  basepoints. Fix degrees $\beta(1)$ and $\beta(2).$ Write
  $\beta_\vartheta(1)$ and $\beta_\vartheta(1)$ for the corresponding
  integers as in Definition \ref{Def:EpsilonStability}. Then there is
  a map
  $$\Phi:\LGQG_{0,m}^\epsilon(X(\theta),\beta)\to\mathcal{Q}'(\beta_\vartheta+\beta(1)_\vartheta+\beta(2)_\vartheta),$$
  which as above sends $(C,u,\kappa)$ to a quasimap $\P^1\to\P^N$,
  with ``artificial'' basepoints added at $\tau(b_1)$ and $\tau(b_2).$
  We define
  $$U(\mathcal{L})_\vartheta^{\beta(1),\beta(2)}:=\Phi^*
  \O_{\mathcal{Q}'(\beta_\vartheta+\beta(1)_\vartheta+\beta(2)_\vartheta)}(1).$$

  Let
  $F_{\mu}(2+m,\beta)\subseteq\LGQG_{0,2+m}^\epsilon(X(\theta),\beta)^T$
  be the open and closed substack of $T$-fixed (but not necessarily
  $\C^*$-fixed) LG-quasimaps for which the parametrized component is
  contracted to $\mu$. It is $\C^*$-invariant but not $\C^*$-fixed, as
  there may be basepoints, nodes, and marked points mapped by $\tau$
  to $\P^1\setminus\{0,\infty\}.$
  Write $$\gamma=\sum_{\beta}q^\beta\gamma_\beta\in\H(\theta)[[q]]$$
  and consider the series of $T$-equivariant integrals:
  \begin{align}\label{eqn:Convolution1}
    &\sum_{m,\beta}\frac{q^{\beta-\beta_0(\theta,2+m)}}{m!}\sum_{\beta(1),\beta(2)}q^{\beta(1)}q^{\beta(2)}\\
    &\quad\int_{[F_\mu(2+m,\beta)]^{\vir}}\frac{e^{c_1(U(\mathcal{L})_\vartheta^{\beta(1),\beta(2)})Y}\ev_1^*(\tilde{\gamma}_{\beta(1)}\otimes
      p_0)\ev_2^*(\tilde{\gamma}_{\beta(2)}\otimes
      p_{\infty})\prod_{i=3}^{2+m}\ev_{i}^*(t)}{e^T(N^{\vir}_{F_\mu(2+m,\beta)})}\in
      K[[\hbar]].\nonumber
  \end{align}
  Since the denominator is a class in the $T$-equivariant cohomology of
  $F_\mu(2+m,\beta)$, it does not contain $\hbar$.

  We apply $\C^*$-localization to compute the integral. The
  contribution from a fixed component
  $F_{B_0,\beta^0,\mu}^{B_\infty,\beta^\infty}:=F_{B_0,\beta^0}^{B_\infty,\beta^\infty}\cap
  F_\mu(2+m,\beta)$
  is zero unless $b_1\in B_0$ and $b_2\in B_\infty.$ In this case,
  since we have seen that
  $e(N^{\vir}_{F_{B_0,\beta^0,\mu}^{B_\infty,\beta^\infty}|F_\mu(2+m,\beta)})=(-\hbar^2)(\hbar-\psi_\bullet)(-\hbar-\psi_{\check\bullet})$,
  we get the integral:
  \begin{align}\label{eqn:Convolution3}
    \int_{[F_{B_0,\beta^0}^{B_\infty,\beta^\infty}\cap F_\mu(2+m,\beta)]^{\vir}}\frac{e^{c_1(U(\mathcal{L})_\vartheta^{\beta(1),\beta(2)})Y}\ev_1^*(\tilde{\gamma}_{\beta(1)})\ev_2^*(\tilde{\gamma}_{\beta(2)})\prod_{i=3}^{2+m}\ev_{i}^*(t)}{e^T\left(N^{\vir}_{F_\mu(2+m,\beta)}\right)(\hbar-\psi_\bullet)(-\hbar-\psi_{\check\bullet})}.
  \end{align}
  As before we may write $F_{B_0,\beta^0,\mu}^{B_\infty,\beta^\infty}$
  as a (not fibered) product $M_0\times\hat{M}\times M_{\infty}$,
  where 
  \begin{align*}
    M_0&=(\LGQ_{0,\abs{B_0}+\bullet}^\epsilon(X(\theta),\beta^0))^T_\mu\\
    \hat{M}&=(\LGQG_{0,2}^\epsilon(X(\theta),\beta_0(\theta,2)))^{T\times\C^*}_\mu\\
    M_\infty&=(\LGQ_{0,\abs{B_\infty}+\check\bullet}^\epsilon(X(\theta),\beta^\infty))^T_\mu.
  \end{align*}
  Here the superscript $\mu$ refers only to components where the extra
  marked point (or, for $\hat{M}$, the entire curve $C$) is mapped to
  $\mu$.  $\hat{M}$ is a union of points, each corresponding to
  choices of monodromies.

  As in Lemma \ref{lem:Recursion}, we write \eqref{eqn:Convolution3}
  as a product of integrals over $M_0$ and $M_\infty$. Smoothing the nodes
  $\bullet$ and $\check\bullet$ shows
  \begin{align*}
    N^{\vir}_{F_\mu(2+m,\beta)}=N^{\vir}_{M_0}\oplus N^{\vir}_{M_\infty},
  \end{align*}
  where the normal bundles on the right are taken relative to the
  ambient spaces
  $\LGQ_{0,\abs{B_0}+\bullet}^\epsilon(X(\theta),\beta^0)$ and
  $\LGQ_{0,\abs{B_\infty}+\check\bullet}^\epsilon(X(\theta),\beta^\infty)$.
  The line bundle $U(\mathcal{L})_\vartheta^{\beta(1),\beta(2)}$ can
  be expressed on the product $M_0\times\hat{M}\times M_\infty$ as
  follows. 
  As the construction above involves restricting to the parametrized
  component, the map
  $M_0\times\hat{M}\times
  M_\infty\to\mathcal{Q}'(\beta_\vartheta+\beta(1)_\vartheta+\beta(2)_\vartheta)$
  is constant, so the restriction of
  $U(\mathcal{L})_\vartheta^{\beta(1),\beta(2)}|_{M_0\times\hat{M}\times
    M_\infty}$ is topologically trivial.

  We compute the $(T\times\C^*)$-weight as follows. Let
  $(C,u,\kappa)\in M_0\times\hat{M}\times M_\infty$, with
  $C\cong C_0\cup\hat{C}\cup C_\infty.$ Then $\Phi(C,u,\kappa)$ is a
  quasimap $\P^1\to\P^N,$ given in coordinates by
  \begin{align*}
    [s:t]\mapsto[a_0s^{\beta^0_\vartheta+\beta(1)_\vartheta}t^{\beta^\infty_\vartheta+\beta(2)_\vartheta}:\cdots:a_Ns^{\beta^0_\vartheta+\beta(1)_\vartheta}t^{\beta^\infty_\vartheta+\beta(2)_\vartheta}].
  \end{align*}
  Here the $a_i$s are determined by $\mu$. The weight of
  $U(\mathcal{L})_\vartheta^{\beta(1),\beta(2)}$ at $(C,u,\kappa)$ is
  equal to the weight of
  $\O_{\mathcal{Q}'(\beta_\vartheta+\beta(1)_\vartheta+\beta(2)_\vartheta}(1)$
  at $\Phi(C,u,\kappa)$. By the definition of the map $\Phi$, this is
  the $T$-weight of $L_\vartheta$ at $\mu$, denoted
  $w_{\mu,\vartheta}$. From the choice of coordinates in Section
  \ref{sec:CStarAction}, the $\C^*$-weight is
  $(\beta^0_\vartheta+\beta(1)_\vartheta)\hbar$.
  
  Now we may factor the integral \eqref{eqn:Convolution3} as:
  \begin{align*}
    e^{(w_{\mu,\vartheta})
      Y}&\left(\int_{[(\LGQ_{0,\abs{B_0}+\bullet}^\epsilon(X(\theta),\beta^0))^T_\mu]^{\vir}}\frac{e^{(\beta^0_\vartheta+\beta_\vartheta(1))\hbar Y}\ev_\bullet(\delta_\mu)\ev_1(\tilde{\gamma}_{\beta(1)})\prod_{i=2}^{\abs{B_0}}\ev_i^*(t)}{e^T(N^{\vir}_{(\LGQ_{0,\abs{B_0}+\bullet}^\epsilon(X(\theta),\beta^0))^T_\mu})(\hbar-\psi_\bullet)}\right)\\
    &\cdot\left(\int_{[(\LGQ_{0,\abs{B_\infty}+\check\bullet}^\epsilon(X(\theta),\beta^\infty))^T_{\upsilon(\mu)}]^{\vir}}\frac{\ev_{\check\bullet}(\upsilon^*\delta_\mu)\ev_1(\tilde{\gamma}_{\beta(2)})\prod_{i=2}^{\abs{B_\infty}}\ev_i^*(t)}{e^T(N^{\vir}_{(\LGQ_{0,\abs{B_\infty}+\check\bullet}^\epsilon(X(\theta),\beta^\infty))^T_{\upsilon(\mu)}})(-\hbar-\psi_{\check\bullet})}\right).
  \end{align*}
  For compactness, we write this as
  $e^{(w_{\mu,\vartheta})
    Y}\mathcal{S}(\abs{B_0})\mathcal{S}(\abs{B_\infty}).$
  Summing gives:
  \begin{align*}
    e^{(w_{\mu,\vartheta})
      Y}&\sum_{\substack{B_0,B_\infty,\\\beta^0,\beta^\infty,\\\beta(1),\beta(2)}}\frac{q^{\beta^0+\beta^\infty+\beta(1)+\beta(2)-\beta_0(\theta,\abs{B_0}+\abs{B_\infty})}}{(\abs{B_0}+\abs{B_\infty})!}\mathcal{S}(\abs{B_0})\mathcal{S}(\abs{B_\infty})\\
    &=e^{(w_{\mu,\vartheta})
      Y}\sum_{\substack{m_0,m_\infty,\\\beta^0,\beta^\infty,\\\beta(1),\beta(2)}}\frac{q^{\beta^0+\beta^\infty+\beta(1)+\beta(2)-\beta_0(\theta,m_0+m_\infty)}}{m_0!m_\infty!}\mathcal{S}(m_0)\mathcal{S}(m_\infty)\\
    &=e^{(w_{\mu,\vartheta})
      Y}\left(\sum_{m_0,\beta^0,\beta(1)}\frac{q^{\beta^0+\beta(1)-\beta_0(\theta,m_0+1)}}{m_0!}\mathcal{S}(m_0)\right)\left(\sum_{m_\infty,\beta^\infty,\beta(2)}\frac{q^{\beta^\infty+\beta(2)-\beta_0(\theta,m_\infty+1)}}{m_\infty!}\mathcal{S}(m_\infty)\right)\\
    &=e^{(w_{\mu,\vartheta})
      Y}\left(\sum_{\beta(1)}(qe^{Y\hbar L_\vartheta})^{\beta(1)}\mathfrak{Z}_\mu^{\epsilon,\theta}(t,qe^{Y\hbar L_\vartheta},\hbar)(\gamma_{\beta(1)})\right)\left(\sum_{\beta(2)}q^{\beta(2)}\mathfrak{Z}_{\upsilon(\mu)}^{\epsilon,\theta}(t,q,-\hbar)(\gamma_{\beta(2)})\right)\\
    &=e^{(w_{\mu,\vartheta})
      Y}\left(\mathfrak{Z}_\mu^{\epsilon,\theta}(t,qe^{Y\hbar L_\vartheta},\hbar)(\gamma)\right)\left(\mathfrak{Z}_{\upsilon(\mu)}^{\epsilon,\theta}(t,q,-\hbar)(\gamma)\right).\qedhere
  \end{align*}
\end{proof}

\bigskip

We have now assembled all of the necessary pieces to prove our mirror
theorem.
\begin{proof}[Proof of Theorem \ref{thm:MirrorTheorem}]
  The theorem now follows from Uniqueness Lemma 7.7.1 of
  \cite{CiocanFontanineKim2014}, applied to the systems:
  \begin{gather*}
    \{\mathfrak{Z}_\mu^{\epsilon,\theta}(t,q,\hbar)(1_\theta),\mu\in \bar{I}X(\theta)^T\}\\
    \{\mathfrak{Z}_\mu^{\infty,\theta}(\mathscr{T}^{\epsilon,\theta}(t),q,\hbar)(1_\theta),\mu\in
    \bar{I}X(\theta)^T\}.
  \end{gather*}
  (As in Section 3.7.3, Item (3) of
  \cite{CheongCiocanFontanineKim2015}, we modify Condition (5) of the
  Uniqueness Lemma slightly.) In particular, the Uniqueness Lemma
  requires five properties to hold, and they are verified in:
  \begin{enumerate}
  \item Lemma \ref{lem:SMuRationalFunction},
  \item Lemma \ref{lem:Recursion},
  \item Lemma \ref{lem:Polynomiality},
  \item Observation \ref{obs:AgreeToFirstOrder}, and
  \item Remark \ref{rem:JIndependentOfEpsilon}.\qedhere
  \end{enumerate}
\end{proof}

\section{Calculating the $I$-functions}\label{sec:CalculatingI}
In this section, we compute $I^\theta(q,\hbar):=J^{0+,\theta}(0,q,\hbar)$
for any $\theta\in\Theta$. The following three observations allow
explicit computations.
\begin{observation}\label{obs:GraphSpaceProper}
  $\LGQG_{0,1}^{0+}(X(\theta),\beta)$ is proper. To see this, consider
  $(C,\mathcal{L},\sigma)=(C,u,\kappa)\in\LGQG_{0,1}^{0+}(X(\theta),\beta).$
  If $x$ is a superscript variable, then $\beta_x\ge0.$ Hence the
  bundle
  $\mathcal{L}_{p_x}\cong\mathcal{L}_x^{-3}\otimes\omega_{C,\log}$ has
  negative degree, since $\deg\omega_{C,\log}=-2+1=-1<0.$ If $x$ is a
  subscript variable, we saw in Proposition \ref{lem:Concave} that
  $\mathcal{L}_x$ and $\mathcal{L}_x\otimes\mathcal{L}_a^*$ have no
  global sections. From these, properness follows by a standard
  argument.

  Alternatively, one may show that $\LGQG_{0,1}^{0+}(X(\theta),\beta)$
  is isomorphic to $\LGQG_{0,1}^{0+}(Z'(\theta),\beta)$, where
  $Z'(\theta')\cong[((\P^2)^2\times B\mu_3)/\mu_3]$ is the critical
  locus of a polynomial $W'$ inside a quotient $X'(\theta').$ Then
  Theorem \ref{thm:DMStack} asserts that
  $\LGQG_{0,1}^{0+}(Z'(\theta),\beta)$ is proper.
\end{observation}
\begin{observation}\label{obs:UniversalCurveTrivial}
  The universal curve $U_\beta'$ over the distinguished fixed part
  $F_\beta'$ (see Definition \ref{Def:FBetaPrime}) is trivial, with
  fibers canonically isomorphic to $\P_{3,1}$, as follows. Recall that
  $F_\beta'$ parametrizes LG-quasimaps $(C,u,\kappa)$ where $C$ has a
  single marking $b_1$ with $\tau(b_1)=\infty\in\P^1$. Further, the
  degree $\beta$ is concentrated at $\tau^{-1}(0).$ The
  $\epsilon$-stability condition implies that $\tau^{-1}(0)$ is a
  single point, and $u$ has a basepoint of degree $\beta$ there. All
  fibers are canonically identified with $\P_{3,1},$ so
  $U_\beta'\to F_\beta'$ is a trivial family. We write $\varpi$ for
  the projection $U_\beta'\to\P_{3,1}$.
\end{observation}
\begin{observation}\label{obs:VectorBundles}
  For each summand $\mathcal{L}_\rho$ of $\mathcal{E},$ at least one
  of $H^0(C,\mathcal{L}_\rho)$ and $H^1(C,\mathcal{L}_\rho)$ vanishes,
  since $C\cong\P_{3,1}$ by Observation
  \ref{obs:UniversalCurveTrivial}. This implies that
  $R^\bullet\pi_*\mathcal{E}=\bigoplus_{\rho\in\mathbf{R}}R^\bullet\pi_*\mathcal{L}_\rho$
  is a complex of \emph{vector bundles}. A basic property of virtual
  fundamental classes (\cite{BehrendFantechi1997}, Proposition 5.6)
  now states that $[F_\beta']^{\vir}=e((R^1\pi_*\mathcal{E})^{\C^*})$.
\end{observation}

By definition, $I^\theta(q,\hbar)$ is the contribution to the
equivariant integral
\begin{align*}
  \sum_{\beta_j}q^{\beta-\beta_0(\theta,1)}\gamma_j\langle\gamma^j\ev_1^*[\infty]\rangle^{0+,\theta,Gr}_{1,\beta}
\end{align*}
coming from the loci $F_\beta$ of Definition \ref{Def:FBeta}. By the
projection formula, this is the contribution from the loci $F_\beta'$
to:
\begin{align}\label{eqn:IFunctionIntegral}
  \sum_{\beta,j}q^{\beta-\beta_0(\theta,1)}\iota^*(\tilde{\gamma_j})\int_{[\LGQG_{0,1}^{0+}(X(\theta),\beta)]^{\vir}}\ev_1^*(\tilde{\gamma^j}\otimes[\infty]),
\end{align}
with $\iota^*(\tilde{\gamma_j})=\gamma_j,$
$\iota^*(\tilde{\gamma^j})=\gamma^j,$ and
$$\langle\underline{Z}\cup\tilde{\gamma_j},\tilde{\gamma^{j'}}\rangle=\delta_j^{j'}.$$

We may choose isomorphisms of $\mathcal{L}_x,$ $\mathcal{L}_y,$
$\mathcal{L}_z,$ and $\mathcal{L}_a$ with the line bundles
$\O_{\P_{3,1}}(\beta_x),$ $\O_{\P_{3,1}}(\beta_y),$
$\O_{\P_{3,1}}(\beta_z),$ $\O_{\P_{3,1}}(\beta_a)$, and write $\sigma$
as a tuple of sections
\begin{align*}
  (\sigma_{x_0}(s,t),\sigma_{x_1}(s,t),\sigma_{x_2}(s,t),\sigma_{y_0}(s,t),\ldots,\sigma_{p_y}(s,t),\sigma_{p_z}(s,t)),
\end{align*}
where the entries are homogeneous polynomials in $s$ and $t$ of the
appropriate degrees, and the degrees of $s$ and $t$ are 3 and 1,
respectively. The fact that $\sigma$ is $\C^*$-fixed implies that
$\sigma$ is of the form
\begin{align}\label{eqn:ExplicitSection}
  \sigma=(x_0s^{\beta_x-\beta_a},x_1s^{\beta_x},x_2s^{\beta_x},y_0s^{\beta_y-\beta_a},\ldots,p_ys^{-3\beta_y-1},p_zs^{-3\beta_z-1}).
\end{align}
In particular, this shows
\begin{prop}\label{prop:IdentifyFixedLocusWithZ}
  Fix $\beta$ so that $(\beta,1)$ is $\theta$-effective. The map
  $\ev_1:F_\beta'\to\bar{I}X(\theta)$ is an embedding.
\end{prop}
\begin{Def}\label{Def:DegreeDeterminesMonodromy}
  On $\P_{3,1}$, a line bundle $\mathcal{L}$ is determined up to
  isomorphism by its degree $\beta$; the fractional part
  $\langle\beta\rangle$ determines the multiplicity
  $\mult_\infty(\mathcal{L})$ at the orbifold point. In turn,
  $\mult_\infty(\mathcal{L})$ determines a component of
  $\bar{I}X(\theta)$, which we denote
  $X(\theta)_{\langle{\beta}\rangle}$, such that $\ev_1$ factors
  through $X(\theta)_{\langle{\beta}\rangle}\into\bar{I}X(\theta).$ We
  denote the fundamental class of this sector by
  $1_{\langle\beta\rangle}.$
\end{Def}
\begin{lem}\label{lem:ImageOfEv1}
  The image of $\ev_1:F_\beta'\into X(\theta)_{\langle\beta\rangle}$ is the substack
  of $X_R(\theta)_{\langle\beta\rangle}:=X(\theta)_{\langle\beta\rangle}\cap X_R(\theta)$ cut out by
  the vanishing of $x_0,$ if $x$ is a superscript variable and
  $\beta_x-\beta_a\in\Z_{<0}.$ (Similarly cut out by the vanishing of
  $y_0$ and $z_0$.)
\end{lem}
\begin{proof}
  An entry of $\sigma$ in \eqref{eqn:ExplicitSection} is necessarily
  zero either of the following holds:
  \begin{enumerate}
  \item The corresponding line bundle $\mathcal{L}_\rho$ has degree
    $\beta_\rho\not\in\Z.$ In this case $s^{\beta_\rho}$ does not make
    sense.
  \item $\mathcal{L}_\rho$ has degree $\beta_\rho\in\Z_{<0}.$ In this
    case $\mathcal{L}_\rho$ has no nonzero global sections.
  \end{enumerate}
  The first case imposes no restriction on
  $X(\theta)_{\langle\beta\rangle}$. In other words, if
  $\beta_\rho\not\in\Z,$ then the the corresponding coordinate
  vanishes on $X(\theta)_{\langle\beta\rangle}$.

  The second case does impose restrictions. First, if $x$ is a
  superscript variable, we must have $\sigma_{p_x}=0.$ (In fact, we
  already observed that $u$ factors through $X_R(\theta).$) This shows
  that $\Im(\ev_1)\subseteq X_R(\theta)_{\langle\beta\rangle}$. Also,
  if $\beta_x-\beta_a\in\Z_{<0},$ then $\sigma_{x_0}=0.$
\end{proof}
\begin{prop}
  If $x$ is a subscript variable, only terms of $I^\theta(q,\hbar)$ with
  $\beta_x-\beta_0(\theta,1)\in\frac{1}{3}\Z_{<0}\setminus\Z$ are
  nonzero. If $x$ is a superscript variable, only terms of
  $I^\theta(q,\hbar)$ with $\beta_x\in\Z_{\ge0}$ are nonzero.
\end{prop}
\begin{proof}
  The first claim is immediate from $\beta_x-\beta_0(\theta,1)<0$
  (Section \ref{sec:EffectiveDegrees}) and Remark
  \ref{rem:BroadVanish}. For the second, Definition
  \ref{Def:LGQuasimapToX} implies that $\beta_x\ge0.$ If
  $\beta_x\not\in\Z,$ by the fact that $\sigma$ does not vanish at
  $\infty\in\P_{3,1},$ we have $x_1=x_2=0$ in
  \eqref{eqn:ExplicitSection}. In particular, the sector
  $X_R(\theta)_{\langle\beta\rangle}$ is supported over the locus
  $x_1=x_2=0.$ Thus for any nonzero term $c\gamma_jq^\beta$ of
  $I^\theta(q,\hbar)$, where $c$ is a scalar and $\beta_x\not\in\Z,$ we
  must have $\gamma_j\in\iota^*(H^*(X_R(\theta)_{\beta^{-1}}))=0.$
\end{proof}

This calculation is adapted from \cite{CheongCiocanFontanineKim2015}.
\begin{prop}\label{prop:VirtualNormalBundle}
  The virtual normal bundle to $F_\beta'$ inside
  $\LGQG_{0,1}^{0+}(X(\theta),\beta)$ has $\C^*$-equivariant
  Euler class
  \begin{align*}
    e^{\C^*}(N_{F_\beta'}^{\vir})=(-\hbar)\frac{\prod_{\substack{\rho\in\mathbf{R}\\\beta_\rho\ge0}}\prod_{0\le\nu\le\ceil{\beta_\rho}-1}(\ev_1^*D_\rho+(\beta_\rho-\nu)\hbar)}{\prod_{\substack{\rho\in\mathbf{R}\\\beta_\rho<0}}\prod_{\floor{\beta_\rho}+1\le\nu\le-1}(\ev_1^*D_\rho+(\beta_\rho-\nu)\hbar)}.
  \end{align*}
  where the divisors $D_\rho$ were defined in Section
  \ref{sec:ToricDivisors}.
\end{prop}
\begin{proof}
  The factor $(-\hbar)$ is from moving the marked point on $C$, and the
  rest comes from the relative perfect obstruction theory
  $R^\bullet\pi_*\mathcal{E}$. By Observation \ref{obs:VectorBundles},
  each summand $\mathcal{L}_\rho$ of $\mathcal{E}$ contributes either
  $(\pi_*\mathcal{L})^{\mov}$ or $-(R^1\pi_*\mathcal{L})^{\mov}$ to
  the virtual normal bundle, whichever is nonzero, where the
  superscript `$\mov$' denotes the $\C^*$-moving invariant subbundle.

  We see from the form of \eqref{eqn:ExplicitSection} that for all
  $\rho\in\mathbf{R}$ with $\beta_\rho\in\Z_{\ge0},$ we have
  \begin{align}\label{eqn:LineBundlesOnFBeta}
    \mathcal{L}_\rho\cong\pi^*\ev_1^*L_\rho\otimes\varpi^*\O_{\P_{3,1}}(3\beta_\rho).
  \end{align}
  (Recall that $\mathcal{L}_\rho=u^*L_\rho$ and that
  $\varpi:U_\beta'\to\P_{3,1}$ is the projection.)

  \medskip
  
  \noindent\textit{Claim.} Equation \eqref{eqn:LineBundlesOnFBeta}
  holds, at least up to torsion, for all
  $\rho\in\mathbf{R}$.
  \begin{proof}[Proof of Claim]
    We check separately for each $\theta\in\Theta$. For all $\theta$,
    $\{a=0\}\subseteq V^{\uns}(\theta),$ so Definition
    \ref{Def:LGQuasimapToX} implies that
    $3\beta_a\in\Z_{\ge0}$. Up to torsion, this identifies
    \begin{align}\label{eqn:AlsoForA}
      \mathcal{L}_a\cong\pi^*\ev_1^*L_{\hat{t_a}}\otimes\varpi^*\O_{\P_{3,1}}(3\beta_a).
    \end{align}

    For $\theta=\theta^{xyza},$ at least one of
    $\beta_x-\beta_a$ and $\beta_x$ is a nonnegative integer. By
    commutativity of tensor products and pullbacks, the fact that
    Equation \eqref{eqn:LineBundlesOnFBeta} holds for one of
    $\mathcal{L}_x$ and $\mathcal{L}_x\otimes\mathcal{L}_a^*,$
    together with Equation \eqref{eqn:AlsoForA}, implies that Equation
    \eqref{eqn:LineBundlesOnFBeta} holds for the other. A similar
    argument works for $y$ and $z$.

    It remains to check Equation \eqref{eqn:LineBundlesOnFBeta} for
    $\rho=-3\hat{t_x}+\hat{t_R}$. Since
    $\mathcal{L}_R\cong\omega_{U_\beta'/F_\beta',\log}$, the
    triviality of $U_\beta'$ over $F_\beta'$ implies
    $\mathcal{L}_{\hat{t_R}}\cong\varpi^*\O_{\P_{3,1}}(-3).$
    
    For $\theta=\theta_{xyz}^a,$ we must have
    $\beta_{-3\hat{t_x}+\hat{t_R}}=-3\beta_x+\beta_{\hat{t_R}}\in\Z_{\ge0}.$
    Therefore Equation \eqref{eqn:LineBundlesOnFBeta} implies
    \begin{align*}
    \mathcal{L}_{-3\hat{t_x}+\hat{t_R}}&\cong\pi^*\ev_1^*L_{-3\hat{t_x}}\otimes\varpi^*\O_{\P_{3,1}}(-9\beta_x+3\beta_{\hat{R}})\\
&=\pi^*\ev_1^*L_{-3\hat{t_x}}\otimes\varpi^*\O_{\P_{3,1}}(-9\beta_x-3).
    \end{align*}
    Again we have
    $\mathcal{L}_{\hat{t_R}}\cong\varpi^*\O_{\P_{3,1}}(-3)$, so
\begin{align*}
  \mathcal{L}_{-3\hat{t_x}}=\mathcal{L}_x^{\otimes-3}\cong\pi^*\ev_1^*L_{-3\hat{t_x}}\otimes\varpi^*\O_{\P_{3,1}}(-9\beta_x).
    \end{align*}
    Thus up to torsion this identifies
    $\mathcal{L}_x\cong\pi^*\ev_1^*L_{\hat{t_x}}\otimes\varpi^*\O_{\P_{3,1}}(3\beta_x).$
    The argument used for $\theta^{xyza}$ to describe $\mathcal{L}_a$
    applies here to to show that Equation
    \eqref{eqn:LineBundlesOnFBeta} holds for
    $\mathcal{L}_{\hat{t_x}-\hat{t_a}}.$ The same argument works for
    the characters $\hat{t_y}-\hat{t_a},$ $\hat{t_y},$
    $\hat{t_z}-\hat{t_a},$ and $\hat{t_z}.$ These two arguments
    together prove the claim for $\theta^{xya}_z$ and
    $\theta_{yz}^{xa}$ also.
  \end{proof}

  Now, by the projection formula,
  \begin{align*}
    R^i\pi_*(\mathcal{E})&=\bigoplus_{\substack{\rho\in\mathbf{R}}}R^i\pi_*(\pi^*\ev_1^*L_\rho\otimes\varpi^*\O_{\P_{3,1}}(3\beta_\rho))\\
               &=\bigoplus_{\substack{\rho\in\mathbf{R}}}\ev_1^*L_\rho\otimes
    R^i\pi_*\varpi^*\O_{\P_{3,1}}(3\beta_\rho).
  \end{align*}
  Now $R^i\pi_*\varpi^*\O_{\P_{3,1}}(3\beta_\rho)$ is trivial with
  fiber $H^i(\P_{3,1},\O_{\P_{3,1}}(3\beta_\rho))$, i.e.
  \begin{align*}
    R^i\pi_*(\mathcal{E})&=\bigoplus_{\substack{\rho\in\mathbf{R}}}\ev_1^*L_\rho\otimes
    H^i(\P_{3,1},\O_{\P_{3,1}}(3\beta_\rho)).
  \end{align*}
  The $\C^*$-action on $\LGQG_{0,1}^{0+}(Z(\theta),\beta)$ is induced
  from an action on $\P_{3,1}$ via the universal map
  $\tau:\UQG_{0,1}^{0+}(Z(\theta),\beta)\to\P_{3,1}$. Restricting to
  $F_\beta'$ shows that the action on $R^i\pi_*(\mathcal{E})$ is
  induced by the projection $\varpi:U_\beta'\to\P_{3,1}$. Thus the
  $\C^*$-action on each factor
  $\ev_1^*L_\rho\otimes H^i(\P_{3,1},\O_{\P_{3,1}}(3\beta_\rho))$ is
  the natural action on $H^i(\P_{3,1},\O_{\P_{3,1}}(3\beta_\rho))$.

  As these groups are identified with the tangent and obstruction
  spaces at the point $\sigma$ of Equation
  \eqref{eqn:ExplicitSection}, the natural $\C^*$-action on a section
  $s^at^b\in H^i(\P_{3,1},\O_{\P_{3,1}}(3\beta_\rho))$ has weight
  $b/3.$ The sections of $H^i(\P_{3,1},\O_{\P_{3,1}}(3\beta_\rho))$
  are, explicitly,
  \begin{align*}
    H^0(\P_{3,1},\O_{\P_{3,1}}(3\beta_\rho))&=
                                              \begin{cases}
                                                \C\{t^{3\beta_\rho},st^{3(\beta_\rho-1)}\ldots,s^{\floor{\beta_\rho}}t^{3\langle\beta_\rho\rangle}\}&\beta_\rho\ge0\\
                                                0&\beta_\rho<0
                                              \end{cases}
\\
    H^1(\P_{3,1},\O_{\P_{3,1}}(3\beta_\rho))&=
                                              \begin{cases}
                                                0&\beta_\rho\ge-1\\
                                                \C\{s^{-1}t^{3(\beta_\rho+1)},s^{-2}t^{3(\beta_\rho+2)}\ldots,s^{\floor{\beta_\rho}+1}t^{3(\langle\beta_\rho\rangle-1)}\}&\beta_\rho<-1
                                              \end{cases}
.
  \end{align*}
  Recalling the notation $D_\rho$ of Section \ref{sec:ToricDivisors},
  we have
  \begin{align}
    e_{\C^*}(R^0\pi_*(\mathcal{E}))&=\prod_{\substack{\rho\in\mathbf{R}\\\beta_\rho\ge0}}\prod_{\substack{0\le\nu\le\floor{\beta_\rho}\\\nu\in\Z}}(\ev_1^*D_\rho+(\beta_\rho-\nu)\hbar)\label{eqn:EqEuler}\\
    e_{\C^*}(R^1\pi_*(\mathcal{E}))&=\prod_{\substack{\rho\in\mathbf{R}\\\beta_\rho<-1}}\prod_{\substack{\floor{\beta_\rho}+1\le\nu\le-1\\\nu\in\Z}}(\ev_1^*D_\rho+(\beta_\rho-\nu)\hbar)\nonumber\\&\nonumber\\\label{eqn:EulerClassNormalBundle}
    e_{\C^*}(N_{F_\beta'}^{\vir})&=(-\hbar)\frac{\prod_{\substack{\rho\in\mathbf{R}\\\beta_\rho\ge0}}\prod_{0\le\nu\le\ceil{\beta_\rho}-1}(\ev_1^*D_\rho+(\beta_\rho-\nu)\hbar)}{\prod_{\substack{\rho\in\mathbf{R}\\\beta_\rho<-1}}\prod_{\floor{\beta_\rho}+1\le\nu\le
    -1}(\ev_1^*D_\rho+(\beta_\rho-\nu)\hbar)}
  \end{align}
\end{proof}
\begin{observation}
  The calculation above also shows that the obstruction bundle
  $R^1\pi_*\mathcal{E}$ has no $\C^*$-fixed part, i.e.
  $[F_\beta']^{\vir}=[F_\beta'].$
  \end{observation}
\begin{rem}
  For $\rho$ such that $\beta_\rho\in\Z,$ the section
  $s^{\floor{\beta_\rho}}t^{3\langle\beta_\rho\rangle}=s^{\beta_\rho}$
  is $\C^*$-fixed, and thus is not part of the virtual normal
  bundle. This explains the difference in indexing between
  \eqref{eqn:EqEuler} and \eqref{eqn:EulerClassNormalBundle}. The
  missing terms span the tangent space to $F_\beta'$.
\end{rem}
\begin{rem}\label{rem:BetaAGreaterThanBetaX}
  For a fixed $\theta$, the conditions of Definition
  \ref{Def:LGQuasimapToX} (with $m=1$) determine the signs of
  $\beta_x,$ $\beta_y,$ $\beta_z,$ and $\beta_a.$ This determines the
  signs of $\beta_\rho$ for all $\rho\in\mathbf{R}$ \emph{except} for
  $\rho\in\{\hat{t_x}-\hat{t_a},\hat{t_y}-\hat{t_a},\hat{t_z}-\hat{t_a}\}.$
  Specifically, if $\beta_x\ge0,$ the quantity $\beta_x-\beta_a$
  changes sign depending on whether $\beta_x\ge\beta_a$ or
  $\beta_x<\beta_a.$ (Similarly for $y$ and $z$.) Therefore, in view
  of Proposition \ref{prop:VirtualNormalBundle}, we will have to treat
  the case $0\le\beta_x<\beta_a$ separately in what follows. (See
  Lemma \ref{lem:BetaAGreaterThanBetaX}.)
\end{rem}
\begin{prop}\label{prop:IFunctionCalculation}
  \begin{align*}
    \iota^R_*(I^\theta(q,\hbar))&=\sum_{\beta}q^{\beta-\beta_0(\theta,1)}\frac{\prod_{\substack{\rho\in\mathbf{R}\\\beta_\rho<-1}}\prod_{\floor{\beta_\rho}+1\le\nu\le-1}(D_\rho+(\beta_\rho-\nu)\hbar)}{\prod_{\substack{\rho\in\mathbf{R}\\\beta_\rho\ge0}}\prod_{0\le\nu\le\ceil{\beta_\rho}-1}(D_\rho+(\beta_\rho-\nu)\hbar)}A_xA_yA_z1_{{\langle-\beta\rangle}},
  \end{align*}
  where $\iota_*^R$ is the embedding
  $Z(\theta)\into X_R^{\rig}(\theta)$,
  \begin{align*}
    A_x=
    \begin{cases}
      D_x-D_a&0\le\beta_x<\beta_a,\quad\beta_x-\beta_a\in\Z\\
      1&\mbox{otherwise},
    \end{cases}
  \end{align*}
  and similarly for $A_y$ and $A_z$.
\end{prop}
\begin{proof}
  Write $e_{\C^*}(N_{F_\beta'}^{\vir})=(-\hbar)\ev_1^*\alpha$ from
  \eqref{eqn:EulerClassNormalBundle}. The projection formula gives
  \begin{align*}
    \iota_*\left(\frac{1}{e_{\C^*}(N_{F_\beta'}^{\vir})}\cap[F_\beta']\right)=\frac{1}{-\alpha
    \hbar}\cap(\ev_1)_*[F_\beta']\in H_*(\bar{I}X(\theta)),
  \end{align*}
  so we have
  \begin{align*}
    \underline{Z}\sum_{\beta,j}q^{\beta-\beta_0(\theta,1)}\tilde{\gamma_j}\int_{[F_\beta']^{\vir}}\frac{\ev_1^*(\tilde{\gamma^j}\otimes[\infty])}{e^{\C^*}(N_{F_\beta'}^{\vir})}&=\underline{Z}\sum_{\beta,j}q^{\beta-\beta_0(\theta,1)}\tilde{\gamma_j}\int_{(\ev_1)_*[F_\beta']}\frac{(-\hbar)\tilde{\gamma^j}}{-\alpha \hbar}\\
                                                                                                                                                                   &=\sum_{\beta,j}q^{\beta-\beta_0(\theta,1)}(\underline{Z}\cup\tilde{\gamma_j})\int_{[X_R^{\rig}(\theta)]}\frac{\tilde{\gamma^j}}{\alpha}\cup
                                                                                                                                                                     (A_xA_yA_z1_{\langle\beta\rangle})\\
                                                                                                                                                                   &=\sum_{\beta}q^{\beta-\beta_0(\theta,1)}\upsilon^*\left(\frac{A_xA_yA_z1_{\langle\beta\rangle}}{\alpha}\right)\\
                                                                                                                                                                   &=\sum_\beta q^{\beta-\beta_0(\theta,1)}\frac{A_xA_yA_z1_{\langle-\beta\rangle}}{\alpha}.
  \end{align*}
  The second equality follows from Lemma \ref{lem:ImageOfEv1} and the
  last equality follows from the fact that $\alpha$ and $A_xA_yA_z$
  are classes on the untwisted sector of $\bar{I}X(\theta).$
\end{proof}
\begin{Def}
  Write
  \begin{align*}
    \iota_*I^\theta(q,\hbar)=
    \sum_\beta q^\beta I^\theta_\beta(\hbar).
  \end{align*}
  The \emph{small Givental
    $I^\theta$-function} $I^{\theta,\Giv}(q,\hbar)$ is defined to be
  \begin{align}\label{eqn:IGiv}
    I^{\theta,\Giv}(q,\hbar)=\sum_{\beta}q^{\beta+\frac{1}{\hbar}(H_x,H_y,H_z,H_a)}(-1)^{3\beta_x+3\beta_y+3\beta_z}I^\theta_\beta(\hbar).
  \end{align}
  By Section \ref{sec:ToricDivisors} we have $H_a=0$, and $H_x=0$ if
  $x$ is a subscript variable. (Similarly for $y$ and $z$.)
\end{Def}
In particular, Proposition \ref{prop:IFunctionCalculation} gives
\begin{align*}
  I^{\theta,\Giv}(q,\hbar)&=\sum_{\beta}q_x^{\beta_x+H_x/\hbar}q_y^{\beta_y+H_y/\hbar}q_z^{\beta_z+H_z/\hbar}q_a^{\beta_a}(-1)^{3\beta_x+3\beta_y+3\beta_z}\\
                      &\quad\quad\quad\quad\cdot\left(\frac{\prod_{\substack{\rho\in\mathbf{R}\\\beta_\rho<-1}}\prod_{\floor{\beta_\rho}+1\le\nu\le-1}(D_\rho+(\beta_\rho-\nu)\hbar)}{\prod_{\substack{\rho\in\mathbf{R}\\\beta_\rho\ge0}}\prod_{0\le\nu\le\ceil{\beta_\rho}-1}(D_\rho+(\beta_\rho-\nu)\hbar)}\right)A_xA_yA_z1_{{\langle\beta\rangle}}.
\end{align*}
\begin{rem}
  In \cite{CheongCiocanFontanineKim2015}, there is defined a \emph{big
    $I$-function} $\mathbb{I}(t,q,\hbar)$, also on the Lagrangian
  cone, where $t$ is restricted to the untwisted sectors of
  $\H(\theta).$ We may mimic their construction with no
  modification. $I^{\theta,\Giv}(q,\hbar)$ is obtained from the result
  by restricting $t$ further to only untwisted degree 2 classes,
  adding the factor $(-1)^{3\beta_x+3\beta_y+3\beta_z}$, and finally
  by identifying $q=e^t.$

  This identification seems mysterious, especially as the symbol
  $q^{H_x/\hbar}$ is otherwise meaningless. In fact, the identification
  arises naturally from the divisor equation in Gromov-Witten theory,
  see Remark 3.1.2 of \cite{ChiodoRuan2010}. It is an important part
  of the formal analytic continuation of Section
  \ref{sec:AnalyticContinuation}.

  The choice of sign in \eqref{eqn:IGiv} comes from
  \cite{ChangLi2011}, in which degree $d$ Gromov-Witten invariants of
  a quintic threefold \emph{with fields} (which play a similar role to
  the LG-quasimap invariants we use) differ from the usual
  Gromov-Witten invariants of the quintic threefold by the sign
  $(-1)^{5d+1}.$
\end{rem}

\section{Relating the generating functions of the different quotients}\label{sec:RelatingIFunctions}

In this section, we use a similar method to that in
\cite{ChiodoRuan2010} to relate the $I$-functions
$I^{\theta,\Giv}(q,\hbar)$ for various $\theta.$ The method involves
analytic continuation of the $\Gamma$-function, and to make sense of
this we set up some minor formalism.

\subsection{The $\Gamma$-function on $\C\times\C$}\label{sec:ExtendedGamma}
\begin{Def}
  For $s+\xi\in\C\times\C,$ we define the \emph{extended
    $\Gamma$-function} $\tilde{\Gamma}:\C\times\C\to\C[[\xi]]$ by
  \begin{align*}
    \tilde{\Gamma}(s+\xi):=\sum_{k=0}^\infty\frac{\Gamma^{(k)}(s)}{k!}\xi^k.
  \end{align*}
  Intuitively, we take $\xi$ to be an extremely small complex number.
\end{Def}
\begin{observation}
  It is an easy exercise that $\tilde{\Gamma}$ satisfies the
  functional equation
  $\tilde{\Gamma}(s+\xi)=(s-1+\xi)\tilde{\Gamma}(s-1+\xi).$ This
  agrees with the intuition that $s+\xi$ is a complex number ``near''
  $s$.
\end{observation}
It is easy to extend this to $\C\times\C^n,$ and we get a map
$\tilde{\Gamma}$ to $\C[[\xi_1,\ldots,\xi_n]].$ In the next section,
we use the functional equation to rewrite $I^{\theta,\Giv}(t,q,\hbar)$
in terms of $\tilde{\Gamma}$. This will allow us to carry out (formal)
analytic continuation.
\begin{notation}
  Hereafter we drop the tilde from $\tilde{\Gamma}$.
\end{notation}

\subsection{Analytic continuation}\label{sec:AnalyticContinuation}
First, we note that the seeming inconsistency in Remark
\ref{rem:BetaAGreaterThanBetaX} can be conveniently ignored, as
follows.
\begin{lem}\label{lem:BetaAGreaterThanBetaX}
  For $\beta_x\ge0,$ the factor in $I^{\theta,\Giv}(t,q,\hbar)$ corresponding to
  $\rho_{x_0}=\hat{t_x}-\hat{t_a}$ is equal to
  \begin{align*}
    \hbar^{-\ceil{\beta_x-\beta_a}}\frac{\Gamma(\ev_1^*D_\rho/\hbar+\beta_x-\beta_a-\ceil{\beta_x-\beta_a}+1)}{\Gamma(1+\ev_1^*D_\rho/\hbar+\beta_x-\beta_a)}.
  \end{align*}
  In particular, this holds whether $\beta_x<\beta_a$ or
  $\beta_x\ge\beta_a.$
\end{lem}
\begin{proof}
  If $\beta_x\ge\beta_a,$ we have $A_x=1,$ and the corresponding
  factor in $I^{\theta,\Giv}(t,q,\hbar)$ is by definition
  \begin{align*}
    \frac{A_x}{\prod_{0\le\nu\le\ceil{\beta_x-\beta_a}-1}(\ev_1^*D_\rho+(\beta_x-\beta_a-\nu)\hbar)}&=\hbar^{-\ceil{\beta_x-\beta_a}}\frac{1}{\prod_{0\le\nu\le\ceil{\beta_x-\beta_a}-1}(\ev_1^*D_\rho/\hbar+\beta_x-\beta_a-\nu)}.
  \end{align*}
  Formally using
  $\prod_{0\le\nu\le
    k}(\alpha-\nu)=\frac{\Gamma(1+\alpha)}{\Gamma(\alpha-k)}$,
  we can write this as
  \begin{align*}
    \hbar^{-\ceil{\beta_x-\beta_a}}\frac{1}{\frac{\Gamma(1+\ev_1^*D_\rho/\hbar+\beta_x-\beta_a)}{\Gamma(\ev_1^*D_\rho/\hbar+\beta_x-\beta_a-\ceil{\beta_x-\beta_a}+1)}}=\hbar^{-\ceil{\beta_x-\beta_a}}\frac{\Gamma(\ev_1^*D_\rho/\hbar+\beta_x-\beta_a-\ceil{\beta_x-\beta_a}+1)}{\Gamma(1+\ev_1^*D_\rho/\hbar+\beta_x-\beta_a)}.
  \end{align*}
Meanwhile, if $\beta_x<\beta_a,$ we have $A_x=D_\rho$, and the
corresponding factor in $I^{\theta,\Giv}(t,q,\hbar)$ is
  \begin{align*}
    &\left(\prod_{\floor{\beta_x-\beta_a}+1\le\nu\le-1}(\ev_1^*D_\rho+(\beta_x-\beta_a-\nu)\hbar)\right)A_x\\
          &\quad=\prod_{\ceil{\beta_x-\beta_a}\le\nu\le-1}(\ev_1^*D_\rho+(\beta_x-\beta_a-\nu)\hbar)\\
          &\quad=\hbar^{-\ceil{\beta_x-\beta_a}}\frac{\Gamma(1+\ev_1^*D_\rho/\hbar+\beta_x-\beta_a-\ceil{\beta_x-\beta_a})}{\Gamma(\ev_1^*D_\rho/\hbar+\beta_x-\beta_a+1)}.
  \end{align*}
\end{proof}

Because of this observation, we no longer need to treat the case
$0\le\beta_x<\beta_a$ separately (and similarly for $y$ and $z$). We
now introduce notation making use of this. Let:
\begin{align*}
  I^x(\beta)&=q_x^{\beta_x+H_x/\hbar}(-1)^{3\beta_x}\frac{\left(\frac{\Gamma(1-3H_x/\hbar-1)}{\Gamma(-3H_x/\hbar-3\beta_x)}\right)}{\left(\frac{\Gamma(1+H_x/\hbar+\beta_x-\beta_a)}{\Gamma(H_x/\hbar-\langle{\beta_a}\rangle+1)}\right)\left(\frac{\Gamma(1+H_x/\hbar+\beta_x)}{\Gamma(H_x/\hbar+1)}\right)^2}\\
&\\
  I_x(\beta)&=q_x^{\beta_x}(-1)^{3\beta_x}\frac{\left(\frac{\Gamma(\langle\beta_x-\beta_a\rangle)}{\Gamma(\beta_x-\beta_a+1)}\right)\left(\frac{\Gamma(\langle\beta_x\rangle)}{\Gamma(\beta_x+1)}\right)^2}{\left(\frac{\Gamma(-3\beta_x)}{\Gamma(1)}\right)}.
\end{align*}
We similarly define $I^y(\beta),I_y(\beta),I^z(\beta),I_z(\beta).$
Finally we let
\begin{align*}
  I^a(\beta)=q_a^{\beta_a}\left(\frac{1}{\Gamma(1+3\beta_a)}\right).
\end{align*}
It is now straightforward to check (if one is very careful with
indices of products) that:
\begin{align*}
  I^{xyza}(t,q,\hbar)&=\sum_{\substack{\beta_x,\beta_y,\beta_z\in\Z_{\ge0}\\\beta_a\in\frac{1}{3}\Z_{\ge0}}}I^x(\beta)I^y(\beta)I^z(\beta)I^a(\beta)1_{\langle\beta\rangle}\\
  I^{xya}_z(t,q,\hbar)&=\sum_{\substack{\beta_x,\beta_y\in\Z_{\ge0}\\\beta_a\in\frac{1}{3}\Z_{\ge0}\\\beta_z\in\frac{1}{3}\Z_{<0}\setminus\Z_{<0}}}I^x(\beta)I^y(\beta)I_z(\beta)I^a(\beta)1_{\langle\beta\rangle}
\end{align*}
\begin{align*}
  I^{xa}_{yz}(t,q,\hbar)&=\sum_{\substack{\beta_x\in\Z_{\ge0}\\\beta_a\in\frac{1}{3}\Z_{\ge0}\\\beta_y,\beta_z\in\frac{1}{3}\Z_{<0}\setminus\Z_{<0}}}I^x(\beta)I_y(\beta)I_z(\beta)I^a(\beta)1_{\langle\beta\rangle}\\
  I_{xyz}^{a}(t,q,\hbar)&=\sum_{\substack{\beta_a\in\frac{1}{3}\Z_{\ge0}\\\beta_x,\beta_y,\beta_z\in\frac{1}{3}\Z_{<0}\setminus\Z_{<0}}}I_x(\beta)I_y(\beta)I_z(\beta)I^a(\beta)1_{\langle\beta\rangle}
\end{align*}
Essentially all we have done is to rewrite these functions formally
using the identity
$$\prod_{0\le\nu\le
  k}(\alpha-\nu)=\frac{\Gamma(1+\alpha)}{\Gamma(\alpha-k)},$$
and using Lemma \ref{lem:BetaAGreaterThanBetaX} in exceptional cases.

\begin{lem}
  Fix $\beta_y,$ $\beta_z,$ and $\beta_a$. Then
  \begin{align*}
    \sum_{\beta_x\in\Z_{\ge0}}I^x(\beta)1_{\langle\beta\rangle}
  \end{align*}
  analytically continues to
  \begin{align*}
    \sum_{\substack{\beta_x=\frac{1}{3}\Z_{<0}\\\beta_x\not\in\Z\\\beta_x-\beta_a\not\in\Z}}\frac{-2\pi i}{3(e^{2\pi
    i(\beta_x-H_x/\hbar)}-1)}\frac{\Gamma(H_x/\hbar-\langle{\beta_a}\rangle+1)\Gamma(H_x/\hbar+1)^2}{\Gamma(\langle\beta_x-\beta_a\rangle)\Gamma(\langle\beta_x\rangle)^2\Gamma(3H_x/\hbar+1)}I_x(\beta)1_{\langle\beta\rangle}
  \end{align*}
\end{lem}
\begin{proof}
  We have
  \begin{align*}
    \sum_{\beta_x\in\Z_{\ge0}}I^x(\beta)1_{\langle\beta\rangle}&=\sum_{\beta_x\in\Z_{\ge0}}q_x^{\beta_x+H_x/\hbar}(-1)^{3\beta_x}\frac{\left(\frac{\Gamma(1-3H_x/\hbar-1)}{\Gamma(-3H_x/\hbar-3\beta_x)}\right)}{\left(\frac{\Gamma(1+H_x/\hbar+\beta_x-\beta_a)}{\Gamma(H_x/\hbar-\langle{\beta_a}\rangle+1)}\right)\left(\frac{\Gamma(1+H_x/\hbar+\beta_x)}{\Gamma(H_x/\hbar+1)}\right)^2}1_{\langle\beta\rangle}
  \end{align*}
  Using that fact
  that
  $$\frac{\Gamma(1+\alpha)}{\Gamma(\alpha-k)}=(-1)^{k+1}\frac{\Gamma(1-\alpha+k)}{\Gamma(-\alpha)},$$
  which can be easily deduced from the function equation, we can write
  \begin{align*}
    (-1)^{3\beta_x}\left(\frac{\Gamma(1-3H_x/\hbar-1)}{\Gamma(-3H_x/\hbar-3\beta_x)}\right)=\frac{\Gamma(1+3H_x/\hbar+3\beta_x)}{\Gamma(3H_x/\hbar+1)}
  \end{align*}
  Now we have
\begin{align*}
    \sum_{\beta_x\in\Z_{\ge0}}I^x(\beta)1_{\langle\beta\rangle}&=\sum_{\beta_x\in\Z_{\ge0}}q_x^{\beta_x+H_x/\hbar}\frac{\Gamma(1+3H_x/\hbar+3\beta_x)}{\Gamma(1+H_x/\hbar+\beta_x-\beta_a)\Gamma(1+H_x/\hbar+\beta_x)^2}\Phi(\langle{\beta_a}\rangle)1_{\langle\beta\rangle},
  \end{align*}
  where
  \begin{align*}
    \Phi(\langle{\beta_a}\rangle)=\frac{\Gamma(H_x/\hbar-\langle{\beta_a}\rangle+1)\Gamma(H_x/\hbar+1)^2}{\Gamma(3H_x/\hbar+1)}.
  \end{align*}
  We rewrite this last expression using residues:
  \begin{align}\label{eqn:ResiduesAtPositiveIntegers}
    \sum_{\beta_x\ge0}2\pi i\Res_{s=\beta_x}\left(\frac{q^{s+H_x/\hbar}}{e^{2\pi
    is}-1}\frac{\Gamma(1+3H_x/\hbar+3s)}{\Gamma(1+H_x/\hbar+(s-\beta_a))\Gamma(1+H_x/\hbar+s)^2}\right)\Phi(\langle{\beta_a}\rangle)1_{\langle\beta\rangle}.
  \end{align}
  The expression in the large parentheses may be thought of as having
  simple poles at $s\in\Z$ and at $s+H_x/\hbar\in\frac{1}{3}\Z_{<0}.$
  (These loci are unions copies of $\C$ inside $\C^2$. The sum of
  residues can therefore be written as the integral
  \begin{align*}
    \left(\int_{-i\infty}^{i\infty}\frac{q^{s+H_x/\hbar}}{e^{2\pi
    is}-1}\frac{\Gamma(1+3H_x/\hbar+3s)}{\Gamma(1+H_x/\hbar+(s-\beta_a))\Gamma(1+H_x/\hbar+s)^2}ds\right)\Phi(\langle{\beta_a}\rangle)1_{\langle\beta\rangle},
  \end{align*}
  along a contour in $\C^2$ such that the poles $s\in\Z_{<0}$ and
  $s+H_x/\hbar\in\frac{1}{3}\Z_{<0}$ are on one side, and the poles
  $s\in\Z_{\ge0}$ are on the other side. For simplicity, the contour
  may be chosen inside a slice $\C\times\{H'\},$ i.e. we may work with
  a contour integral in $\C$.

  Integrals of this form have been well-studied for a long time, see
  page 49 of \cite{Bateman1953}. From there we see that the integral
  converges to \eqref{eqn:ResiduesAtPositiveIntegers} when
  $q\not\in\R_{<0}$ and $\abs{q}<3^3$. When $q\not\in\R_{<0}$ and
  $\abs{q}>3^3,$ the integral converges to the sum over the remaining
  poles\footnote{Lemma 3.3 of \cite{Horja1999} instead puts $3^{-3}$ as
    the boundary, and \cite{ChiodoRuan2010} agrees.}.

  First we consider the poles $s\in\Z_{<0}$. Consider the expression
  \begin{align}\label{eqn:PolesAtIntegers}
    \frac{\Gamma(1+3H_x/\hbar+3s)}{\Gamma(1+H_x/\hbar+s-\beta_a)\Gamma(1+H_x/\hbar+s)^2}
  \end{align}
  as a function of $H_x/\hbar,$ treating $H_x/\hbar$ as a (small) complex
  number, and $s$ as a fixed negative integer. At $H_x/\hbar=0,$ the expression
  \eqref{eqn:PolesAtIntegers}
  \begin{itemize}
  \item has a zero of order 1 if $\langle\beta_a\rangle\ne0$ (a pole
    from the numerator and two zeroes from the denominator),
  \item has a zero of order 1 if $\beta_a\ge s$, and
  \item has a zero of order 2 if
    $\langle\beta_a\rangle=0$ and $\beta_a<s$.
  \end{itemize}
  In the first and second cases, $H_x$ restricts to zero on $F_\beta$,
  see Lemma \ref{lem:ImageOfEv1}. In the last case, we know that
  $H_x^2=0$, since $H_x$ is the hyperplane class on the 1-dimensional
  space $E$. Together, these say that the residues at the
  poles $s\in\Z_{<0}$ vanish when multiplied by
  $1_{\langle\beta\rangle}.$

  Thus the analytic continuation is the sum
  \begin{align}
    &\sum_{\tilde{\beta_x}\in\frac{H_x}{\hbar}+\frac{1}{3}\Z_{<0}}2\pi i\Res_{s=\tilde{\beta_x}}\left(\frac{q^{s+H_x/\hbar}}{e^{2\pi
          is}-1}\frac{\Gamma(1+3H_x/\hbar+3s)}{\Gamma(1+H_x/\hbar+(s-\beta_a))\Gamma(1+H_x/\hbar+s)^2}\right)\Phi(\langle{\beta_a}\rangle)1_{\langle\beta^{\old}\rangle}\nonumber
   \\
    &=\sum_{\beta_x\in\frac{1}{3}\Z_{<0}}2\pi i\Res_{s=\beta_x}\left(\frac{q^{s}}{e^{2\pi
          i(s-H_x/\hbar)}-1}\frac{\Gamma(1+3s)}{\Gamma(1+(s-\beta_a))\Gamma(1+s)^2}\right)\Phi(\langle{\beta_a}\rangle)1_{\langle\beta^{\old}\rangle}.\label{eqn:ResiduesAtH}
  \end{align}
  Here $e^{2\pi i(s-H_x/\hbar)}$ should be interpreted as in Section
  \ref{sec:ExtendedGamma}, via its expansion at $H_x/\hbar=0.$
  \begin{rem}
    We write $\beta^{\old}$ rather than $\beta$ to emphasize that it
    has not changed and in particular is independent of
    $\langle\beta_x\rangle$ (as it has been all along --- $\beta_x$
    has been an integer). Later in this section we will use
    $1_{\langle{\beta}\rangle}$ to refer to an element of $\H(\theta')$,
    where $\theta'$ is obtained by changing $x$ from a superscript
    variable to a subscript variable.
  \end{rem}
  What remains is to calculate the residues. When $\beta_x\in\Z$ or
  $\beta_x-\beta_a\in\Z,$ the residue in \eqref{eqn:ResiduesAtH}
  vanishes because the simple pole in $\Gamma(1+3s)$ is canceled by
  the poles in $\Gamma(1+(s-\beta_a)),$ or $\Gamma(1+s)$. The residue
  of $\Gamma(1+3s)$ at $\beta_x\in\frac{1}{3}\Z_{<0}$ is
  \begin{align*}
    \frac{(-1)^{3\beta_x+1}}{3\Gamma(-3\beta_x)}.
  \end{align*}
  Thus we rewrite \eqref{eqn:ResiduesAtH} as
  \begin{align*}
    \sum_{\substack{\beta_x\in\frac{1}{3}\Z_{<0}\\\beta_x\not\in\Z\\\beta_x-\beta_a\not\in\Z}}2\pi i\left(\frac{q^{\beta_x}}{e^{2\pi
          i(\beta_x-H_x/\hbar)}-1}\frac{\left(\frac{(-1)^{3\beta_x+1}}{3\Gamma(-3\beta_x)}\right)}{\Gamma(1+(\beta_x-\beta_a))\Gamma(1+\beta_x)^2}\right)\Phi(\langle{\beta_a}\rangle)1_{\langle\beta^{\old}\rangle}.
  \end{align*}
  Rearranging slightly gives
  \begin{align}\label{eqn:RewriteWithIX}
    \sum_{\substack{\beta_x\in\frac{1}{3}\Z_{<0}\\\beta_x\not\in\Z\\\beta_x-\beta_a\not\in\Z}}&\frac{2\pi
    iq^{\beta_x}(-1)^{3\beta_x+1}}{3(e^{2\pi
    i(\beta_x-H_x/\hbar)}-1)}\frac{\frac{\Gamma(\langle\beta_x-\beta_a\rangle)}{\Gamma(\beta_x-\beta_a+1)}\left(\frac{\Gamma(\langle\beta_x\rangle)}{\Gamma(\beta_x+1)}\right)^2}{\left(\frac{\Gamma(-3\beta_x)}{\Gamma(1)}\right)\Gamma(\langle\beta_x-\beta_a\rangle)\Gamma(\langle\beta_x\rangle)^2}\Phi(\langle{\beta_a}\rangle)1_{\langle\beta^{\old}\rangle}\\
 &=\sum_{\substack{\beta_x\in\frac{1}{3}\Z_{<0}\\\beta_x\not\in\Z\\\beta_x-\beta_a\not\in\Z}}\frac{-2\pi i}{3(e^{2\pi
    i(\beta_x-H_x/\hbar)}-1)}\frac{\Phi(\langle{\beta_a}\rangle)}{\Gamma(\langle\beta_x-\beta_a\rangle)\Gamma(\langle\beta_x\rangle)^2}I_x(\beta)1_{\langle\beta^{\old}\rangle}.\nonumber
  \end{align}
\end{proof}
We now use Section \ref{sec:ExtendedGamma} to expand the factors. That
is, again using $H_x^2=0$, we have
\begin{align*}
  \frac{1}{e^{2\pi i(\beta_x-H_x/\hbar)}-1}=\frac{1}{e^{2\pi
  i\beta_x}-1}+\frac{2\pi ie^{2\pi
  i\beta_x}}{(e^{2\pi i\beta_x}-1)^2}H_x/\hbar
\end{align*}
and
\begin{align*}
  \Phi(\langle\beta_a\rangle)=\Gamma(1-\langle\beta_a\rangle)+\Gamma(1-\langle\beta_a\rangle)(\mathfrak{h}_{-\langle\beta_a\rangle})H_x/\hbar,
\end{align*}
where $\mathfrak{h}_0=0$,
$\mathfrak{h}_{-1/3}=\frac{\pi}{2\sqrt{3}}-\frac{3\log3}{2},$ and
$\mathfrak{h}_{-1/3}=-\frac{\pi}{2\sqrt{3}}-\frac{3\log3}{2}.$
Finally we write \eqref{eqn:RewriteWithIX} as:
\begin{align*}
  \sum_{\substack{\beta_x\in\frac{1}{3}\Z_{<0}\\\beta_x\not\in\Z\\\beta_x-\beta_a\not\in\Z}}&\frac{-2\pi
                                                                                              i}{3\Gamma(\langle\beta_x-\beta_a\rangle)\Gamma(\langle\beta_x\rangle)^2}\cdot\\
                                                                                            &\cdot\left(\frac{\Gamma(1-\langle\beta_a\rangle)}{e^{2\pi
                                                                                              i\beta_x}-1}+\left(\frac{2\pi ie^{2\pi
                                                                                              i\beta_x}\Gamma(1-\langle\beta_a\rangle)}{(e^{2\pi
                                                                                              i\beta_x}-1)^2}+\frac{\Gamma(1-\langle\beta_a\rangle)(\mathfrak{h}_{-\langle\beta_a\rangle})}{e^{2\pi
    i\beta_x}-1}\right)\frac{H_x}{\hbar}\right)I_x(\beta)1_{\langle\beta^{\old}\rangle}.
\end{align*}
What remains is to make the identification in Section
\ref{sec:CompactTypeStateSpace}. Namely, write
$1_{\langle\beta^{\old}\rangle}=(1\otimes\alpha_y\otimes\gamma_z)_g\in\H(\theta),$
and define
\begin{align*}
  1_{\langle\beta^{\old}\rangle}\mapsto1_{\zeta,\langle\beta^{\old}\rangle}:&=(1_\zeta\otimes\alpha_y\otimes\gamma_z)_g\in\H(\theta')\\
  H_x1_{\langle\beta^{\old}\rangle}\mapsto1_{\zeta^2,\langle\beta^{\old}\rangle}:&=(1_{\zeta^2}\otimes\alpha_y\otimes\gamma_z)_g\in\H(\theta').
\end{align*}
Thus the coefficients
\begin{align*}
  \frac{-2\pi
  i}{3\Gamma(\langle\beta_x-\beta_a\rangle)\Gamma(\langle\beta_x\rangle)^2}\cdot\frac{\Gamma(1-\langle\beta_a\rangle)}{e^{2\pi
  i\beta_x}-1}
\end{align*}
and 
\begin{align*}
  \frac{-2\pi i}{3\Gamma(\langle\beta_x-\beta_a\rangle)\Gamma(\langle\beta_x\rangle)^2}\cdot\left(\frac{2\pi ie^{2\pi i\beta_x}\Gamma(1-\langle\beta_a\rangle)}{(e^{2\pi i\beta_x}-1)^2}+\frac{\Gamma(1-\langle\beta_a\rangle)(\mathfrak{h}_{-\langle\beta_a\rangle})}{e^{2\pi i\beta_x}-1}\right)
\end{align*}
define an isomorphism $\H(\theta)\to\H(\theta')$. We have proved:
\begin{thm}[LG/CY correspondence]\label{thm:LGCY}
  This isomorphism identifies the analytically continued $I$-function
  $I^{\theta,\Giv}(q,\hbar)$ with $I^{\theta',\Giv}(q,\hbar).$
\end{thm}

\section{Notation Table}
\label{sec:NotationTable}
\begin{longtable}{p{5cm}|p{12cm}}
  Notation&Description\\\hline
  $1_\theta$&Generator of $\H^0(\theta)$\\
  $1_g$&The fundamental class of $g$ twisted sectors of $X$\\
  $1_{\langle\beta\rangle}$&Fundamental class of certain twisted
  sectors, Section \ref{sec:CalculatingI}\\
  $a$&Coordinates on $V$\\
  $b_i$&Marked point on $C$\\
  $BG$&Stack of principal $G$-bundles, i.e. $BG=[\Spec\C/G]$\\
  $B_0,B_\infty$&Sets of marked points of G-graph quasimap over $0,\infty\in\P^1$\\
  $\beta$&Shorthand for $(\beta_x,\beta_y,\beta_z,\beta_a)$\\
  $\beta(P)$&Degree of the basepoint of $\sigma$ at $P$, Section \ref{sec:ModuliSpaces}\\
  $\beta^0,\beta^\infty$&Degrees of LG-graph quasimap supported over $0,\infty\in\P^1$\\
  $\beta_0(\theta,m)$&Extremal degree of $m$-marked LG-quasimaps to $X(\theta)$\\
  $\beta_\rho$&Degree of $\mathcal{L}_\rho$\\
  $\beta_x,\beta_y,\beta_z,\beta_a,\beta_R$&Degrees of $\mathcal{L}_x,\mathcal{L}_y,\mathcal{L}_z,\mathcal{L}_a,\mathcal{L}_R$\\
  $\bullet,\check\bullet$&Points of $\hat{C}$ mapping to $0,\infty\in\P^1$\\
  $C$&$m$-marked genus zero twisted curve\\
  $C_0,C_\infty$&Components of a graph quasimap over $0,\infty\in\P^1$\\
  $U_\beta'$&Universal curve over $F_\beta'$\\
  $\C^*_R$&Group acting on $V$\\
  $\C_\rho$&The representation associated to a character $\rho$\\
  $\Crit(W)$&Critical locus of $W$ in $V$\\
  $\hat{C}$&Parametrized component of a graph quasimap\\
  $d_P$&The order of a point $P$ on $C$\\
  $D_\rho$&Toric divisor on $X(\theta)$\\
  $\ev_i$&Evaluation maps to $Z(\theta)$ or $X(\theta)$\\
  $E$&Elliptic curve in $\P^2$\\
  $\mathcal{E}$&$\mathcal{P}\times_{(G\times\C^*_R)}V$\\
  $\epsilon$&Stability parameter, in $\Q_{>0}$\\
  $F_\beta,F_\beta'$&Special components of $\C^*$-fixed LG-graph quasimaps to $Z(\theta),X(\theta)$\\
  $F_{B_0,\beta^0}^{B_\infty,\beta^\infty}$&$\C^*$-fixed LG-graph quasimaps inducing partitions $B_0\sqcup B_\infty$ and $\beta^0+\beta^\infty$\\
  $G$&$(\C^*)^4$\\
  $\{\gamma_j\},\{\gamma^j\}$&Basis and dual basis for $\H(\theta)$\\
  $H$&Class $[L_0]=[L']\in H^2(\P^2/\mu_3)$\\
  $H_x,H_y,H_z$&Divisor classes on $[E^3/\mu_3]$\\
  $\H(\theta)$&Compact type state space associated to $\theta$\\
  $\hbar$&Generator of $\C^*$-equivariant cohomology ring of a point\\
  $I\mathcal{X}$&(Nonrigidified) inertia stack of $\mathcal{X}$\\
  $\bar{I}\mathcal{X}$&Rigidified inertia stack of $\mathcal{X}$\\
  $IX^{\nar},\bar{I}X^{\nar}$&Narrow components of $IX,\bar{I}X$\\
  $I^\theta(q,\hbar)$&$I$-function\\
  $\iota$&Embedding $Z\into X$ or $Z(\theta)\into X(\theta)$\\
  $\kappa$&Isomorphism $\mathcal{L}_R\to\omega_{C,\log}$\\
  $\ell^\sigma(P)$&Length of $\sigma$ at $P$\\
  $L_0,L'$&Certain lines in $[\P^2/\mu_3]$\\
  $L_\rho$&Line bundle on $X(\theta)$ (resp. $[X(\theta)/\C^*_R]$) corresponding to character $\rho$ of $G$ (resp. $G\times\C^*_R$)\\
  $\LGQ_{0,m}^\epsilon(X(\theta),\beta),$  $\LGQ_{0,m}^\epsilon(Z(\theta),\beta)$&Stack of $\epsilon$-stable genus zero $m$-marked LG-quasimaps to $X(\theta)$ (resp, $Z(\theta)$) of degree $\beta$\\
  $\LGQ_{0,m}^\epsilon(X(\theta),\beta),$ $\LGQ_{0,m}^\epsilon(Z(\theta),\beta)$&Stack of LG-graph quasimaps to $X(\theta)$ (resp. $Z(\theta)$)\\
  $\mathcal{L}_\rho$&$u^*(L_\rho)$\\
  $\mathcal{L}_x,\mathcal{L}_y,\mathcal{L}_z,\mathcal{L}_a,\mathcal{L}_R$&Line bundles on $C$ built from $\mathcal{P}$\\
  $\mult_P(\mathcal{L})$&The multiplicity (monodromy) of $\mathcal{L}$ at $P$\\
  $m$&Number of marked points on $C$\\
  $\mu,\nu$&$T$-fixed points of $\bar{I}X(\theta)$\\
  $\mu_d$&The group of $d$th roots of unity in $\C^*$\\
  $N^{\vir}$&Virtual normal bundle\\
  $p_x,p_y,p_z$&Coordinates on $V$\\
  $P$&Class $[P_0]=[P']\in H^4(\P^2/\mu_3)$\\
  $P_0,P'$&Certain points in $[\P^2/\mu_3]$\\
  $\P_{3,1}$&$\P^1$ with an order 3 orbifold point at $[\infty]$\\
  $\mathcal{P}$&Principal $G\times\C^*_R$-bundle\\
  $\pi$&Map from universal curve to moduli stack\\
  $q,q_x,q_y,q_z,q_a$&Formal parameters keeping track of $\beta,\beta_x,\beta_y,\beta_z,\beta_a$\\
  $\mathbf{R}$&$\{\rho_{x_0},\ldots,\rho_{p_z}\}$\\
  $\rho_{x_0},\ldots,\rho_{p_z}$&Characters of $G\times\C^*_R$, which define $V$ as a direct sum\\
  $s$&Complexification of $\beta_x$\\
  $s,t$&Coordinates on $\P_{3,1}$\\
  $\sigma$&Section of $\mathcal{E}$\\
  $\sigma_{x_0},\ldots,\sigma_{p_z}$&Components of $\sigma$, sections of $\mathcal{L}_\rho$ for $\rho\in\mathbf{R}$\\
  $(t_x,t_y,t_z,t_a)$&Element of $G$\\
  $\hat{t_x},\hat{t_y},\hat{t_z},\hat{t_a},\hat{t_R}$&Characters $(t_x,t_y,t_z,t_a,t_R)\mapsto t_x$, etc., of $G\times\C^*_R$\\
  $t$&Coordinates of cohomology ring\\
  $t_R$&Element of $\C^*_R$\\
  $T$&$(\C^*)^{13}$\\
  $\tau$&Parametrization map $C\to\P^1$\\
  $\theta,\theta^{xyza},\ldots,\theta_{xyz}^a$&GIT characters of $G$\\
  $\theta,\theta^{xyza},\ldots,\theta_{xyz}^a$&Lifts of $\theta,\theta^{xyza},\ldots,\theta_{xyz}^a$ to $G\times\C^*_R$\\
  $\Theta$&$\{\theta^{xyza},\theta_z^{xya},\theta_{yz}^{xa},\theta_{xyz}^a\}$\\
  $u$&Map $C\to[V/(G\times\C^*_R)]$\\
  $V$&$\C^{13}$\\
  $V^{ss}(\theta)$&$\theta$-semistable locus of $V$\\
  $V^{uns}(\theta)$&$\theta$-unstable locus of $V$\\
  $[V\sslash_\theta G]$&The GIT stack quotient $[V^{ss}(\theta)/G]$\\
  $w(\mu,\nu)$&Tangent weight at $\mu$ along curve from $\mu$ to $\nu$\\
  $W$&Function $p_x(ax_0^3+x_1^3+x_2^3)+p_y(ay_0^3+y_1^3+y_2^3)+p_z(az_0^3+z_1^3+z_2^3)$\\
  $\omega_{C,\log}$&Log canonical bundle of $C$\\
  $x_0,x_1,x_2$&Coordinates on $V$\\
  $X$&$[V/G]$\\
  $X(\theta)$&$[V\sslash_{\theta}G]$\\
  $X(\theta)_{\langle\beta\rangle}$&Component of $\bar{I}X(\theta)$,
  Section \ref{sec:CalculatingI}\\
  $X_R(\theta)$&Points $P\in[X(\theta)/\C^*_R]$ with $\C^*_R\subseteq G_P$\\
  $y_0,y_1,y_2$&Coordinates on $V$\\
  $z_0,z_1,z_2$&Coordinates on $V$\\
  $Z$&$[\Crit(W)/G]$\\
  $Z(\theta)$&$[\Crit(W)\cap V^{ss}(\theta)/G]$\\
  $\zeta$&$e^{2\pi i/3}\in\mu_3$\\
\end{longtable}

\bibliographystyle{plain}
\bibliography{LGCYRefs}

\begin{thebibliography}{10}

\bibitem{AbramovichGraberVistoli2008}
Dan Abramovich, Tom Graber, and Angelo Vistoli.
\newblock Gromov-{W}itten theory of {D}eligne-{M}umford stacks.
\newblock {\em American Journal of Mathematics}, 130(5):1337--1398, 2008.

\bibitem{AbramovichVistoli2002}
Dan Abramovich and Angelo Vistoli.
\newblock Compactifying the space of stable maps.
\newblock {\em Journal of the American Mathematical Society}, 15:27--75, 2002.

\bibitem{Acosta2014}
Pedro Acosta.
\newblock Asymptotic expansion and the {LG}/({F}ano, general type)
  correspondence.
\newblock {\em ArXiv e-prints}, November 2014.
\newblock \href{http://arxiv.org/abs/1411.4162}{\texttt{arXiv:1411.4162}}.

\bibitem{Bateman1953}
Harry Bateman and staff of~the Bateman Manuscript~Project.
\newblock {\em Higher Transcendental Functions}.
\newblock McGraw-Hill Book Company, Inc., 1953.

\bibitem{BehrendFantechi1997}
Kai Behrend and Barbara Fantechi.
\newblock The intrinsic normal cone.
\newblock {\em Inventiones mathematicae}, 128(1):45--88, 3 1997.

\bibitem{ChangLi2011}
Huai-Liang Chang and Jun Li.
\newblock {G}romov-{W}itten invariants of stable maps with fields.
\newblock {\em International Mathematics Research Notices},
  2012(18):4163--4217, 2012.

\bibitem{CheongCiocanFontanineKim2015}
Daewoong Cheong, Ionu{\c{t}} Ciocan-Fontanine, and Bumsig Kim.
\newblock Orbifold quasimap theory.
\newblock {\em Mathematische Annalen}, 363(3):777--816, 2015.

\bibitem{ChiodoIritaniRuan2013}
Alessandro Chiodo, Hiroshi Iritani, and Yongbin Ruan.
\newblock Landau-{G}inzburg/{C}alabi-{Y}au correspondence, global mirror
  symmetry and {O}rlov equivalence.
\newblock {\em Publications math{\'e}matiques de l'IH{\'E}S}, 119(1):127--216,
  2013.

\bibitem{ChiodoNagel2015}
Alessandro Chiodo and Jan Nagel.
\newblock The hybrid {L}andau-{G}inzburg models of {C}alabi-{Y}au complete
  intersections.
\newblock {\em ArXiv e-prints}, June 2015.
\newblock \href{http://arxiv.org/abs/1506.02989}{\texttt{arXiv:1506.02989}}.

\bibitem{ChiodoRuan2010}
Alessandro Chiodo and Yongbin Ruan.
\newblock {L}andau--{G}inzburg/{C}alabi--{Y}au correspondence for quintic
  three-folds via symplectic transformations.
\newblock {\em Inventiones mathematicae}, 182(1):117--165, 2010.

\bibitem{ChiodoZvonkine2009}
Alessandro Chiodo and Dmitri Zvonkine.
\newblock Twisted $r$-spin potential and {G}ivental's quantization.
\newblock {\em Advances in Theoretical and Mathematical Physics},
  13(5):1335--1369, 10 2009.

\bibitem{CiocanFontanineKim2013}
Ionu{\c{t}} Ciocan-Fontanine and Bumsig Kim.
\newblock Higher genus quasimap wall-crossing for semi-positive targets.
\newblock {\em Journal of the European Mathematical Society}.
\newblock To appear.

\bibitem{CiocanFontanineKim2014}
Ionu{\c{t}} Ciocan-Fontanine and Bumsig Kim.
\newblock Wall-crossing in genus zero quasimap theory and mirror maps.
\newblock {\em Algebraic Geometry}, 1(4):400--448, 2014.

\bibitem{Clader2013}
Emily Clader.
\newblock Landau-{G}inzburg/{C}alabi-{Y}au correspondence for the complete
  intersections {$X_{3,3}$} and {$X_{2,2,2,2}$}.
\newblock {\em ArXiv e-prints}, January 2013.
\newblock \href{http://arxiv.org/abs/1301.5530}{\texttt{arXiv:1301.5530}}.

\bibitem{CladerRoss2015}
Emily Clader and Dustin Ross.
\newblock Sigma models and phase transitions for complete intersections.
\newblock {\em ArXiv e-prints}, November 2015.
\newblock \href{http://arxiv.org/abs/1511.02027}{\texttt{arXiv:1511.02027}}.

\bibitem{CoatesCortiIritaniTseng2015}
Tom Coates, Alessio Corti, Hiroshi Iritani, and Hsian-Hua Tseng.
\newblock A mirror theorem for toric stacks.
\newblock {\em Compositio Mathematica}, 151:1878--1912, 10 2015.

\bibitem{CoatesIritaniJiang2014}
Tom Coates, Hiroshi Iritani, and Yunfeng Jiang.
\newblock The crepant transformation conjecture for toric complete
  intersections.
\newblock {\em ArXiv e-prints}, September 2014.
\newblock \href{http://arxiv.org/abs/1410.0024}{\texttt{arXiv:1410.0024}}.

\bibitem{DixonHarveyVafaWitten1986}
Lance Dixon, Jeffrey Harvey, Cumrun Vafa, and Edward Witten.
\newblock Strings on orbifolds ({II}).
\newblock {\em Nuclear Physics B}, 274:285--314, 9 1986.

\bibitem{DolgachevHu1998}
Igor Dolgachev and Yi~Hu.
\newblock Variation of geometric invariant theory quotients.
\newblock {\em Publications math{\'e}matiques de l'IH{\'E}S}, 87:5--51, 1998.

\bibitem{FanJarvisRuan2013}
Huijun Fan, Tyler Jarvis, and Yongbin Ruan.
\newblock The {W}itten equation, mirror symmetry, and quantum singularity
  theory.
\newblock {\em Annals of Mathematics}, 178(1):1--106, 2013.

\bibitem{FanJarvisRuan2015}
Huijun Fan, Tyler Jarvis, and Yongbin Ruan.
\newblock A mathematical theory of the gauged linear sigma model.
\newblock {\em ArXiv e-prints}, June 2015.
\newblock \href{http://arxiv.org/abs/1506.02109}{\texttt{arXiv:1506.02109}}.

\bibitem{FischerRatzTorradoVaudrevange2013}
Maximilian Fischer, Michael Ratz, Jes{\'u}s Torrado, and Patrick~K.S.
  Vaudrevange.
\newblock Classification of symmetric toroidal orbifolds.
\newblock {\em Journal of High Energy Physics}, 2013(1):1--53, 2013.

\bibitem{Givental1998}
Alexander Givental.
\newblock A mirror theorem for toric complete intersections.
\newblock In Masaki Kashiwara, Atsushi Matsuo, Kyoji Saito, and Ikuo Satake,
  editors, {\em Topological Field Theory, Primitive Forms and Related Topics},
  pages 141--175. Birkh{\"a}user Boston, Boston, MA, 1998.

\bibitem{Horja1999}
Richard~Paul Horja.
\newblock Hypergeometric functions and mirror symmetry in toric varieties.
\newblock {\em ArXiv Mathematics e-prints}, December 1999.
\newblock
  \href{http://arxiv.org/abs/math/9912109}{\texttt{arXiv:math/9912109}}.

\bibitem{KiemLi2013}
Young-Hoon Kiem and Jun Li.
\newblock Localizing virtual cycles by cosections.
\newblock {\em Journal of the American Mathematical Society}, 26:1025--1050,
  2013.

\bibitem{LeePriddisShoemaker2014}
Y.-P. Lee, Nathan Priddis, and Mark Shoemaker.
\newblock A proof of the {L}andau-{G}inzburg/{C}alabi-{Y}au correspondence via
  the crepant transformation conjecture.
\newblock {\em ArXiv e-prints}, October 2014.
\newblock \href{http://arxiv.org/abs/1410.5503}{\texttt{arXiv:1410.5503}}.

\bibitem{Olsson2007}
Martin Olsson.
\newblock On (log) twisted curves.
\newblock {\em Compositio Mathematica}, 143(2):476--494, feb 2007.

\bibitem{PriddisShoemaker2014}
Nathan Priddis and Mark Shoemaker.
\newblock A {L}andau-{G}inzburg/{C}alabi-{Y}au correspondence for the mirror
  quintic.
\newblock {\em ArXiv e-prints}, September 2013.
\newblock \href{http://arxiv.org/abs/1309.6262}{\texttt{arXiv:1309.6262}}.

\bibitem{RossRuan2015}
Dusty Ross and Yongbin Ruan.
\newblock Wall-crossing in genus zero {L}andau-{G}inzburg theory.
\newblock {\em Journal f\"{u}r die reine und angewandte Mathematik}.
\newblock To appear.

\bibitem{Ruan2011}
Yongbin Ruan.
\newblock The {W}itten equation and the geometry of the {L}andau-{G}inzburg
  model.
\newblock In Jonathan Block, Jacques Distler, Ron Donagi, and Eric Sharpe,
  editors, {\em String-Math 2011}, pages 209--240. American Mathematical
  Society, Providence, Rhode Island, 2011.

\bibitem{Schaug2015}
Andrew Schaug.
\newblock The {G}romov-{W}itten theory of {B}orcea-{V}oisin orbifolds and its
  analytic continuations.
\newblock {\em ArXiv e-prints}, June 2015.
\newblock \href{http://arxiv.org/abs/1506.07226}{\texttt{arXiv:1506.07226}}.

\bibitem{Thaddeus1996}
Michael Thaddeus.
\newblock Geometric invariant theory and flips.
\newblock {\em Journal of the American Mathematical Society}, 9:691--723, 1996.

\bibitem{Tseng2010}
Hsian-Hua Tseng.
\newblock Orbifold quantum {R}iemann-{R}och, {L}efschetz and {S}erre.
\newblock {\em Geometry and Topology}, 14:1--81, 2010.

\bibitem{VoightZureickBrown2015}
John Voight and David Zureick-Brown.
\newblock The canonical ring of a stacky curve.
\newblock {\em ArXiv e-prints}, January 2015.
\newblock \href{http://arxiv.org/abs/1501.04657}{\texttt{arXiv:1501.04657}}.

\bibitem{Witten1993}
Edward Witten.
\newblock Phases of {$N=2$} theories in two dimensions.
\newblock {\em Nuclear Physics B}, 403:259--222, 1993.

\end{thebibliography}

\end{document}